\documentclass[a4paper,10pt]{article}

\usepackage[T2A]{fontenc}
\usepackage[ngerman]{babel} \babelprovide[hyphenrules=ngerman-x-latest]{ngerman}
\usepackage[backend=biber,style=alphabetic]{biblatex}
\usepackage{csquotes}
\usepackage{amsthm,amssymb,amsmath,amscd,empheq,mathrsfs,mathtools, bm}
\usepackage{tikz-cd} \tikzset{labl/.style={anchor=south, rotate=90, inner sep=.5mm}} \usetikzlibrary{babel} \tikzset{labr/.style={anchor=south, rotate=270, inner sep=.5mm}} \usetikzlibrary{babel}
\usepackage{todonotes,xcolor,graphicx,epigraph,multicol,tabularx,booktabs,changepage,mathtools}
\usepackage[margin=1.0in]{geometry}
\usepackage[shortlabels]{enumitem}
\usepackage{titlesec}\titleformat*{\section}{\LARGE\bfseries}
\usepackage{hyperref}
\hypersetup{
  colorlinks   = true,    
  urlcolor     = blue,    
  linkcolor    = blue,    
  citecolor    = red      
}

\numberwithin{equation}{section}

\newtheorem{lemma}[equation]{Lemma}

\newtheorem{satz}[equation]{Satz}
\newtheorem{korollar}[equation]{Korollar}

\theoremstyle{definition}\newtheorem{definition}[equation]{Definition}
\theoremstyle{definition}
\theoremstyle{definition}\newtheorem{konstruktion}[equation]{Konstruktion}
\theoremstyle{definition}\newtheorem{bemerkung}[equation]{Bemerkung}

\newtheorem*{notation}{Notationen und Terminologie}

\DeclareMathOperator{\Ima}{Im}
\DeclareMathOperator{\colim}{colim}

\DeclareMathOperator{\Spec}{Spec}

\DeclareMathOperator{\Spa}{Spa}
\DeclareMathOperator{\Ind}{Ind}

\DeclareMathOperator{\cofib}{cofib}

\author{Grigory Andreychev}
\title{$K$-Theorie adischer Räume}
\date{\vspace{-5ex}} 

\newenvironment{titlefront}{
	\clearpage
	\newgeometry{hmargin={3.25cm,3.25cm},vmargin={4cm,4.5cm}}
	\thispagestyle{empty}
	\parindent=0pt
	\parskip=0pt
	\centering
}{\clearpage\aftergroup\restoregeometry}
\newenvironment{titleback}{
	\clearpage
	\newgeometry{hmargin={4.5cm,4.5cm},vmargin={4cm,4.5cm}} 
	\parindent=0pt
	\parskip=0pt
}{\cleardoublepage\aftergroup\restoregeometry}

\usepackage{hyperref}
\usepackage[hyphenbreaks]{breakurl}

\begin{document}
\begin{titlefront}
 \vspace{2cm}
{\LARGE\textbf{$K$-Theorie adischer Räume}}

\vspace{5cm}

Dissertation
  		
\vspace{3mm}
zur
\vspace{3mm}
  		
Erlangung des Doktorgrades (Dr. rer. nat)
\vspace{3mm}
  		
der
\vspace{3mm}
  		
Mathematisch-Naturwissenschaftlichen Fakultät
\vspace{3mm}
  		
der
\vspace{3mm}
  		
Rheinischen Friedrich-Wilhelms-Universität Bonn

\vspace{4cm}
  		
vorgelegt von
\vspace{3mm}

{\textbf{Grigory Andreychev}}
\vspace{3mm}

aus

\vspace{3mm}

Norilsk, Russland

\vfill

Bonn, 2023
\thispagestyle{empty}
\end{titlefront}

\begin{titleback}
\thispagestyle{empty}
 	
\begin{center}
\noindent  Angefertigt mit Genehmigung der Mathematisch-Naturwissenschaftlichen Fakult\"  at  der  Rheinischen Friedrich-Wilhelms-Universit\" at Bonn
\end{center}
\vfill
 	
\noindent 1. Gutachter/Betreuer: Prof. Dr. Peter Scholze
 	
\noindent 2. Gutachter: Prof. Dr. Georg Tamme
 	
\vspace{1cm}
 		
\noindent Tag der Promotion: 30. August 2023 
 	
\noindent Erscheinungsjahr: 2023
 	
\end{titleback}

\begin{center}
\thispagestyle{empty}
\vspace*{7cm}
\Large\rightline{Моим родителям}
\vspace*{\fill}
\end{center}
\clearpage
\vfill
\newpage
\newpage\null\thispagestyle{empty}\newpage
\newpage
\tableofcontents
\newpage
\section{Einleitung}

Seit ihrer Erfindung nimmt die $K$-Theorie einen besonderen Platz in der Mathematik ein. Informell gesprochen handelt es sich um eine Sammlung von topologischen und algebraischen Methoden zur Analyse einer geeignet definierten „Kategorie der Vektorbündel“ auf einem Raum $X$. Diese Analyse sollte eine Folge von Invarianten liefern, die sich wie eine verallgemeinerte Kohomologietheorie verhält und in einer fundamentalen Beziehung zu anderen algebraischen Invarianten steht. Beispielsweise kann die $K$-Theorie in der algebraischen Topologie über die Atiyah-Hirzebruch-Spektralsequenz mit singulärer Kohomologie in Beziehung gesetzt werden. In der algebraischen Geometrie verbindet die berühmte Quillen-Lichtenbaum-Vermutung, welche von Voevodsky und Rost bewiesen wurde, die $K$-Theorie mit der étalen Kohomologie über eine ähnliche Spektralsequenz. Das Ziel dieser Arbeit ist es, die Grundlagen der $K$-Theorie im Kontext der adischen Geometrie zu legen und insbesondere ein Analogon des Satzes von Grothendieck-Riemann-Roch zu beweisen. Die wichtigste konzeptionelle Neuerung dieser Arbeit ist zweifellos die Definition der $K$-Theorie im adischen Kontext. Sobald der richtige Rahmen gefunden ist, sind die anderen Teile der Arbeit rein formal. Bevor wir den Plan dieser Dissertation vorstellen, möchten wir daher hier kurz auf die Entwicklung der Definition der $K$-Theorie eingehen und die Motivation für unseren Ansatz erläutern. 

Der erste vollständige kohomologische Formalismus dieser Art wurde in der algebraischen Topologie entwickelt und wird heute als \textit{topologische $K$-Theorie} bezeichnet. Für einen parakompakten topologischen Hausdorff-Raum $X$ betrachten wir das Gruppoid\footnote{d.h. eine Kategorie, deren Morphismen alle invertierbar sind} $\mathrm{Vect}_X^{\mathbb{R}/\mathbb{C}}$ der endlichdimensionalen lokal freien reellen oder komplexen Vektorbündel auf $X$, welche wir im Folgenden einfach \textit{Vektorbündel} nennen. Der Ausgangspunkt ist die folgende Definition von Grothendieck:

\begin{definition}
Die Gruppe $K_0(X)$ ist definiert als die Gruppenvervollständigung des Monoids der Isomorphieklassen der Vektorbündel auf $X$, dessen Monoidstruktur durch direkte Summe gegeben ist. 
\end{definition}

Aus dieser Definition wird dann mittels des \textit{Einhängungsisomorphismus} und der \textit{Bott-Periodizität} eine ganzzahlig indizierte Folge von algebraischen Invarianten von $X$ konstruiert, welche \textit{die höheren $K$-Gruppen} genannt werden. In der algebraischen oder analytischen Geometrie stehen solche Methoden jedoch nicht zur Verfügung. Wenn man also die $K$-Theorie in diesen Kontexten definieren will, muss man eine andere, abstraktere Definition finden. Die übliche moderne Herangehensweise an die algebraische $K$-Theorie macht sich die Theorie der $\infty$-Kategorien in wesentlicher Weise zunutze: Statt der Gruppenvervollständigung der Menge der Isomorphieklassen der Vektorbündel betrachtet man die \textit{$\infty$-Gruppenvervollständigung} des Gruppoids der Vektorbündel, welches als über $\mathrm{Top}$ angereicherte Kategorie angesehen wird. Dies liefert ein Spektrum im Sinne der algebraischen Topologie und ermöglicht die folgende Definition der topologischen $K$-Theorie, die für endlichdimensionale Zellkomplexe mit der klassischen übereinstimmt:

\begin{definition}
Die Gruppe $K_{i}(X)$ ist definiert als die $i$-te Homotopiegruppe $\pi_{i}((\mathrm{Vect}_X^{\mathbb{R}/\mathbb{C}})^{\mathrm{grp}})$ der $\infty$-Gruppenvervollständigung des Gruppoids der Vektorbündel auf $X$, dessen Monoidstruktur durch direkte Summe gegeben ist. 
\end{definition}

Einer der Hauptvorteile dieser Definition besteht darin, dass sie direkt in den Kontext der kommutativen Algebra übertragen werden kann. Ersetzt man nämlich die Kategorie der Vektorbündel durch die Kategorie der endlichen projektiven Moduln über einem kommutativen Ring $R$, so erhält man die folgende Definition:

\begin{definition}
Sei $R$ ein kommutativer Ring. Die Gruppe $K_{i}(R)$ ist definiert als die $i$-te Homotopiegruppe $\pi_{i}((\mathrm{Proj}_R)^{\mathrm{grp}})$ der $\infty$-Gruppenvervollständigung des Gruppoids der endlichen projektiven Moduln über $R$, dessen Monoidstruktur durch direkte Summe gegeben ist. 
\end{definition}

Will man jedoch noch weiter gehen und die $K$-Theorie für alle Schemata definieren, so muss man aufgrund einiger technischer Probleme etwas anders vorgehen. Die erste allgemeine Konstruktion geht auf Quillen zurück, der zu diesem Zweck den Begriff der \textit{exakten Kategorie} einführte. Ohne ins Detail zu gehen, sei hier nur erwähnt, dass sein Ansatz dem obigen sehr ähnlich ist: Die $K$-Gruppen sind dabei als die Homotopiegruppen des Spektrums gegeben, das sich durch Anwenden der sogenannten Quillen'schen $Q$-Konstruktion auf die exakte Kategorie der Vektorbündel ergibt. Es stellt sich jedoch heraus, dass diese Definition „falsch“ ist. Das Problem besteht darin, dass die so definierte $K$-Theorie im Allgemeinen schlechte Lokal-Global-Eigenschaften hat, mit anderen Worten erfüllt sie nicht Abstieg bezüglich der Zariski-Topologie. In seiner berühmten Arbeit \cite{TT90} hat Robert Thomason eine andere, heute allgemein akzeptierte Definition vorgeschlagen, die mit Quillens Version in guten Situationen übereinstimmt. Insbesondere stimmt die „korrekte“ $K$-Theorie von Thomason mit der von Quillen im affinen Fall überein. In moderner Sprache lautet die Definition wie folgt:

\begin{definition}
Sei $X$ ein Schema und $\mathrm{Cat}_\infty^\mathrm{perf}$ die $\infty$-Kategorie der kleinen idempotent vollständigen stabilen $\infty$-Kategorien, deren Morphismen exakte Funktoren sind. Mit $\mathrm{Perf}(X)$ bezeichnen wir die derivierte $\infty$-Kategorie der perfekten Komplexe auf $X$ und mit $K:\mathrm{Cat}_{\infty}^{\mathrm{perf}}\xrightarrow{} \mathrm{Sp}$ den $K$-Theorie-Funktor, welcher eine kleine stabile $\infty$-Kategorie auf ein Spektrum abbildet. Dann ist das $K$-Theorie-Spektrum $K(X)$ definiert als $K(\mathrm{Perf}(X))$.
\end{definition}

Wir möchten einige Anmerkungen zu dieser Definition machen. Es bezeichne $\mathcal{D}_{\mathrm{qc}}(X)$ die derivierte $\infty$-Kategorie der quasi-kohärenten Garben auf $X$. Man kann sich die derivierte $\infty$-Kategorie der perfekten Komplexe auf $X$ als die volle $\infty$-Unterkategorie von $\mathcal{D}_{\mathrm{qc}}(X)$ vorstellen, die „lokal“ durch Vektorbündel unter Verschiebungen und endlichen Kolimites erzeugt wird. Diese Beschreibung hat den Vorteil, dass der Übergang von Quillens Definition zu Thomasons, zumindest aus heutiger Sicht, nicht so überraschend erscheint: Man folgt dabei Grothendieck und geht zum „derivierten Niveau“ über. Es gibt jedoch eine andere Beschreibung, die aus vielen konzeptionellen Gründen viel wichtiger ist. Nämlich ist die derivierte $\infty$-Kategorie $\mathcal{D}_{\mathrm{qc}}(X)$ der quasi-kohärenten Garben auf einem quasi-kompakten quasi-separierten Schema $X$ \textit{kompakt erzeugt} und die $\infty$-Unterkategorie $\mathrm{Perf}(X)$ ist genau die \textit{$\infty$-Unterkategorie der kompakten Objekte} von $\mathcal{D}_{\mathrm{qc}}(X)$. Mithilfe dieser Charakterisierung gelang es Thomason, viele geometrische Fragen über die $K$-Theorie der Schemata auf formale Eigenschaften kompakt erzeugter Kategorien und des $K$-Theorie-Funktors zu reduzieren. Da letzterer eine universelle „algebraische“ Invariante kleiner stabiler $\infty$-Kategorien ist, schafft diese Umformulierung einen formalen Rahmen, in dem viele sonst sehr komplizierte Argumente klar werden.

Wie bereits erwähnt, besteht der Hauptvorteil von Thomasons Definition darin, dass seine Version der $K$-Theorie Abstieg bezüglich der Zariski-Topologie erfüllt. Tatsächlich erfüllt sie zudem Abstieg bezüglich der sogenannten \textit{Nisnevich-Topologie}, welche feiner als die Zariski-Topologie, aber gröber als die étale Topologie ist. Genauer hat Thomason für eine bestimmte Variante von $K:\mathrm{Cat}_{\infty}^{\mathrm{perf}}\xrightarrow{} \mathrm{Sp}$ den folgenden Satz bewiesen:

\begin{satz}[{\cite[Theorem 8.4]{TT90}}]\label{thomason}
Sei $X$ ein quasi-kompaktes quasi-separiertes Schema und $U$ ein beliebiges quasi-kompaktes étales Schema über $X$. Es bezeichne $\mathbb{K}:\mathrm{Cat}_{\infty}^{\mathrm{perf}}\xrightarrow{} \mathrm{Sp}$ den nicht konnektiven $K$-Theorie-Funktor. Dann definiert die Zuordnung $U \mapsto \mathbb{K}(\mathrm{Perf}(U))$ eine Garbe von Spektren bezüglich der Nisnevich-Topologie.
\end{satz}

Vor der Veröffentlichung von \cite{TT90} gab es mehrere Ad-hoc-Versuche, den Zariski-Abstieg für die $K$-Theorie unter bestimmten Annahmen an $X$ zu beweisen. Der wichtigste Punkt in Thomasons Arbeit ist, dass er den optimalen Abstiegssatz beweisen konnte. In der vorliegenden Arbeit untersuchen wir seine Ideen im Kontext der adischen Geometrie. Dort gibt es unter verschiedenen Annahmen an den betrachteten adischen Raum ebenfalls nicht triviale Definitionen der $K$-Theorie und zugehörige Abstiegsresultate. Unser abstrakter Ansatz erlaubt es jedoch, wie im Falle der Schemata, den optimalen Abstiegssatz in dieser Situation zu beweisen. Offensichtlich sind wir aber nicht die ersten, die über die Anwendung von Thomasons Methoden in der adischen Geometrie nachgedacht haben. Wir beschreiben nun kurz zwei Probleme im adischen Fall, die den Fortschritt seit Jahren behinderten und eine naive Übertragung der oben skizzierten Theorie auf den adischen Fall unmöglich machten. Deren Lösung basiert auf einigen theoretischen Grundlagen, die erst seit kurzer Zeit zur Verfügung stehen.

Das erste Problem lässt sich darauf zurückführen, dass es vor der Erfindung der verdichteten Mathematik durch Clausen und Scholze keinen ausreichend flexiblen algebraischen Formalismus der Garben von Moduln auf analytischen Räumen (sowohl nicht-archimedischen als auch komplexen) ohne Endlichkeitsbedingungen gab, d. h., es gab keine Theorie der quasi-kohärenten Garben in diesen Kontexten. Man könnte hoffen, dass dieser Grad an Allgemeinheit für die Zwecke der $K$-Theorie tatsächlich überflüssig ist, denn auf den ersten Blick verwenden alle Definitionen im algebraischen Falle nur kohärente Garben. Außerdem gibt es unter bestimmten noetherschen Annahmen, die von rigid-analytischen Räumen erfüllt sind, eine gute Theorie der kohärenten Garben. Leider ist diese aus mehreren Gründen für unsere Zwecke nicht ausreichend. Der wichtigste davon ist, dass die $K$-Theorie, die die derivierte Kategorie der perfekten Komplexe als formale Eingabe verwenden würde, Abstieg nicht erfüllen würde: In der rigiden Geometrie gibt es kein Analogon von Thomasons Offen-Abgeschlossen-Sequenz für perfekte Komplexe, welche die Schlüsselrolle in seinem Beweis spielt. Ein weiterer Grund liegt darin, dass wir nach einem solchen Formalismus suchen, der auch für eine bestimmte Klasse von nicht noetherschen adischen Räumen, nämlich den sogenannten \textit{perfektoiden} Räumen, funktioniert. Bevor wir uns der $K$-Theorie zuwenden können, müssen wir also diese Grundsatzfrage klären. In \cite{firstpaper} haben wir eine Definition der derivierten $\infty$-Kategorie der quasi-kohärenten Garben für allgemeine analytische adische Räume eingeführt. Wie bereits erwähnt, basiert unser Ansatz auf den Ideen von Clausen-Scholze und verwendet verdichtete Mathematik in wesentlicher Weise. Ohne zu sehr ins Detail zu gehen, möchten wir dennoch einige der wichtigsten Ergebnisse jener Arbeit präzise darlegen.

\begin{satz}[\cite{firstpaper}, Theorem 3.28, Proposition 3.34, Theorem 5.9, Lemma 5.10]\label{analytische Ringe}
Es bezeichne $\mathrm{CAff}$ die Kategorie der vollständigen Huber-Paare und $\mathrm{AnRing}$ die Kategorie der vollständigen kommutativen analytischen animierten verdichteten Ringe im Sinne von Clausen-Scholze. Dann gibt es einen volltreuen Funktor 
\[\mathrm{CAff}\xrightarrow{}\mathrm{AnRing},\ (A,A^+)\mapsto (A,A^+)_{\blacksquare}.\] Insbesondere erhält man für jeden affinoiden adischen Raum $X=\Spa (A,A^+)$ eine stabile $\infty$-Kategorie $\mathcal{D}_\blacksquare(X)$, die als die derivierte $\infty$-Kategorie $\mathcal{D}((A,A^+)_{\blacksquare})$ der festen Moduln über $(A,A^+)_{\blacksquare}$ definiert ist. Außerdem existiert eine volltreue Einbettung der üblichen algebraischen derivierten $\infty$-Kategorie $\mathcal{D}(A)$ in $\mathcal{D}_\blacksquare(X)$, die mit dem Basiswechsel kompatibel ist.
\end{satz}

\begin{satz}[\cite{firstpaper}, Theorem 4.1]\label{A}
Sei $X$ ein analytischer adischer Raum und $U=\Spa (A,A^+)$ eine beliebige affinoide offene Teilmenge von $X$. Dann definiert der Funktor, der $U$ auf die derivierte $\infty$-Kategorie $\mathcal{D}((A,A^+)_{\blacksquare})$ abbildet, eine Garbe von $\infty$-Kategorien auf $X$ bezüglich der analytischen Topologie. Insbesondere erhält man durch Verkleben eine stabile $\infty$-Kategorie $\mathcal{D}_\blacksquare(X)$, die wir die \textit{derivierte $\infty$-Kategorie der festen Garben auf $X$} nennen.
\end{satz}

Es sei erwähnt, dass das Analogon dieser Aussage, bei dem $\mathcal{D}((A,A^+)_{\blacksquare})$ durch $\mathcal{D}(A)$ ersetzt wird, falsch ist. Daher ist es wichtig anstatt mit den üblichen „diskreten“ Moduln mit verdichteten Moduln zu arbeiten. Eine Schlüsselrolle in der vorliegenden Arbeit spielt jedoch nicht die $\infty$-Kategorie $\mathcal{D}((A,A^+)_{\blacksquare})$, sondern deren volle $\infty$-Unterkategorie der sogenannten \textit{nuklearen Moduln}, welche ebenfalls Abstieg erfüllt und als Analogon der Kategorie der quasi-kohärenten Garben für unsere Zwecke angesehen werden kann. Der Grund dafür ist, dass $\mathcal{D}((A,A^+)_{\blacksquare})$ zwar kompakt erzeugt ist, seine volle Unterkategorie der kompakten Objekte allerdings zu groß ist, was dazu führt, dass die mit ihrer Hilfe definierte $K$-Theorie von $X$ trivial wäre. Die Kategorie der nuklearen Garben hat dagegen „die richtige Größe“, was sich insbesondere im diskreten Fall zeigt: Für ein diskretes Huber-Paar $(A,A^+)$ ist sie äquivalent zur derivierten Kategorie $\mathcal{D}(A)$. Darüber hinaus kann für nukleare Garben eine „Offen-Abgeschlossen-Sequenz“ konstruiert werden, was die Anwendung der Methoden von Thomason ermöglicht. Der Übergang zu dieser Kategorie führt jedoch zu dem zweiten Problem: Im Allgemeinen ist die $\infty$-Kategorie der nuklearen Garben nicht kompakt erzeugt, sondern nur dualisierbar. Für dieses Problem gibt es jedoch eine rein kategorielle Lösung nach Efimov, die das letzte Hindernis für die Übertragung des Ansatzes von Thomason auf den adischen Fall beseitigt:

\begin{satz}[\cite{Efimov}, Theorem 10]\label{efimov}
Es bezeichne $\mathrm{Cat}_{\infty}^{\mathrm{st,cg}}$ (bzw. $\mathrm{Cat}_{\infty}^{\mathrm{st,dual}}$) die $\infty$-Kategorie der stabilen kompakt erzeugten $\infty$-Kategorien (bzw. die $\infty$-Kategorie der stabilen dualisierbaren $\infty$-Kategorien), deren Morphismen kompakte (bzw. dualisierbare) Funktoren sind. Man betrachte den Funktor 
\[\mathrm{Cat}_{\infty}^{\mathrm{st,cg}}\xrightarrow{\makebox[1.5cm]{$\mathcal{C}\mapsto \mathcal{C}^\omega$}}\mathrm{Cat}_{\infty}^{\mathrm{perf}}\xrightarrow{\makebox[1.5cm]{$\mathbb{K}$}} \mathrm{Sp},\] 
wobei $\mathcal{C}^\omega$ die volle $\infty$-Unterkategorie der kompakten Objekte in $\mathcal{C}$ bezeichnet. Dann besitzt er eine eindeutige Fortsetzung zu einem Funktor $\mathbb{K}_\mathrm{stet}: \mathrm{Cat}_{\infty}^{\mathrm{st,dual}}\rightarrow \mathrm{Sp}$, der Verdier-Sequenzen auf Fasersequenzen abbildet.
\end{satz}

Die vorliegende Dissertation ist wie folgt gegliedert. In Abschnitt 2 legen wir den notwendigen formalen kategoriellen Rahmen für unsere Arbeit. Wir erinnern an die Begriffe der kompakt erzeugten und der dualisierbaren Kategorie und untersuchen ihre Beziehung. Unser wichtigstes Resultat ist Satz~\ref{dualisierbarer_abstieg}, der sich mit den Abstiegseigenschaften dualisierbarer $\infty$-Kategorien beschäftigt. In Abschnitt 3 erinnern wir an die Definition eines nuklearen Moduls über einem analytischen verdichteten Ring nach Clausen und Scholze. Dann untersuchen wir für ein Huber-Paar $(A,A^+)$ die $\infty$-Kategorie der nuklearen Moduln über $(A,A^+)_{\blacksquare}$, wobei $(A,A^+)_{\blacksquare}$ den über den Funktor $\mathrm{CAff}\xrightarrow{}\mathrm{AnRing},\ (A,A^+)\mapsto (A,A^+)_{\blacksquare}$ konstruierten analytischen Ring bezeichnet. Unser Hauptergebnis ist die Tatsache, dass diese $\infty$-Kategorie für eine große Klasse von Huber-Paaren dualisierbar ist. Insbesondere ist sie für alle garbigen tateschen Huber-Paare dualisierbar. In Abschnitt 4 verallgemeinern wir diese Aussage mithilfe der Resultate aus Abschnitt 2 auf den Fall allgemeiner analytischer adischer Räume. Konkret führen wir die $\infty$-Kategorie der nuklearen Garben auf einem analytischen adischen Raum ein und beweisen, dass sie für quasi-kompakte quasi-separierte Räume ebenfalls dualisierbar ist. Der nächste logische Schritt findet sich im Anhang. In dessen ersten Teil definieren wir die Nisnevich-Topologie für analytische adische Räume und zeigen dann, dass sie die zu erwartenden Eigenschaften besitzt. Insbesondere beweisen wir, dass sie durch eine $cd$-Struktur im Sinne von Voevodsky gegeben werden kann. Dann untersuchen wir im nächsten Abschnitt des Anhangs die Beziehung zwischen Garben und Hypergarben bezüglich der analytischen, Nisnevich- und étalen Topologien eines analytischen adischen Raums. Dabei folgen wir dem Ansatz von \cite{CM21} im schematischen Falle und erhalten ähnliche Resultate. Mit anderen Worten ist das Ziel des Anhangs zu zeigen, dass die Eigenschaften der Nisnevich- bzw. der étalen Topologie für Schemata und analytische adische Räume analog sind. In Abschnitt 5 verwendet wir unsere bisherigen Ergebnisse zur Analyse der Abstiegseigenschaften von lokalisierenden Invarianten auf analytischen adischen Räumen. Dabei folgen wir den Ideen von Thomason in ihrerer modernen Form wie in \cite{CMNN20}, \cite{K1} und \cite{CM21}. Insbesondere beweisen wir das Analogon des Satzes~\ref{thomason} und untersuchen die étalen Abstiegseigenschaften lokalisierender Invarianten nach chromatischer Lokalisierung. Im letzten Abschnitt dieser Arbeit formulieren und beweisen wir ein Analogon des Satzes von Grothendieck-Riemann-Roch, dessen Bedeutung keiner weiteren Erläuterung bedarf.

\bigskip\large{\textbf{Vorschläge zum Lesen:}}\normalsize

\begin{adjustwidth}{-0.17cm}{}
\begin{tabular}{p{0.48\linewidth}p{0.48\linewidth}}
\multicolumn{1}{l}{Bist du ein*e:} & \multicolumn{1}{l}{so lese die Arbeit in dieser Folge:} \\
Ehrenfrau*mann & §2, §3, §4, Anhang A und B, §5 und §6; \\
draufgängerische*r Schummler*in & §2, §3, §4, §5 und §6 mit der Analogie zwischen Schemata und adischen Räumen im Hinterkopf;\\
Adrenalinsüchtige*r  & §4, §5 und §6.
  \end{tabular}
\end{adjustwidth}

\bigskip\large{\textbf{Danksagung}}\normalsize

Mein herzlicher Dank gilt meinem Doktorvater Peter Scholze für seine umfassende und äußerst engagierte Betreuung. Trotz der Seltenheit unserer Treffen habe ich in der Zeit als Doktorand so viel gelernt, dass es wohl noch einige Jahre dauern wird, bis ich alles verstanden habe, was er mir beigebracht hat. Dank ihm bin ich hoffentlich endlich von einem Sammler zufälligen mathematischen Wissens zu einem Mathematiker geworden. Ich möchte auch Dustin Clausen meinen aufrichtigen Dank aussprechen. Er hat mir und Peter das Thema dieser Dissertation vorgeschlagen und hat mich während meiner Promotion mehrfach unterstützt. Darüber hinaus möchte ich mich bei Christian Dahlhausen und Georg Tamme für den Workshop „Condensed Mathematics and K-Theory“ in Mainz bedanken, wo ich die Chance hatte, über meine Doktorarbeit zu berichten. Natürlich haben Diskussionen mit vielen Menschen diese Arbeit beeinflusst, aber ich möchte besonders jene mit Alexander Efimov, Jacob Lurie, Lucas Mann und Ko Aoki erwähnen. Wenn ihr diese Zeilen lest, wisst ihr, dass irgendwo hier in meinen unbeholfenen Worten euer mathematisches Wissen begraben liegt. Ein herzliches Dankeschön geht auch an Ferdinand Wagner, der mir erlaubt hat, seine wunderschöne \TeX-Version von Grothendiecks berühmter Zeichnung in meiner Arbeit zu verwenden. Abschließend möchte ich mich bei Céline Fietz, Youshua Kesting und Maximilian von Consbruch bedanken, ohne die diese Worte nicht nur unbeholfenen, sondern auch grammatikalisch falsch wären: Mit ihrer Hilfe war es mir möglich, meinen kleinen Traum zu verwirklichen und diese Arbeit – zur Verwirrung vieler – in deutscher Sprache zu verfassen.
\newpage
\section{Dualisierbare Kategorien und deren Verklebung}\label{abstrakter_unsinn}

Eine der wichtigsten mathematischen Leistungen von Robert Thomason besteht zweifellos darin, dass er den „richtigen“ Blickwinkel auf die $K$-Theorie der Schemata gefunden hat. Vor der Veröffentlichung seiner bahnbrechenden Arbeit \cite{TT90} wurde die $K$-Theorie eines Schemas $X$ durch Anwenden der sogenannten Quillen'schen $Q$-Konstruktion auf die Kategorie der Vektorbündel auf $X$ definiert. Ein wesentlicher Grund, weshalb dieser Ansatz konzeptionell unbefriedigend blieb, war die Tatsache, dass die so definierte $K$-Theorie im Allgemeinen Zariski-Abstieg nicht erfüllte. Die Haupterkenntnis von Thomason, die er der Schule von Grothendieck zuschrieb, war die Idee, statt Vektorbündeln sogenannte \textit{perfekte Komplexe} zu betrachten. Per Definition sind dies jene Objekte der derivierten Kategorie der quasi-kohärenten Garben, die „lokal“ als endliche Kolimites von Verschiebungen von Vektorbündeln geschrieben werden können. Man könnte sagen, dies ist eine weitere Manifestation der heute in der algebraischen Geometrie allgegenwärtigen Idee, dass man bei der Untersuchung vieler Fragen konsequent auf dem „derivierten Niveau“ arbeiten sollte. Tatsächlich sind die perfekten Komplexe auf einem Schema $X$ genau durch die kompakten Objekte in der derivierten Kategorie der quasi-kohärenten Garben auf $X$ gegeben, was eine abstrakte Charakterisierung derselben bietet. Dementsprechend spielt der Begriff der \textit{kompakt erzeugten Kategorie} eine besondere Rolle in Thomasons Ansatz zur $K$-Theorie, und der Hauptteil der Arbeit \cite{TT90} besteht in der Untersuchung der Abstiegseigenschaften solcher Kategorien. Wir erinnern an die Definition:

\begin{definition}[{\cite[Definition 5.5.7.1]{HTT}}]
Sei $\mathcal{C}$ eine zugängliche $\infty$-Kategorie\footnote{D. h., $\mathcal{C}\cong \Ind_\kappa (\mathcal{C}^0)$ für eine kleine $\infty$-Kategorie $\mathcal{C}^0$ und eine reguläre Kardinalzahl $\kappa$, siehe \cite[Definition 5.4.2.1]{HTT}.}, abgeschlossen unter kleinen filtrierten Kolimites.
\begin{enumerate}[label=(\roman*)]
\item Ein Objekt $c\in\mathcal{C}$ heißt \textit{kompakt}, wenn der Funktor $\mathrm{hom}_\mathcal{C}(c,-)$ mit kleinen filtrierten Kolimites vertauscht. Wir schreiben $\mathcal{C}^\omega$ für die volle $\infty$-Unterkategorie der kompakten Objekte in $\mathcal{C}$.
\item Die $\infty$-Kategorie $\mathcal{C}$ heißt \textit{kompakt erzeugt}, wenn $\Ind (\mathcal{C}^\omega)\xrightarrow[]{\sim} \mathcal{C}$ gilt. Das ist dazu äquivalent, dass es eine kleine $\infty$-Kategorie $\mathcal{D}_0$ gibt, sodass $\mathcal{C}\cong \Ind \mathcal{D}_0$.
\end{enumerate}
\end{definition}

Sei $X$ ein quasi-kompaktes quasi-separiertes Schema. Es bezeichne $\mathcal{D}_{qc}(X)$ die derivierte $\infty$-Kategorie der quasi-kohärenten Garben auf $X$. Ihre volle $\infty$-Unterkategorie der kompakten Objekte heißt die \textit{derivierte $\infty$-Kategorie der perfekten Komplexe auf $X$} und wird mit $\mathrm{Perf}(X)$ bezeichnet. Nach Thomason ist die $K$-Theorie von $X$ durch Anwenden des $K$-Theorie-Funktors $K:\mathrm{Cat}_\infty^{\mathrm{st}}\rightarrow \mathrm{Sp}$ auf $\mathrm{Perf}(X)$ definiert. Einer der Hauptvorteile dieser Definition besteht darin, dass der Zariski- und der Nisnevich-Abstieg der $K$-Theorie von $X$ nun eine formale Konsequenz der Tatsache sind, dass die $\infty$-Kategorie $\mathcal{D}_{qc}(X)$ kompakt erzeugt ist und die kompakten Objekte den jeweiligen Abstieg erfüllen. Um einzusehen, dass $\mathcal{D}_{qc}(X)$ tatsächlich kompakt erzeugt ist, zitieren wir folgenden Satz:

\begin{satz}[{\cite[Beweis des Satzes 3.1.1]{BB02}}]\label{kompakte_verklebung}
Sei $X$ ein quasi-kompakter quasi-separierter topologischer Raum mit einer Basis $\mathcal{B}$, welche aus quasi-kompakten Teilmengen besteht. Sei $U\mapsto \mathcal{C}_U$ eine Garbe von stabilen $\infty$-Kategorien auf $X$, deren Wert auf jeder Teilmenge $B\in\mathcal{B}$ kompakt erzeugt ist. Angenommen, für alle $V\in \mathcal{B}$ und alle quasi-kompakten offenen Teilmengen $V'\subset V$ ist der Einschränkungsfunktor $\mathcal{C}_{V}\rightarrow \mathcal{C}_{V'}$ eine linke Bousfield-Lokalisierung, die kompakte Objekte erhält und deren Faser kompakt erzeugt ist. Dann ist die Kategorie $\mathcal{C}_X$ kompakt erzeugt.
\end{satz}

Aus diesem Satz lässt sich die gewünschte Eigenschaft von $\mathcal{D}_{qc}(X)$ leicht ableiten, denn für jedes affine Schema $X=\Spec R$ ist $\mathcal{D}_{qc}(X)$ äquivalent zur derivierten $\infty$-Kategorie $\mathcal{D}(R)$ und somit kompakt erzeugt.\footnote{Für den Beweis, dass die Einschränkungsfunktoren für quasi-kohärente Garben die technische Bedingung des Satzes erfüllen, verweisen wir auf \cite[Proposition 6.1]{BN93}.} Wenn wir Thomasons Ideen zur Definition der $K$-Theorie in der rigiden Geometrie anwenden wollen, brauchen wir also eine „Kategorie der quasi-kohärenten Garben“ für adische Räume. In Abschnitt 4 führen wir eine solche Kategorie ein, die wir die Kategorie der \textit{nuklearen Garben} nennen, und untersuchen ihre Eigenschaften. Das Hauptproblem, mit dem wir uns im vorliegenden Abschnitt befassen, ist die Tatsache, dass die Kategorie der nuklearen Garben selbst im affinoiden Fall nicht kompakt erzeugt, sondern nur \textit{dualisierbar} ist. Daher müssen wir zunächst die Abstiegseigenschaften dualisierbarer Kategorien untersuchen. Da diese eine sehr natürliche Verallgemeinerung der kompakt erzeugten Kategorien sind, sind die Situationen sehr ähnlich. Am Ende dieses Abschnitts beweisen wir ein Analogon von Satz~\ref{kompakte_verklebung} für dualisierbare Kategorien, welches wir dann in Abschnitt 4 anwenden, um die Dualisierbarkeit der Kategorie der nuklearen Garben auf einem analytischen adischen Raum nachzuweisen. Wir beginnen mit einer Erinnerung an alle notwendigen kategoriellen Begriffe, insbesondere an die Definition und die grundlegenden Eigenschaften dualisierbarer Kategorien.

\begin{notation} 
\begin{enumerate}[label=(\roman*)]
\item[]
\item Mit $\widehat{\mathrm{Cat}}_{\infty}^\omega$ bezeichnen wir die $\infty$-Kategorie der zugänglichen unter kleinen filtrierten Kolimites abgeschlossenen $\infty$-Kategorien, deren Morphismen kleine filtrierte Kolimites erhaltende Funktoren sind. 
\item Mit $\mathcal{P}r^{\mathrm{L}}$ bezeichnen wir die $\infty$-Kategorie der präsentierbaren $\infty$-Kategorien, deren Morphismen (kleine) Kolimites erhaltende Funktoren sind.
\item Mit $\mathcal{P}r^{\mathrm{St}}$ bezeichnen wir die volle $\infty$-Unterkategorie von $\mathcal{P}r^{\mathrm{L}}$ der stabilen präsentierbaren $\infty$-Kategorien.
\item Mit $\mathrm{Cat}_{\infty}^{\mathrm{perf}}$ bezeichnen wir die $\infty$-Kategorie der kleinen idempotent vollständigen stabilen $\infty$-Kategorien, deren Morphismen exakte Funktoren sind.
\end{enumerate}
\end{notation}

Wir erinnern zunächst an die Definition einer \textit{Verdier-Sequenz} von stabilen $\infty$-Kategorien.

\begin{definition}
Sei $\mathcal{C}\xrightarrow[]{F} \mathcal{D}\xrightarrow[]{G} \mathcal{E}$ eine Sequenz in $\mathrm{Cat}_{\infty}^{\mathrm{perf}}$ (bzw. $\mathcal{P}r^{\mathrm{St}}$) mit $G\circ F\cong 0$. Sie heißt eine \textit{Verdier-Sequenz}, falls sie eine Kofasersequenz und der Funktor $F$ volltreu ist.
\end{definition}

Im Folgenden sind die Verdier-Sequenzen, bei denen der Funktor $G$ eine linke Bousfield-Lokalisierung ist, von besonderer Bedeutung. In diesem Falle besitzt der Funktor $F$ ebenso einen rechtsadjungierten Funktor, wie die beiden folgenden Lemmata zeigen.

\begin{lemma}[{\cite[Lemma A.2.5]{HKT2}}]\label{verdier-sequenzen}
Für eine Sequenz $\mathcal{C}\xrightarrow[]{F} \mathcal{D}\xrightarrow[]{G} \mathcal{E}$ in $\mathrm{Cat}_{\infty}^{\mathrm{perf}}$ mit $G\circ F\cong 0$ sind äquivalent:
\begin{enumerate}[label=(\roman*)]
\item Die Sequenz $\mathcal{C}\xrightarrow[]{F} \mathcal{D}\xrightarrow[]{G} \mathcal{E}$ ist eine Kofasersequenz und der Funktor $G$ besitzt einen volltreuen rechts- bzw. linksadjungierten Funktor. 
\item Die Sequenz $\mathcal{C}\xrightarrow[]{F} \mathcal{D}\xrightarrow[]{G} \mathcal{E}$ ist eine Fasersequenz und der Funktor $F$ ist volltreu und besitzt einen rechts- bzw. linksajungierten Funktor.
\end{enumerate}
\end{lemma}

Man beweist in ähnlicher Weise das folgende Analogon des obigen Lemmas für Verdier-Sequenzen in $\mathcal{P}r^{\mathrm{St}}$.

\begin{lemma}
Für eine Sequenz $\mathcal{C}\xrightarrow[]{F} \mathcal{D}\xrightarrow[]{G} \mathcal{E}$ in $\mathcal{P}r^{\mathrm{St}}$ mit $G\circ F\cong 0$ sind äquivalent:
\begin{enumerate}[label=(\roman*)]
\item Die Sequenz $\mathcal{C}\xrightarrow[]{F} \mathcal{D}\xrightarrow[]{G} \mathcal{E}$ ist eine Kofasersequenz und der Funktor $G$ besitzt einen volltreuen rechtsadjungierten Funktor, der mit Kolimites vertauscht.
\item Die Sequenz $\mathcal{C}\xrightarrow[]{F} \mathcal{D}\xrightarrow[]{G} \mathcal{E}$ ist eine Fasersequenz und der Funktor $F$ ist volltreu und besitzt einen rechtsajungierten Funktor, der mit Kolimites vertauscht.
\end{enumerate}
\end{lemma}

\begin{definition}
Sei $\mathcal{C}\xrightarrow[]{F} \mathcal{D}\xrightarrow[]{G} \mathcal{E}$ eine Verdier-Sequenz in $\mathrm{Cat}_{\infty}^{\mathrm{perf}}$ (bzw. $\mathcal{P}r^{\mathrm{St}}$). Sie heißt \textit{rechts spaltend}, falls der Funktor $G$ einen volltreuen rechtsadjungierten Funktor (bzw. einen volltreuen Kolimites erhaltenden rechtsadjungierten Funktor) besitzt.
\end{definition}

Wir führen nun den Begriff der Dualisierbarkeit für stabile $\infty$-Kategorien ein und erinnern an ihre Grundeigenschaften.

\begin{lemma}\label{dualisierbar}
Für eine $\infty$-Kategorie $\mathcal{C}$ sind äquivalent:
\begin{enumerate}[label=(\roman*)]
\item Die Kategorie $\mathcal{C}$ ist ein Retrakt einer kompakt erzeugten Kategorie in der Kategorie $\widehat{\mathrm{Cat}}^{\omega}$.
\item Die Kategorie $\mathcal{C}$ ist ein Retrakt einer kompakt erzeugten Kategorie in der Kategorie $\mathcal{P}r^{\mathrm{L}}$.
\item Die Kategorie $\mathcal{C}$ ist präsentierbar und der Funktor $\colim: \Ind \mathcal{C}\rightarrow \mathcal{C}$ besitzt einen linksadjungierten Funktor.
\end{enumerate}
Ist $\mathcal{C}$ zudem stabil, so sind die obigen Bedingungen äquivalent zu den folgenden Bedingungen:
\begin{enumerate}[label=(\roman*)]
 \setcounter{enumi}{3}
\item Die Kategorie $\mathcal{C}$ ist ein Retrakt einer kompakt erzeugten Kategorie in der Kategorie $\mathcal{P}r^{\mathrm{St}}$.
\item Die Kategorie $\mathcal{C}$ ist ein dualisierbares Objekt in $\mathcal{P}r^{\mathrm{St}}$ bezüglich des Lurie-Tensorprodukt.
\end{enumerate}
\end{lemma}

\begin{proof}
Kombiniere \cite[Theorem 21.1.2.10]{SAG}, \cite[Korollar 21.1.2.18]{SAG} und \cite[Proposition D.7.3.1]{SAG}.
\end{proof}

\begin{definition}
Eine $\infty$-Kategorie (bzw. stabile $\infty$-Kategorie) $\mathcal{C}$ heißt \textit{kompakt gefertigt} (bzw. \textit{dualisierbar}), wenn sie eine der äquivalenten Bedingungen $(i)$ -- $(iii)$ (bzw. $(i)$ -- $(v)$) von Lemma~\ref{dualisierbar} erfüllt.
\end{definition}

\begin{bemerkung}
Aus Lemma~\ref{dualisierbar} folgt, dass insbesondere jede stabile kompakt erzeugte $\infty$-Kategorie $\mathcal{C}$ ein dualisierbares Objekt in $\mathcal{P}r^{\mathrm{St}}$ ist. Man kann zeigen, dass das Dual von $\mathcal{C}$ durch $\Ind ((\mathcal{C}^\omega)^\mathrm{op})$ gegeben ist.
\end{bemerkung}

Oft ist es notwendig, nicht allgemeine Funktoren in $\mathcal{P}r^{\mathrm{St}}$ zu betrachten, sondern nur solche, die kompakte Objekte erhalten. Ein Funktor zwischen kompakt erzeugten $\infty$-Kategorien, der diese Eigenschaft besitzt, heißt \textit{kompakt}. Ein Grund für die Beschränkung auf diese Klasse von Funktoren ist die Tatsache, dass die $K$-Theorie kompakt erzeugter $\infty$-Kategorien nur bezüglich solcher Morphismen $\mathcal{P}r^{\mathrm{St}}$ funktoriell ist. Ein weiterer Grund, der besonders wichtig für unser aktuelles Ziel ist, besteht darin, dass diese Klasse von Funktoren in der Formulierung des Satzes~\ref{kompakte_verklebung} auftritt. Die gleiche Definition ist jedoch für dualisierbare $\infty$-Kategorien nicht fein genug, da sie nicht genügend kompakte Objekte besitzen. Um eine passende Klasse von Funktoren im Falle von dualisierbaren $\infty$-Kategorien zu finden, beachten wir die folgende äquivalente Beschreibung der kompakten Funktoren, die für unsere Zwecke besser geeignet ist: 

\begin{lemma}\label{kompakte_funktoren}
Es seien $\mathcal{C}$ und $\mathcal{D}$ präsentierbare $\infty$-Kategorien und $F:\mathcal{C}\rightarrow \mathcal{D}$ ein linksadjungierter Funktor. Es bezeichne $G$ den rechtsadjungierten Funktor von $F$. Vertauscht $G$ mit filtrierten Kolimites, so erhält $F$ kompakte Objekte. Ist die Kategorie $\mathcal{C}$ kompakt erzeugt, so vertauscht $G$ mit filtrierten Kolimites genau dann, wenn $F$ kompakte Objekte erhält.
\end{lemma}

\begin{proof}
Sei $X$ ein kompaktes Objekt in $\mathcal{C}$. Dann gilt
\[\mathrm{Hom}_{\mathcal{D}}(F(X),\colim Y_i)\cong\mathrm{Hom}_{\mathcal{C}}(X,\colim G(Y_i))\cong \colim \mathrm{Hom}_{\mathcal{D}}(F(X),Y_i).\]

Angenommen, $\mathcal{C}$ ist kompakt erzeugt und $F$ erhält kompakte Objekte. Sei $X$ ein kompaktes Objekt in $\mathcal{C}$. Wir führen folgende elementare Rechnung durch:
\[\mathrm{Hom}_{\mathcal{C}}(X,G(\colim Y_i))\cong\mathrm{Hom}_{\mathcal{D}}(F(X),\colim Y_i)\cong \colim \mathrm{Hom}_{\mathcal{D}}(F(X),Y_i)\]
\[\cong \colim \mathrm{Hom}_{\mathcal{C}}(X,G(Y_i))\cong \mathrm{Hom}_{\mathcal{C}}(X,\colim G(Y_i)).\qedhere\]\end{proof}

\begin{definition}
Es seien $\mathcal{C}$ und $ \mathcal{D}$ dualisierbare $\infty$-Kategorien und $F:\mathcal{C}\rightarrow \mathcal{D}$ ein linksadjungierter Funktor. Es bezeichne $G$ den rechtsadjungierten Funktor von $F$. Der Funktor $F$ heißt \textit{dualisierbar}, wenn $G$ mit filtrierten Kolimites vertauscht.
\end{definition}

Sei $\mathcal{F}$ eine Garbe von stabilen $\infty$-Kategorien auf einem topologischen Raum $X$. Wenn wir mithilfe des Satzes~\ref{kompakte_verklebung} zeigen wollen, dass der Wert von $\mathcal{F}$ auf $X$ kompakt erzeugt ist, genügt es nicht zu wissen, dass $\mathcal{F}$ „lokal“ eine Garbe kompakt erzeugter $\infty$-Kategorien ist. Die Garbe $\mathcal{F}$ muss noch eine weitere wichtige Bedingung erfüllen: Die Fasern der Lokalisierungsfunktoren müssen ebenso kompakt erzeugt sein. Die analoge Bedingung für dualisierbare Kategorien ist jedoch automatisch erfüllt, wie das folgende Lemma zeigt. 

\begin{lemma}\label{dualisierbare_faser}
Es seien $\mathcal{C}$ und $\mathcal{D}$ stabile präsentierbare $\infty$-Kategorien und $F:\mathcal{C}\rightarrow \mathcal{D}$ eine linke Bousfield-Lokalisierung, deren rechtsadjungierter Funktor $G$ mit Kolimites vertauscht. Ist $\mathcal{C}$ dualisierbar, so ist die Faser von $F$ ebenso dualisierbar.
\end{lemma}

\begin{proof}
Es bezeichne $\mathcal{F}$ die Faser von $F$. Da der Funktor $G$ mit Kolimites vertauscht, liefert Lemma~\ref{verdier-sequenzen} folgende rechts spaltende Verdier-Sequenz:
\begin{center}
\begin{tikzcd}
\mathcal{F} \arrow[r,"F'"] & \mathcal{C} \arrow[r, "F"] \arrow[l, "G'", shift left=2] & \mathcal{D} \arrow[l, "G", shift left=2] 
\end{tikzcd}
\end{center}
Die $\infty$-Kategorie $\mathcal{F}$ ist also ein Retrakt von $\mathcal{C}$ in $\mathcal{P}r^L$ und somit dualisierbar.
\end{proof}

Es seien $\mathcal{C}$ und $\mathcal{D}$ stabile $\infty$-Kategorien und $F:\mathcal{C}\rightarrow \mathcal{D}$ eine linke Bousfield-Lokalisierung. Die einzige wesentliche Schwierigkeit des Beweises von Satz~\ref{kompakte_verklebung} liegt in der Hochhebung von kompakten Objekten in $\mathcal{D}$ nach $\mathcal{C}$. Im Falle von Schemata ist dies zu der Frage äquivalent, wann ein perfekter Komplex auf einer offenen Teilmenge $U$ von einem Schema $X$ zu einem perfekten Komplex auf ganz $X$ fortgesetzt werden kann. Als letzten Schritt, bevor wir uns dem Analogon von Satz~\ref{kompakte_verklebung} für dualisierbare Kategorien widmen, geben wir einen anderen (und etwas kompakteren) Beweis des folgenden klassischen Satzes, dessen Herzstück im sogenannten \textit{Thomason-Trick} liegt.

\begin{satz}[{\cite[Theorem 2.1]{Ne92}}]\label{allgemeiner_thomason-trick}
Es seien $\mathcal{C}$ und $\mathcal{D}$ stabile $\infty$-Kategorien und $F:\mathcal{C}\rightarrow \mathcal{D}$ eine linke Bousfield-Lokalisierung, deren rechtsadjungierter Funktor $G$ mit Kolimites vertauscht. Dann gilt:

\begin{enumerate}[label=(\roman*)]
\item Ist die $\infty$-Kategorie $\mathcal{C}$ dualisierbar, so ist auch die $\infty$-Kategorie $\mathcal{D}$ dualisierbar.

\item Angenommen, die $\infty$-Kategorie $\mathcal{C}$ ist kompakt erzeugt. Dann ist die Kategorie $\mathcal{D}$ ebenso kompakt erzeugt. Es seien $X$ ein kompaktes Objekt in $\mathcal{D}$, $Y$ ein beliebiges Objekt in $\mathcal{C}$ und $f$ ein Morphismus $f:X\rightarrow F(Y)$. Dann gibt es ein kompaktes Objekt $\Tilde{X}$ in $\mathcal{C}$ und einen Morphismus $\tilde{f}:\Tilde{X}\rightarrow Y$, sodass $F(\tilde{f})$ wie folgt faktorisiert:
\[F(\tilde{X})\xrightarrow{\sim} X\oplus X'\xrightarrow{\pi_{X}}X\xrightarrow{f}F(Y),\]
wobei $X'$ ein kompaktes Objekt in $\mathcal{D}$ ist. Insbesondere ist jedes kompakte Objekt in $\mathcal{D}$ ein Retrakt eines kompakten Objektes im wesentlichen Bild von $F$.

\item Thomasons Trick: Angenommen, die $\infty$-Kategorie $\mathcal{C}$ und die Faser von $F$ sind beide kompakt erzeugt. Es seien $X$ ein kompaktes Objekt in $\mathcal{C}$, $Y$ eine beliebiges Objekt in $\mathcal{C}$ und $f$ ein Morphismus $f:F(X)\rightarrow F(Y)$. Dann gibt es ein kompaktes Objekt $X'$ in $\mathcal{C}$ und Morphismen $\phi:X'\rightarrow X$ und $\psi:X'\rightarrow Y$, sodass $F(\phi)$ ein Isomorphismus ist und $f$ wie folgt faktorisiert:
\[f:F(X)\xrightarrow{F(\phi)^{-1}} F(X')\xrightarrow{F(\psi)}F(Y).\]
Mit anderen Worten besitzt jeder Morphismus in $\mathcal{D}$ der obigen Form eine Hochhebung nach $\mathcal{C}$. Insbesondere liegt ein kompaktes Objekt in $\mathcal{D}$ im wesentlichen Bild von $F$ genau dann, wenn seine Klasse in $K_0(\mathcal{D})$ in $\Ima K_0(\mathcal{C})$ liegt. 
\end{enumerate}
\end{satz}

\begin{proof}

\begin{enumerate}[label=(\roman*)]
\item Aus der Annahme folgt direkt, dass $\mathcal{D}$ ein Retrakt von $\mathcal{C}$ und somit dualisierbar ist.

\item Wie man leicht einsehen kann, folgt die erste Aussage aus der zweiten. Man schreibe das Objekt $G(X)\in \mathcal{C}$ als filtrierten Kolimes $\colim X_i$ von kompakten Objekten in $\mathcal{C}$. Da $G$ volltreu ist und $F$ mit Kolimites vertauscht, gilt $X\cong \colim F(X_i)$. Der Isomorphismus $X\cong \colim F(X_i)$ faktorisiert über ein $F(X_i)$, denn $X$ ist kompakt. Das Objekt $X$ ist also ein direkter Summand von $F(X_i)$. Es bezeichne $p_i$ den kanonischen Morphismus $F(X_i)\rightarrow X$ und $f'$ den Morphismus $X_i\rightarrow GF(Y)$, welcher der adjungierte Morphismus zu $f\circ p_i: F(X_i)\rightarrow F(Y)$ ist. Wir betrachten nun folgendes kommutative Diagramm:

\begin{center}
\begin{tikzcd}[]

 & Y \arrow[d] \\

X_i \arrow[rd, dashrightarrow] \arrow[r,"f'"] & GF(Y) \arrow[d] \\
 & \mathrm{cofib} (Y \rightarrow GF(Y))
\end{tikzcd}
\end{center}

Man schreibe $\mathrm{cofib} (Y \rightarrow GF(Y))$ als filtrierten Kolimes $\colim Z_j'$ von kompakten Objekten in $\mathcal{C}$. Da $X_i$ kompakt ist, faktorisiert die Abbildung $X_i\rightarrow \mathrm{cofib} (Y \rightarrow GF(Y))$ über einen Morphismus $g:X_i\rightarrow Z_j$. Außerdem dürfen wir annehmen, dass der Morphismus $F(g)$ trivial ist. Wir erhalten also das folgende kommutative Diagramm:

\begin{center}
\begin{tikzcd}[]

\mathrm{fib} (X_i \rightarrow Z_j)\arrow[r]\arrow[d] & Y \arrow[d] \\

X_i \arrow[d, "g"] \arrow[r,"f'"] & GF(Y) \arrow[d] \\
Z_j \arrow[r] & \mathrm{cofib} (Y \rightarrow GF(Y))
\end{tikzcd}
\end{center}

Man sieht nun ohne Mühe, dass das Paar $(\mathrm{fib} (X_i \rightarrow Z_j),\mathrm{fib} (X_i \rightarrow Z_j)\rightarrow Y)$ ein gewünschtes Paar $(\tilde{X},\tilde{f})$ bildet, denn $\mathrm{fib} (X_i \rightarrow Z_j)$ ist kompakt und es gilt $ F(\mathrm{fib} (X_i \rightarrow Z_j))\cong F(X_i)\oplus F (Z_j[-1])$.

\item Man betrachte das Diagramm oben und ersetze $X_i$ durch $X$. Man sieht dann leicht, dass wir annehmen dürfen, dass $F(Z_j)=0$ für jedes $Z_j$ gilt, denn die Faser von $F$ ist kompakt erzeugt. Daraus folgt unmittelbar, dass man Morphismen hochheben kann. Sei nun $K$ ein kompaktes Objekt in $\mathcal{D}$, dessen Klasse in $\Ima K_0(\mathcal{C})$ liegt, und $\Tilde{K}$ ein kompaktes Objekt in $\mathcal{C}$ mit $[F(\tilde{K})]=[K]\in K_0(\mathcal{D})$. Man prüft leicht nach, dass es ein kompaktes Objekt $K'\in\mathcal{D}$ gibt, sodass $K\oplus K'\cong F(\Tilde{K})\oplus K'$ gilt. Nach dem Teil $(ii)$ des Satzes dürfen wir annehmen, dass $K'$ im wesentlichen Bild von $F$ liegt, denn jedes kompakte Objekt in $\mathcal{D}$ ist ein Retrakt eines kompakten Objektes im wesentlichen Bild von $F$. Es genügt also zu zeigen, dass für jede Fasersequenz $K_1\rightarrow K_2 \rightarrow K_3$ von kompakten Objekten in $\mathcal{D}$, deren zwei von drei Termen im wesentlichen Bild von $F$ liegen, auch ihr dritter Term im wesentlichen Bild von $F$ liegt. Dies folgt aber unmittelbar aus der ersten Hälfte von $(iii)$.
\end{enumerate} 
\end{proof}

Wir haben nun alle Vorbereitungen für die Formulierung und den Beweis des Analogons von Satz~\ref{kompakte_verklebung} für dualisierbare $\infty$-Kategorien getroffen. Wir werden es als formale Folgerung des kompakt erzeugten Falles erhalten. Wir behandeln also die beiden Fälle gleichzeitig und beweisen unten auch Satz~\ref{kompakte_verklebung}.

\begin{satz}\label{dualisierbarer_abstieg}
Sei $X$ ein quasi-kompakter quasi-separierter topologischer Raum mit einer Basis $\mathcal{B}$, die aus quasi-kompakten Teilmengen besteht. Sei $U\mapsto \mathcal{C}_U$ eine Garbe von stabilen $\infty$-Kategorien auf $X$, deren Wert auf jeder Teilmenge $B\in\mathcal{B}$ kompakt erzeugt bzw. dualisierbar ist. Angenommen, für alle $V\in \mathcal{B}$  und alle quasi-kompakten offenen Teilmengen $V'\subset V$ ist der Einschränkungsfunktor $C_{V}\rightarrow C_{V'}$ eine kompakte bzw. dualisierbare linke Bousfield-Lokalisierung mit kompakt erzeugter bzw. dualisierbarer Faser (im dualisierbaren Fall ist die letzte Bedingung nach Lemma~\ref{dualisierbare_faser} und Satz~\ref{allgemeiner_thomason-trick} automatisch erfüllt). Dann ist die Kategorie $\mathcal{C}_X$ kompakt erzeugt bzw. dualisierbar.
\end{satz}

\begin{proof}
Wir betrachten zunächst den kompakt erzeugten Fall. Wir argumentieren über eine Induktion nach der minimalen Anzahl von offenen Teilmengen in einer Überdeckung $\mathcal{U}\subset \mathcal{B}$ von $X$. Seien $U_1,\dots, U_n$ die Elemente von $\mathcal{U}$. Es bezeichne $V$ (bzw. $U$) die Teilmenge $\underset{i=1}{\overset{n-1}{\cup}}U_i$ (bzw. $U_n$). Dann ist folgendes Diagramm kartesisch:
\begin{center}
\begin{tikzcd}
\mathcal{C}_X \arrow[r] \arrow[d] & \mathcal{C}_U \arrow[d] \\
\mathcal{C}_V \arrow[r] & \mathcal{C}_{U\cap V}
\end{tikzcd}\
\end{center}

Sei $M$ ein Objekt in $\mathcal{C}_X$. Es bezeichne $M|_{U}$ (bzw. $M|_{V}$ bzw. $M|_{V\cap U}$) das Bild von $M$ in $\mathcal{C}_U$ (bzw. $\mathcal{C}_V$ bzw. $\mathcal{C}_{V\cap U}$). Um zu beweisen, dass $\mathcal{C}_X$ kompakt erzeugt ist, genügt es zu zeigen, dass es für jedes $M\in \mathcal{C}_X$ und jedes kompakte $N\in \mathcal{C}_U$ (bzw. $N\in \mathcal{C}_V$) zusammen mit einem Morphismus $f:N\rightarrow M|_{U}$ (bzw. $f:N\rightarrow M|_{V}$) ein kompaktes Objekt $\Tilde{N}\in\mathcal{C}_X$ und einen Morphismus $\tilde{f}:\Tilde{N}\rightarrow M$ gibt, sodass $\tilde{f}|_U$ (bzw. $\tilde{f}|_V$) wie folgt faktorisiert:
\[\tilde{N}|_U\xrightarrow{\sim} N\oplus L\xrightarrow{\pi_{N}}N\xrightarrow{f}M|_U\ (\text{bzw}.\ \tilde{N}|_V\xrightarrow{\sim} N\oplus L\xrightarrow{\pi_{N}}N\xrightarrow{f}M|_V),\]
wobei $L$ ein kompaktes Objekt in $\mathcal{C}_U$ (bzw. $\mathcal{C}_V$) ist. Wir behandeln nur den Fall von $U$, denn der Fall von $V$ ist völlig analog. Man betrachte den Morphismus $f:N|_{V\cap U}\rightarrow M|_{V\cap U}$. Wie man leicht nachprüft, ist der Funktor $\mathcal{C}_V\rightarrow \mathcal{C}_{U\cap V}$ eine kompakte linke Bousfield-Lokalisierung. Deswegen gibt es – nach Satz~\ref{allgemeiner_thomason-trick}(ii) – ein kompaktes Objekt $\hat{N}\in\mathcal{C}_{V}$ und einen Morphismus $\hat{f}:\hat{N}\rightarrow M|_{V}$, sodass $\hat{f}|_{U\cap V}$ wie folgt faktorisiert:
\[\hat{N}|_{U\cap V}\xrightarrow{\sim} N|_{U\cap V}\oplus N'\xrightarrow{\pi_{N}|_{U\cap V}}N\xrightarrow{f|_{U\cap V}}M|_{U\cap V},\]
wobei $N'$ ein kompaktes Objekt in $\mathcal{C}_{U\cap V}$ ist. Sei jetzt $N''$ eine Hochhebung von $N'\oplus N'[1]$ nach $U$, die nach Satz~\ref{allgemeiner_thomason-trick}(iii) existiert. Man betrachte folgende zwei Morphismen: \[\hat{N}\oplus  \hat{N}[1]\xrightarrow{\pi_{\hat{N}}}\hat{N}\xrightarrow{\hat{f}} M|_{V}\] in $C_{V}$ und \[N \oplus  N[1]\oplus N''\xrightarrow{\pi_{N}}N\xrightarrow{f} M|_{U}\] in $C_{U}$. Man überzeugt sich leicht davon, dass sie auf $U\cap V$ kompatibel sind, daher verkleben sie sich zu einem Paar ($\tilde{N}$, $\tilde{f}$) mit den gewünschten Eigenschaften.

Wir behandeln nun den dualisierbaren Fall. Wir erinnern zunächst an den folgenden Satz; für den Beweis verweisen wir auf \cite[Proposition I.3.5]{TCH}.

\begin{satz}[Thomason-Neeman]
Sei $\mathcal{C}'\xrightarrow{f} \mathcal{C} \xrightarrow{g} \mathcal{C}''$ eine Sequenz in $\mathrm{Cat}_{\infty}^{\mathrm{perf}}$ mit $g\circ f=0$. Dann ist sie eine Verdier-Sequenz genau dann, wenn $\Ind \mathcal{C}'\xrightarrow{\Ind f} \Ind\mathcal{C} \xrightarrow{\Ind g} \Ind\mathcal{C}''$ eine Verdier-Sequenz in $\mathcal{P}r^{\mathrm{St}}$ ist.
\end{satz}

Sei $\mathcal{C}$ jetzt eine dualisierbare $\infty$-Kategorie. Es bezeichne $\hat{y}:\mathcal{C}\rightarrow \Ind \mathcal{C}$ den linksadjungierten Funktor zu $\colim: \Ind \mathcal{C}\rightarrow \mathcal{C}$. Man prüft leicht nach, dass es eine hinreichend große reguläre Kardinalzahl $\kappa$ gibt, sodass der volltreue Funktor $\hat{y}:\mathcal{C}\rightarrow \Ind \mathcal{C}$ über $\Ind (\mathcal{C}^\kappa) \rightarrow \Ind \mathcal{C}$ faktorisiert, wobei $\mathcal{C}^\kappa$ die volle $\infty$-Unterkategorie der $\kappa$-kompakten Objekte\footnote{siehe \cite[Definition 5.3.4.5]{HTT}} in $\mathcal{C}$ bezeichnet. Die volle $\infty$-Unterkategorie $\mathcal{C}^\kappa$ ist nach \cite[Bemerkung 5.4.2.13]{HTT} klein, daher ist die $\infty$-Kategorie $\Ind (\mathcal{C}^\kappa)$ kompakt erzeugt. Für den Rest des Beweises wählen wir implizit eine hinrechend große Kardinalzahl $\kappa$, sodass alle betrachteten dualisierbaren $\infty$-Kategorien eine solche Faktorisierung besitzen, und schreiben für $\Ind (\mathcal{C}^\kappa)$ einfach $\Ind \mathcal{C}$. 

Wir betrachten folgende Diagramme: 

\begin{center}
\begin{tikzcd}
\mathcal{C}_X \arrow[r,"F_U'"] \arrow[d,"F_V'"] & \mathcal{C}_U \arrow[d,"F_U"] \\
\mathcal{C}_V \arrow[r,"F_V"] & \mathcal{C}_{U\cap V}
\end{tikzcd}\
\begin{tikzcd}
\Ind \mathcal{C}_V \underset{\Ind \mathcal{C}_{U\cap V}}{\times} \Ind \mathcal{C}_U \arrow[r,"\pi_U"] \arrow[d,"\pi_V"] & \Ind \mathcal{C}_U \arrow[d,"\Ind F_U"] \\
\Ind \mathcal{C}_V \arrow[r,"\Ind F_V"] & \Ind \mathcal{C}_{U\cap V},
\end{tikzcd}\
\end{center}
wobei $F_U,F_V,F_U',F_V'$ die Lokalisierungsfunktoren bezeichnen und $G_U,G_V,G_U',G_V'$ deren rechtsadjungierten Funktoren. Nach dem Satz von Thomason-Neeman ist die Faser von $\Ind F_U$ (bzw. $\Ind F_V$) äquivalent zu $\Ind (\mathrm{fib}F_U)$ (bzw. $\Ind (\mathrm{fib}F_V)$). Das Faserprodukt $\Ind \mathcal{C}_V \underset{\Ind \mathcal{C}_{U\cap V}}{\times} \Ind \mathcal{C}_U$ ist also nach der ersten Hälfte des Beweises kompakt erzeugt. Deshalb genügt es zu zeigen, dass $C_X$ ein Retrakt von $\Ind \mathcal{C}_V \underset{\Ind \mathcal{C}_{U\cap V}}{\times} \Ind \mathcal{C}_U$ ist.

Man betrachte folgende Paare von adjungierten Funktoren:
\[
\begin{tikzcd}
           \mathcal{C}_U \arrow[r, shift left=1ex, "L_U"{name=G}] & \Ind \mathcal{C}_U \arrow[l, shift left=.5ex, "\colim_U"{name=F}]
            \arrow[phantom, from=F, to=G, , "\scriptscriptstyle\boldsymbol{\bot}"],
        \end{tikzcd}\begin{tikzcd}
             \mathcal{C}_V \arrow[r, shift left=1ex, "L_V"{name=G}] & \Ind \mathcal{C}_V \arrow[l, shift left=.5ex, "\colim_V"{name=F}]
            \arrow[phantom, from=F, to=G, , "\scriptscriptstyle\boldsymbol{\bot}"],
        \end{tikzcd}\begin{tikzcd}
            \mathcal{C}_{U\cap V} \arrow[r, shift left=1ex, "L_{U\cap V}"{name=G}] & \Ind \mathcal{C}_{U\cap V} \arrow[l, shift left=.5ex, "\colim_{U\cap V}"{name=F}]
            \arrow[phantom, from=F, to=G, , "\scriptscriptstyle\boldsymbol{\bot}"],
        \end{tikzcd} \]
Wir behaupten, dass folgende Diagramme kommutativ sind:
\begin{center}
\begin{tikzcd}[row sep=large, column sep=large]
\Ind \mathcal{C}_V \underset{\Ind \mathcal{C}_{U\cap V}}{\times} \Ind \mathcal{C}_U \arrow[r,"\pi_U"] \arrow[d,"\pi_V"] & \Ind \mathcal{C}_U \arrow[d,"\Ind F_U"] \arrow[dr, "\colim_U"] \\

\Ind \mathcal{C}_V \arrow[r,"\Ind F_V"] \arrow[dr,"\colim_{V}"] & \Ind \mathcal{C}_{U\cap V} \arrow[dr,"\colim_{U\cap V}"] & \mathcal{C}_{U} \arrow[d,"F_{U}"] \\

& \mathcal{C}_{V}\arrow[r,"F_{V}"] & \mathcal{C}_{U\cap V}
\end{tikzcd}\
\begin{tikzcd}[row sep=large, column sep=large]
\mathcal{C}_X \arrow[r,"F_U'"] \arrow[d,"F_V'"] & \mathcal{C}_U \arrow[d,"F_U"]\arrow[dr, "L_U"] \\
\mathcal{C}_V \arrow[r,"F_V"] \arrow[dr,"L_V"]& \mathcal{C}_{U\cap V} \arrow[dr,"L_{U\cap V}"] & \Ind \mathcal{C}_{U}\arrow[d, "\Ind F_U"]\\
& \Ind \mathcal{C}_{V}\arrow[r,"\Ind F_{V}"] & \Ind \mathcal{C}_{U\cap V}
\end{tikzcd}\
\end{center}
Wir betrachten zunächst die Diagramme
\begin{center}
\begin{tikzcd}
\mathcal{C}_V \arrow[r,"F_V"] \arrow[d,"L_V"]& \mathcal{C}_{U\cap V} \arrow[d,"L_{U\cap V}"]\\
\Ind \mathcal{C}_{V}\arrow[r,"\Ind F_{V}", swap] & \Ind \mathcal{C}_{U\cap V}
\end{tikzcd}
\begin{tikzcd}
\mathcal{C}_U \arrow[d,"F_U"] \arrow[r,"L_U"] & \Ind \mathcal{C}_{U} \arrow[d,"\Ind F_U"] \\
\mathcal{C}_{U\cap V} \arrow[r,"L_{U\cap V}", swap] & \Ind \mathcal{C}_{U\cap V}
\end{tikzcd}
\end{center} 
Da alle Funktoren in den beiden Diagrammen linksadjungiert sind, genügt es nach Lemma~\ref{ind_adj} zu zeigen, dass die Diagramme
\begin{center}
\begin{tikzcd}
\mathcal{C}_V & \mathcal{C}_{U\cap V} \arrow[l,"G_V", swap] \\
\Ind \mathcal{C}_{V}  \arrow[u,"\colim_V", swap] &  \Ind \mathcal{C}_{U\cap V} \arrow[l,"\Ind G_{V}"] \arrow[u,"\colim_{U\cap V}", swap]
\end{tikzcd}
\begin{tikzcd}
\mathcal{C}_U & \Ind \mathcal{C}_{U} \arrow[l,"\colim_U", swap] \\
\mathcal{C}_{U\cap V}  \arrow[u,"G_U", swap] &  \Ind \mathcal{C}_{U\cap V} \arrow[l,"\colim_{U\cap V}"] \arrow[u,"\Ind G_{U}", swap]
\end{tikzcd}
\end{center}
kommutativ sind. Es bezeichne $y_U$ (bzw. $y_V$ bzw. $y_{U\cap V}$) die Yoneda-Einbettung $\mathcal{C}_U\rightarrow \Ind \mathcal{C}_U$ (bzw. $\mathcal{C}_V\rightarrow \Ind \mathcal{C}_V$ bzw. $\mathcal{C}_{U\cap V}\rightarrow \Ind \mathcal{C}_{U\cap V}$). Man prüft direkt nach, dass 
\[\colim_V \circ \Ind G_V \circ y_{U\cap V}\cong G_V \circ \colim_{U\cap V} \circ y_{U\cap V}\ \mathrm{und} \  \colim_U \circ \Ind G_U \circ y_{U\cap V}\cong G_U \circ \colim_{U\cap V} \circ y_{U\cap V} \]
gilt. Die gewünschte Kommutativität ergibt sich nun aus der universellen Eigenschaft von $\Ind(-)$. Die Kommutativität der Diagramme

\begin{center}
\begin{tikzcd}
\Ind \mathcal{C}_V \arrow[r,"\Ind F_V"] \arrow[d,"\colim_V"]& \Ind \mathcal{C}_{U\cap V} \arrow[d,"\colim_{U\cap V}"]\\
\mathcal{C}_{V}\arrow[r,"F_{V}"] & \mathcal{C}_{U\cap V}
\end{tikzcd}
\begin{tikzcd}
\Ind \mathcal{C}_U \arrow[d,"\Ind F_U"] \arrow[r,"\colim_U"] & \mathcal{C}_{U} \arrow[d,"F_U"] \\
\Ind \mathcal{C}_{U\cap V} \arrow[r,"\colim_{U\cap V}"] & \Ind \mathcal{C}_{U\cap V}
\end{tikzcd}
\end{center}
wird völlig analog bewiesen.

Die universelle Eigenschaft des Faserprodukts liefert also Funktoren \[\Phi: \mathcal{C}_X \rightarrow \Ind \mathcal{C}_V \underset{\Ind \mathcal{C}_{U\cap V}}{\times} \Ind \mathcal{C}_U\ \mathrm{und} \ \Psi: \Ind \mathcal{C}_V \underset{\Ind \mathcal{C}_{U\cap V}}{\times} \Ind \mathcal{C}_U \rightarrow \mathcal{C}_X.\] Es gilt außerdem $\Psi\circ \Phi \cong \mathrm{Id}_{\mathcal{C}_X}$, was man wiederum unter Benutzung der universellen Eigenschaft des Faserprodukts nachweist. Man überzeugt sich leicht davon, dass die Funktoren $\Psi$ und $\Phi$ filtrierte Kolimites erhalten, weswegen die $\infty$-Kategorie $\mathcal{C}_X$ ein Retrakt des Faserprodukts $\Ind \mathcal{C}_V \underset{\Ind \mathcal{C}_{U\cap V}}{\times} \Ind \mathcal{C}_U$ in der $\infty$-Kategorie $\mathcal{P}r^{\mathrm{St}}$ ist.
\end{proof}

\begin{lemma}\label{ind_adj}
Seien   \[
        \begin{tikzcd}
            \mathcal{C} \arrow[r, shift left=1ex, "F"{name=G}] & \mathcal{D}\arrow[l, shift left=.5ex, "G"{name=F}]
            \arrow[phantom, from=F, to=G, , "\scriptscriptstyle\boldsymbol{\bot}"]
        \end{tikzcd}
    \]
    
adjungierte Funktoren zwischen kleinen $\infty$-Kategorien. Dann sind die Funktoren
 \[
        \begin{tikzcd}
            \Ind \mathcal{C} \arrow[r, shift left=1ex, "\Ind F"{name=G}] & \Ind \mathcal{D}\arrow[l, shift left=.5ex, "\Ind G"{name=F}]
            \arrow[phantom, from=F, to=G, , "\scriptscriptstyle\boldsymbol{\bot}"]
        \end{tikzcd}
    \]
adjungiert. Ist $G$ hierbei volltreu, so ist $\Ind G$ ebenso volltreu.
\end{lemma}

\begin{proof}
Es bezeichne $\mathrm{Cat}_{\infty}$ (bzw. $\widehat{\mathrm{Cat}}_{\infty}$) die $(\infty,2)$-Kategorie der kleinen $\infty$-Kategorien (bzw. die $(\infty,2)$-Kategorie aller $\infty$-Kategorien). Dann folgt die gewünschte Aussage ganz formal aus der Tatsache, dass der Funktor $\Ind: \mathrm{Cat}_{\infty}\rightarrow \widehat{\mathrm{Cat}}_{\infty}$ ein Funktor von $(\infty,2)$-Kategorien ist.
\end{proof}

\newpage
\section{Nukleare Moduln}

Im modernen Ansatz zur $K$-Theorie der Schemata definiert man diese durch Anwenden einer rein kategoriellen Konstruktion auf die derivierte Kategorie der quasi-kohärenten Garben. Um in unserer Situation analog vorgehen zu können, benötigen wir dementsprechend eine solche Kategorie für analytische adische Räume. Im nächsten Abschnitt führen wir die $\infty$-Kategorie der \textit{nuklearen Garben} ein und untersuchen ihre Eigenschaften. Das Ziel des vorliegenden Abschnitts ist es, diese Diskussion vorzubereiten. Konkret behandeln wir hier den affinoiden Fall, indem wir kurz an die Definition der $\infty$-Kategorie der \textit{nuklearen Moduln} über einem Huber-Ring erinnern und ihre Eigenschaften analysieren. Unser Hauptresultat ist der Beweis der Tatsache, dass diese $\infty$-Kategorie dualisierbar ist, siehe Satz~\ref{dualisierbarkeit_moduln}. Im Folgenden setzen wir Vertrautheit mit den Grundlagen der verdichteten Mathematik voraus. Für eine kurze Zusammenfassung der benötigten Definitionen und Sätze verweisen wir auf \cite[Abschnitt 2]{firstpaper}, mehr dazu findet man in \cite{Condensed} und \cite{Analytic}.

Der Hauptbestandteil unserer Argumentation ist der technische Begriff der \textit{schwachen Proregularität} für Ideale eines Ringes, den man sich als das Analogon der noetherschen Bedingung für nicht noethersche Ringe vorstellen kann. Wir erinnern kurz an die Definition und ihre wichtigsten Eigenschaften, aber wir wollen hier auf weitere Details verzichten und verweisen stattdessen auf \cite{WPR}.

\begin{definition}[{\cite[Abschnitt 2 und Definitionen 3.1 und 3.2]{WPR}}]
Sei $A$ ein (kommutativer) Ring und $\bm{a}=(a_1,\cdots,a_n)$ eine endliche Folge von Elementen von $A$. Für $i\in \mathbb{N}$ bezeichne $\bm{a}^i$ die Folge $(a_1^i,\dots,a_n^i)$.
\begin{enumerate}[label=(\roman*)]
\item Sei $a$ ein Element von $A$. Unter dem zu $a$ assoziierten \textit{Koszul-Komplex} verstehen wir den Komplex von $A$-Moduln
\[K(A;a)=(\dots\rightarrow 0 \rightarrow A\xrightarrow[]{\cdot a} A\rightarrow 0 \rightarrow \dots),\]
welcher in den homologischen Graden $0$ und $1$ konzentriert ist. Für jedes Paar von natürlichen Zahlen $i,j\in\mathbb{N}$ mit $j\geq i$ bezeichnen wir mit $\mu_{ji}$ die Abbildung $K(A;a^j)\rightarrow K(A;a^i)$ von Komplexen über $A$, die durch den Identitätsmorphismus im Grad $0$ und die Multiplikation mit $a^{j-i}$ im Grad $1$ gegeben ist.
\item Unter dem zu $\bm{a}$ assoziierten \textit{Koszul-Komplex} von $A$-Moduln verstehen wir das Tensor-Produkt $K(A;\bm{a})=K(A;a_1)\underset{A}{\otimes}\cdots \underset{A}{\otimes} K(A;a_n)$. Die oben definierten Morphismen $\mu_{ji}$ induzieren in offensichtlicher Weise Morphismen $K(A;\bm{a}^j)\rightarrow K(A;\bm{a}^i)$, die wir ebenfalls mit $\mu_{ji}$ bezeichnen.
\item Die Folge $\bm{a}$ heißt \textit{schwach proregulär}, falls für jedes $q>0$ das inverse System $\{H_q(K(A;\bm{a}^i)) \}_{i\in \mathbb{N}}$ von $A$-Moduln pro-trivial ist, d. h., falls für jedes $i\geq 0$ es ein $j\geq i$ gibt, sodass die Abbildung $H_q(\mu_{ji}):H_q(K(A;\bm{a}^j))\rightarrow H_q(K(A;\bm{a}^i))$ trivial ist.
\item Ein Ideal $I\subset A$ heißt \textit{schwach proregulär}, falls es eine endliche schwach proreguläre Folge gibt, dessen Elemente $I$ erzeugen.
\end{enumerate}
\end{definition}

\begin{satz}[{\cite[Theorem 3.3]{WPR}}]
Sei $A$ ein noetherscher Ring. Dann ist jedes Ideal von $A$ schwach proregulär.
\end{satz}

Die schwache Proregularität eines Ideals ist tatsächlich unabhängig von der Wahl der erzeugenden Folge, wie das folgende Lemma zeigt.

\begin{lemma}[{\cite[Korollar 3.5]{WPR}}]
Sei $A$ ein Ring und $I\subset A$ ein Ideal. Ist $I$ schwach proregulär, so ist jede endliche Folge, die das Ideal $I$ erzeugt, schwach proregulär.
\end{lemma}

Für unsere Analyse der $\infty$-Kategorie der nuklearen Moduln benötigen wir die folgende Version des Begriffes der schwachen Proregularität für Huber-Ringe.

\begin{definition}
Ein vollständiger Huber-Ring $A$ heißt \textit{schwach proregulär}, wenn er ein Definitionspaar $(A_0,I)$ mit $I$ schwach proregulär in $A_0$ besitzt.
\end{definition}

Genauso wie für Ideale ist die Proregularität eines Huber-Ringes von der Wahl des Definitionspaars unabhängig, wie das folgende Lemma zeigt.

\begin{lemma}\label{nukleare_Huber-Ringe}
Sei $A$ ein vollständiger schwach pro-regulärer Huber-Ring und $(A_0,I)$ ein beliebiges Definitionspaar von $A$. Dann ist das Ideal $I$ schwach proregulär in $A_0$.
\end{lemma}

\begin{proof}
Da das Produkt zweier Definitionsringe ebenfalls ein Definitionsring ist, genügt es folgende zwei Aussagen zu beweisen.
\begin{enumerate}[label=(\roman*)]
\item Angenommen, das Ideal $I\subset A_0$ ist schwach proregulär. Dann ist jedes Definitionsideal $J\subset A_0$ ebenso schwach proregulär.
\item Sei $A_0'$ ein Definitionsring von $A$ mit $A_0\subset A_0'$. Dann ist das Ideal $I$ genau dann schwach proregulär, wenn das Ideal $I\cdot A_0'$ schwach proregulär ist.
\end{enumerate}

Die erste Aussage ergibt sich als direkte Folgerung aus \cite[Theorem 3.4]{WPR}, da das Radikal jedes Definitionsideals von $A_0$ gleich dem Durchschnitt $A_0\cap A^{\circ \circ}$ ist. Wir beweisen nun die zweite Aussage. Es bezeichne $\bm{a}$ eine endliche Folge von Elementen von $A_0$, die das Ideal $I$ erzeugt. Die Koszul-Komplexe $K(A_0;\bm{a}^i)$ sind in offensichtlicher Weise Unterkomplexe der Koszul-Komplexe $K(A_0';\bm{a}^i)$. Man prüft nun direkt nach, dass für $q>0$ das inverse System $\{H_q(K(A_0;\bm{a}^i)) \}_{i\in \mathbb{N}}$ genau dann pro-trivial ist, wenn das System $\{H_q(K(A_0';\bm{a}^i)) \}_{i\in \mathbb{N}}$ pro-trivial ist, denn es gibt ein $k\geq 0$, sodass $I^k\cdot A_0'\subset A_0$ gilt.
\end{proof}

Die Hauptanwendung des Begriffes der schwachen Proregularität besteht darin, dass man mit seiner Hilfe die Beziehung zwischen verschiedenen Vollständigkeitsbegriffen für Komplexe von $A$-Moduln bestimmen kann. Obwohl wir hier auf eine vollständige Diskussion der schwachen Proregularität verzichten, kann man diesen Abschnitt trotzdem ohne Beeinträchtigung des Verständnisses lesen: Im Weiteren brauchen wir nur eine ganz konkrete Aussage, für welche wir an der entsprechenden Stelle eine Referenz geben. 

Im Folgenden sei $(A,A^+)$ ein festes vollständiges schwach proreguläres Huber-Paar und $(A_0,I)$ dessen festes Definitionspaar. Sei $(f_1,\dots,f_n)$ ein endliches System von Elementen von $A_0$, welches das Ideal $I$ erzeugt. Es bezeichne $(R,\mathfrak{m})$ den Ring $\mathbb{Z}[[x_1,\dots,x_n]]$ zusammen mit dem Ideal $(x_1,\dots,x_n)\subset R$. In diesem Abschnitt versehen wir $R$ mit der $\mathfrak{m}$-adischen Topologie und bezeichnen den zugehörigen analytischen verdichteten Ring mit $R_\blacksquare$, siehe \cite[Theorem 3.28]{firstpaper} und die dem Theorem vorangehende Diskussion. Wir versehen $A_0$, und damit auch $A$, mit der Struktur eines Moduls über $R$, welche durch die Abbildung $R\rightarrow A_0,\ x_i\mapsto f_i$ für $i=1,\dots,n$ gegeben ist. Die Notation $(\mathcal{A},\mathcal{M})$, wenn sie ohne anderslautenden Kommentar auftritt, bezeichnet stets den analytischen verdichteten Ring $(A,A^+)_{\blacksquare}$, siehe wieder \cite[Theorem 3.28]{firstpaper}. Wir verwenden den Unterstrich, um die verdichteten Moduln zu bezeichnen, welche zu üblichen topologischen Moduln assoziiert sind; wenn der betrachtete Modul $M$ jedoch diskret ist, verzichten wir auf den Unterstrich und bezeichnen den zugehörigen verdichteten Modul ebenfalls mit $M$.  

Wir erinnern kurz an die allgemeine Definition der Nuklearität für analytische verdichtete Ringe und ihre äquivalente Beschreibung im Kontext der adischen Geometrie; bezüglich Details konsultiere man \cite[Vorlesung 13]{Analytic} und \cite[Abschnitt 5.3]{firstpaper}.

\begin{definition}[{\cite[Definition 13.10]{Analytic}}]
Sei $(\mathcal{A},\mathcal{M})$ ein beliebiger analytischer animierter verdichteter Ring. Ein Objekt $C\in\mathcal{D}(\mathcal{A},\mathcal{M})$ heißt \textit{nuklear}, wenn für jeden extremal unzusammenhängenden Raum $S$ die natürliche Abbildung $(\mathcal{M}[S]^\vee\underset{(\mathcal{A},\mathcal{M})}{\otimes}C)(\ast)\rightarrow C(S)$ in $\mathcal{D}(\mathrm{Ab})$ ein Isomorphismus ist, wobei $\mathcal{M}[S]^\vee$ das derivierte interne Dual $\mathrm{R}\underline{\mathrm{Hom}}_{(\mathcal{A},\mathcal{M})}(\mathcal{M}[S],\mathcal{A})$ bezeichnet.
\end{definition}

\begin{lemma}[{\cite[Proposition 5.35]{firstpaper}}]
Sei $(A,A^+)$ ein beliebiges vollständiges Huber-Paar. Es bezeichne $(\mathcal{A},\mathcal{M})$ den analytischen verdichteten Ring $(A,A^+)_\blacksquare$. Ein Objekt $C\in \mathcal{D}(\mathcal{A},\mathcal{M})$ ist nuklear genau dann, wenn für jede proendliche Menge $S$ die natürliche Abbildung

\[\underline{C(S,A)}\underset{(\mathcal{A},\mathcal{M})}{\overset{L}{\otimes}} C \xrightarrow[]{} \mathrm{R}\underline{\mathrm{Hom}}_{(\mathcal{A},\mathcal{M})}(\mathcal{M}[S],C) \]
ein Isomorphismus ist, wobei $C(S,A)$ den Modul der stetigen Funktionen auf $S$ mit Werten in $A$ bezeichnet.
\end{lemma}

Wir beginnen nun mit der Analyse der Nuklearitätsbedingung für Moduln über dem analytischen verdichteten Ring $(\mathcal{A},\mathcal{M})$.  

\begin{lemma}
\label{grundnuklearität}
Ist $S$ eine proendliche Menge, so ist der Modul $\underline{C(S,R)}$ der stetigen Funktionen auf $S$ mit Werten in $R$ nuklear über $R_{\blacksquare}$.
\end{lemma}

\begin{proof}
Wir führen zunächst folgende elementare Rechnung durch:
\[\underline{C(S,R)}\cong \varprojlim C(S,R/\mathfrak{m}^n)\cong \varprojlim \underset{J}{\bigoplus} R/\mathfrak{m}^n\cong \underset{\substack{\tilde{J}\subset J \\ \textnormal{abz\"ahlbar}}}{\colim} \underline{R\langle T \rangle}.\]
Hierbei ist $T$ eine formale Variable und $R\langle T\rangle$ die Tate-Algebra. Bei dieser Rechnung haben wir die Tatsache benutzt, dass der Modul der stetigen Funktionen auf $S$ mit Werten in einem diskreten Modul frei ist, siehe \cite[Theorem 5.4]{Condensed}. Da die $\infty$-Kategorie der nuklearen Moduln unter Kolimites abgeschlossen ist, genügt es zu zeigen, dass der Modul $\underline{R\langle T \rangle}$ nuklear ist. Sei $S'$ eine proendliche Menge. Dann gilt
\[R_\blacksquare[S']^\vee \underset{R_{\blacksquare}}{\overset{L}{\otimes}} \underline{R\langle T \rangle} \cong \underline{C(S',R)} \underset{R_{\blacksquare}}{\overset{L}{\otimes}} \underline{R\langle T \rangle} \cong\underset{\substack{\tilde{J}\subset J \\ \textnormal{abz\"ahlbar}}}{\colim} ( \underline{R\langle T' \rangle} \underset{R_{\blacksquare}}{\overset{L}{\otimes}} \underline{R\langle T \rangle}),\]
\[\underline{C(S',R\langle T \rangle )}\cong \varprojlim C(S,R/\mathfrak{m}^n[T])\cong \varprojlim \underset{J}{\bigoplus} R/\mathfrak{m}^n[T]\cong \underset{\substack{\tilde{J}\subset J \\ \textnormal{abz\"ahlbar}}}{\colim} \underline{R\langle T, T' \rangle},\]
wobei $T'$ ebenfalls eine formale Variable ist. Daher genügt es zu zeigen, dass das feste Tensorprodukt $\underline{R\langle T' \rangle} \underset{R_{\blacksquare}}{\overset{L}{\otimes}} \underline{R\langle T \rangle}$ zur Tate-Algebra $\underline{R\langle T, T' \rangle}$ isomorph ist. 

Für eine natürliche Zahl $k\geq 0$ bezeichne $\mathfrak{m}_{k}$ das von $x_1^{k},\dots, x_n^{k}$ erzeugte Ideal von $R$. Wie man leicht nachprüft, lässt sich die Tate-Algebra $\underline{R\langle T \rangle}$ als Kolimes \[\underset{\substack{f:\mathbb{N}\rightarrow \mathbb{N} \\ f(n)\rightarrow +\infty}}{\colim} \underset{i\in\mathbb{N}}{\prod}\underline{\mathfrak{m}}_{f(i)} T^i\] in der Kategorie der verdichteten Moduln schreiben, wobei $f:\mathbb{N}\rightarrow \mathbb{N}$ die Menge der nicht negativen gegen $+\infty$ konvergenten Folgen durchläuft. Wir behaupten nun, dass jedes $\underset{i\in\mathbb{N}}{\prod}\underline{\mathfrak{m}}_{f(i)}$ kompakt in $\mathcal{D}(R_\blacksquare)$ ist. In der Tat ist $\mathfrak{m}_{f(i)}$ isomorph zum Komplex $\mathrm{fib}\, (R\rightarrow K(R;x_1^{f(i)},\dots, x_n^{f(i)}))$, wobei $K(R;x_1^{f(i)},\dots, x_n^{f(i)})$ den zur Folge $(x_1^{f(i)},\dots, x_n^{f(i)})$ assoziierten Koszul-Komplex bezeichnet. Daraus folgt sofort, dass der verdichtete Modul $\underset{i\in\mathbb{N}}{\prod}\underline{\mathfrak{m}}_{f(i)}$ quasi-isomorph zu einem endlichen Komplex ist, dessen Terme alle der Form ${\prod} \underline{R}$ sind. Da für jedes Paar von Mengen $I,I'$ das Tensorprodukt $\underset{I}{\prod}\underline{R}  \underset{R_{\blacksquare}}{\overset{L}{\otimes}} \underset{J}{\prod}\underline{R}$ isomorph zum Produkt $\underset{I\times J}{\prod}\underline{R}$ ist, gilt
\[\underset{i\in\mathbb{N}}{\prod}\underline{\mathfrak{m}}_{f(i)}  \underset{R_{\blacksquare}}{\overset{L}{\otimes}} \underset{j\in\mathbb{N}}{\prod}\underline{\mathfrak{m}}_{g(j)}\cong \underset{i,j\in\mathbb{N}}{\prod}(\underline{\mathfrak{m}}_{f(i)}  \underset{R_{\blacksquare}}{\overset{L}{\otimes}} \underline{\mathfrak{m}}_{g(j)}),\] 
wobei $g:\mathbb{N}\rightarrow \mathbb{N}$ eine nicht negative gegen $+\infty$ konvergente Folge ist. Es bezeichne $\mathfrak{m}_{f(i),g(j)}$ das von $\{x_k^{f(i)}x_l^{g(j)}\}_{k,l=1}^n$ erzeugte Ideal von $R$. Man prüft direkt nach, etwa unter Benutzung der Koszul-Auflösung, dass das Tensorprodukt $\underline{\mathfrak{m}}_{f(i)}  \underset{R_{\blacksquare}}{\overset{L}{\otimes}} \underline{\mathfrak{m}}_{g(j)}$ isomorph zu $\underline{\mathfrak{m}}_{f(i),g(j)}$ ist, denn die Folge $(x_1^{f(i)},\dots, x_n^{f(i)})$ ist $\mathfrak{m}_{g(j)}$-regulär. Man verifiziert dann ohne Schwierigkeiten, dass die Tate-Algebra $\underline{R\langle T,T'\rangle}$ sich als Kolimes 
\[\underset{\substack{f,g:\mathbb{N}\rightarrow \mathbb{N} \\ f(n),g(n)\rightarrow +\infty}}{\colim} \underset{i,j\in\mathbb{N}}{\prod}\underline{\mathfrak{m}}_{f(i),g(j)} T^i\cdot (T')^j\] 
schreiben lässt.
\end{proof}

Wir führen nun eine Version des \textit{(idealistischen) derivierten Vervollständigungsfunktors} ein. Für eine Diskussion des klassischen, nicht verdichteten Falls verweisen wir auf \cite[Abschnitt 1]{WPR}.

\begin{definition}
Sei $\mathcal{D}(R)$ die algebraische derivierte $\infty$-Kategorie des Ringes $R$. Der \textit{derivierte Vervollständigungsfunktor} $\mathbb{L}\Lambda_\mathfrak{m}: \mathcal{D}_{\geq 0}(R)\rightarrow \mathcal{D}_{\geq 0}(R_\blacksquare)$ ist gegeben durch 
\[(\cdots \rightarrow\underset{J}{\bigoplus}R\rightarrow \underset{J'}{\bigoplus}R\rightarrow 0 \rightarrow \cdots)\mapsto (\cdots \rightarrow\varprojlim\underset{J}{\bigoplus}R/\mathfrak{m}^n\rightarrow \varprojlim\underset{J'}{\bigoplus}R/\mathfrak{m}^n\rightarrow 0 \rightarrow \cdots).\]
\end{definition}
Man bemerke dabei, dass für jede Indexmenge $J$ der verdichtete Modul $\varprojlim\underset{J}{\bigoplus}R/\mathfrak{m}^n$ isomorph zum Modul $\underset{\substack{\tilde{J}\subset J \\ \textnormal{abz\"ahlbar}}}{\colim} \underline{R\langle T \rangle}$ ist. Daher ergibt sich als Konsequenz des obigen Lemmas der folgende Satz:

\begin{satz}\label{monoidal}
Der Funktor $\mathbb{L}\Lambda_\mathfrak{m}:\mathcal{D}_{\geq 0}(R)\rightarrow \mathcal{D}_{\geq 0}(R_\blacksquare)$ ist monoidal.
\end{satz}

Es bezeichne $R_{\mathrm{disk}}$ (bzw. $(A_0)_{\mathrm{disk}}$) den Ring $R$ (bzw. $A_0$) versehen mit der diskreten Topologie. Wir beweisen zunächst das folgende Lemma. 

\begin{lemma}
\label{vervollständigung}
Ist $S$ eine proendliche Menge, so gilt $\mathbb{L}\Lambda_\mathfrak{m}(C(S,(A_0)_{\mathrm{disk}}))\cong \underline{C(S,A_0)}$. Insbesondere ist der verdichtete Modul $\mathbb{L}\Lambda_\mathfrak{m}((A_0)_{\mathrm{disk}})$ isomorph zum Modul $\underline{A_0}$. 
\end{lemma}

\begin{proof}
Der Modul $C(S,(A_0)_{\mathrm{disk}})$ ist nach \cite[Theorem 5.4]{Condensed} isomorph zu einer direkten Summe $\underset{J}{\bigoplus} (A_0)_{\mathrm{disk}}$.
Sei $P_{\bullet}$ eine freie Auflösung von $(A_0)_{\mathrm{disk}}$ (in der Kategorie $\mathrm{Mod}_{R}$). Wir führen folgende Rechnung durch:
\[\mathbb{L}\Lambda_\mathfrak{m}(C(S,(A_0)_{\mathrm{disk}}))\cong \varprojlim\underset{J}{\bigoplus}P_{\bullet}/\mathfrak{m}^nP_{\bullet}\cong  \underset{\substack{\tilde{J}\subset J \\ \textnormal{abz\"ahlbar}}}{\colim} \varprojlim(P_{\bullet}/\mathfrak{m}^nP_{\bullet} [T]) \overset{\dagger}{\cong} \underset{\substack{\tilde{J}\subset J \\ \textnormal{abz\"ahlbar}}}{\colim} \underline{A_0\langle T \rangle}  \cong \underline{C(S,A_0)}.\]
Wir müssen die Isomorphie $(\dagger)$ nachweisen. Da der Ring $A_0$ klassisch $I$-vollständig ist, ist er auch deriviert $I$-vollständig. Die Sequenz
\[ \cdots \xrightarrow{\phi_2} \varprojlim P_1/\mathfrak{m}^nP_1 \xrightarrow{\phi_1} \varprojlim P_0/\mathfrak{m}^nP_0\xrightarrow{\phi_0} A_0\langle T \rangle\rightarrow 0\]
ist nach \cite[Theorem 3.11]{WPR} exakt in der Kategorie $\mathrm{Mod}_{R}$, denn die Ideale $\mathfrak{m}\subset R$ und $I\subset A_0$ sind schwach proregulär. Es bleibt zu zeigen, dass sie auch in der Kategorie der verdichteten Moduln über $\underline{R}$ exakt ist. Aus der Definition und der Exaktheit der oberen Sequenz folgt unmittelbar, dass der Kern von $\phi_i$ zur $I$-adischen Vervollständigung des Kerns der Abbildung $P_i\rightarrow P_{i-1}$ isomorph ist, weshalb die Abbildungen $\phi_i: \varprojlim P_i/\mathfrak{m}^nP_i \rightarrow \Ima \phi_i$ alle offen sind. Die gewünschte Exaktheit folgt nun aus \cite[Lemma 3.1]{firstpaper}.
\end{proof}

Wir wollen nun zeigen, dass der Modul $\underline{C(S,A)}$ für jede proendliche Menge $S$ nuklear über dem analytischen verdichteten Ring $(\mathcal{A},\mathcal{M})$ ist. Wir beweisen zunächst das folgende Lemma. 

\begin{lemma}
\label{torsion}
Für jeden $\mathfrak{m}^{\infty}$-Torsionsmodul $M$ über $R$ gilt $M \underset{R_{\mathrm{disk}}}{\overset{L}{\otimes}}R_{\blacksquare}\cong M$.
\end{lemma}

\begin{proof}
Da der Funktor $-\underset{R_{\mathrm{disk}}}{\overset{L}{\otimes}}R_{\blacksquare}$ mit Kolimites vertauscht, genügt es zu zeigen, dass für jedes $i\geq 0$ die kanonische Abbildung
\[R/(x_1^i,\cdots,x_n^i) \underset{R_{\mathrm{disk}}}{\overset{L}{\otimes}}R_{\blacksquare}\rightarrow R/(x_1^i,\cdots,x_n^i)\]
ein Isomorphismus ist. Dies folgt direkt aus der Koszul-Auflüsung.
\end{proof}

\begin{satz}\label{A_ist_nuklear}
Der Ring $\underline{A}$ ist nuklear als verdichteter Modul über $R_{\blacksquare}$.
\end{satz}

\begin{proof}
Man betrachte folgende exakte Sequenz von verdichteten Moduln:
\[0\rightarrow \underline{A_0}\rightarrow \underline{A}\rightarrow A/A_0\rightarrow 0.\]
Wie man leicht nachweist, ist der Modul $A/A_0$ wegen der Stetigkeit der Multiplikation ein diskreter Torsionsmodul über $R$. Hieraus folgt nun unter Benutzung von Lemma~\ref{torsion}, dass der Modul $A/A_0$ nuklear über $R_\blacksquare$ ist. Deshalb dürfen wir annehmen, dass $A=A_0$ gilt, denn die volle $\infty$-Unterkategorie der nuklearen Moduln ist abgeschlossen unter Kolimites. Wir führen nun folgende leichte Rechnung durch:
\[A_{\mathrm{disk}}\underset{R_{\mathrm{disk}}}{\overset{L}{\otimes}} C(S, R_{\mathrm{disk}})\cong A_{\mathrm{disk}}\underset{R_{\mathrm{disk}}}{\overset{L}{\otimes}}(\underset{J}{\bigoplus} R_{\mathrm{disk}})\cong \underset{J}{\bigoplus} A_{\mathrm{disk}}\cong C(S, A_{\mathrm{disk}}).\]
Aus Satz~\ref{monoidal} und Lemma~\ref{vervollständigung} folgt daher
\[\underline{C(S,R)}\underset{R_{\blacksquare}}{\overset{L}{\otimes}}\underline{A}\cong \mathbb{L}\Lambda_\mathfrak{m}(C(S,R_{\mathrm{disk}}) \underset{R_{\mathrm{disk}}}{\overset{L}{\otimes}} A_{\mathrm{disk}})\cong \mathbb{L}\Lambda_\mathfrak{m}(C(S,A_{\mathrm{disk}}))\cong \underline{C(S,A)}.\qedhere\]
\end{proof}

\begin{korollar}
\label{nuklearität}
Für jede proendliche Menge $S$ ist der Modul $\underline{C(S,A)}$ nuklear über dem analytischen verdichteten Ring $(\mathcal{A},\mathcal{M})$.
\end{korollar}

\begin{proof}
Sei $S'$ eine proendliche Menge. Aus Lemma~\ref{grundnuklearität} und Satz~\ref{A_ist_nuklear} folgt
\[\underline{C(S,A)} \underset{(\mathcal{A},\mathcal{M})}{\overset{L}{\otimes}} \underline{C(S',A)} \cong \underline{A} \underset{R_\blacksquare}{\overset{L}{\otimes}} \underline{C(S,R)} \underset{R_\blacksquare}{\overset{L}{\otimes}} \underline{C(S',R)}\cong \underline{A} \underset{R_\blacksquare}{\overset{L}{\otimes}} \underline{\mathrm{Hom}}_{R_\blacksquare} (R_\blacksquare[S], \underline{C(S',R)})\cong \underline{C(S\times S',A)}.\]
\end{proof}

\begin{korollar}\label{tensor_produkt_von_nuklearen}
Seien $S$ und $S'$ proendliche Mengen. Dann gilt
\[\underline{C(S,A)} \underset{(\mathcal{A},\mathcal{M})}{\overset{L}{\otimes}} \underline{C(S',A)} \cong \underline{C(S\times S',A)}.\]
\end{korollar}

Wir kommen nun zum Beweis der Tatsache, dass die $\infty$-Kategorie der nuklearen Moduln über $(A,A^+)_\blacksquare$ dualisierbar ist. Dazu benötigen wir eine allgemein gültige Aussage für beliebige analytische animierte verdichtete Ringe. Sei $(\mathcal{A},\mathcal{M})$ ein solcher Ring. Es bezeichne $\mathcal{C}$ die derivierte $\infty$-Kategorie $\mathcal{D}(\mathcal{A},\mathcal{M})$ der Moduln über $(\mathcal{A},\mathcal{M})$. Man betrachte die Äquivalenz $\mathcal{C}\xrightarrow[]{\sim} \mathrm{Fun}^{\mathrm{lim}}(\mathcal{C}^\mathrm{op},\mathrm{Sp})$, welche von der Yoneda-Einbettung induziert wird. Hierbei bezeichnet $\mathrm{Fun}^{\mathrm{lim}}(\mathcal{C}^\mathrm{op},\mathrm{Sp})$ die volle $\infty$-Unterkategorie der Limites erhaltenden Funktoren. Da $\mathcal{C}$ kompakt erzeugt ist, induziert diese Äquivalenz durch Einschränken auf die volle $\infty$-Unterkategorie $\mathcal{C}^{\omega}$ der kompakten Objekte in $\mathcal{C}$ eine Äquivalenz $\mathcal{C} \xrightarrow{\sim} \mathrm{Fun}^{\mathrm{ex}}((\mathcal{C}^{\omega})^{\mathrm{op}}, \mathrm{Sp})$, wobei $\mathrm{Fun}^{\mathrm{ex}}((\mathcal{C}^{\omega})^{\mathrm{op}}, \mathrm{Sp})$ die volle $\infty$-Unterkategorie der exakten Funktoren ist. Es bezeichne $(-)^{\vee}:\mathcal{C}\rightarrow \mathcal{C}$ das interne Dual $\mathrm{R}\underline{\mathrm{Hom}}_{(\mathcal{A},\mathcal{M})}(-,\mathcal{A})$. Unter diesen Äquivalenzen definiert der Funktor 
\[\mathcal{C} \rightarrow \mathrm{Fun}^{\mathrm{ex}}((\mathcal{C}^{\omega})^{\mathrm{op}}, \mathrm{Sp}),\ M\mapsto M^{\mathrm{spur}}:=((-)^{\vee} \underset{(\mathcal{A},\mathcal{M})}{\overset{L}{\otimes}} M)(\ast)\]
einen Funktor $\mathcal{C}\rightarrow \mathcal{C}$, den wir den \textit{Spurfunktor} nennen. Sei $M$ nun ein Objekt von $\mathcal{C}$. Der Spurfunktor ausgestattet mit einer natürlichen Transformation $(-)^{\mathrm{spur}}\rightarrow \mathrm{id}$ von Endofunktoren von $\mathcal{C}$, welche von der natürlichen Transformation $((-)^{\vee} \underset{(\mathcal{A},\mathcal{M})}{\overset{L}{\otimes}} M)(\ast)\rightarrow \mathrm{Hom}_{\mathcal{C}}(-,M)$ induziert wird. Aus den Definitionen folgt direkt, dass der Morphismus $M^{\mathrm{spur}}\rightarrow M$ für $M \in \mathcal{C}$ genau dann ein Isomorphismus ist, wenn $M$ nuklear ist. Um die Dualisierbarkeit der $\infty$-Kategorie der nuklearen Moduln zu nachzuweisen, benötigen wir die folgende alternative Beschreibung des Spurfunktors.

\begin{satz}
\label{alternative_formel}
Sei $(\mathcal{A},\mathcal{M})$ ein beliebiger analytischer animierter verdichteter Ring. Dann ist der Spurfunktor durch
\[M\in \mathcal{C}\mapsto M^{\mathrm{spur}}=\underset{P,\ \mathcal{A}\rightarrow P\otimes M}{\colim}\, P^{\vee}\]
gegeben, wobei $P$ die kompakten Objekte von $\mathcal{C}$ zusammen mit einem Morphismus $\mathcal{A}\rightarrow P \otimes M$ durchläuft.
\end{satz}

\begin{proof}
Gegeben sei ein Objekt $M \in \mathcal{C}$. Da die $\infty$-Kategorie $\mathcal{C}$ kompakt erzeugt ist, gilt $M^\mathrm{spur}= \underset{Q\rightarrow M^\mathrm{spur}}{\colim}\, Q$, wobei $Q$ die kompakten Objekte von $\mathcal{C}$ zusammen mit einem Morphismus $Q\rightarrow M^\mathrm{spur}$ durchläuft. Wie wir in der obigen Diskussion gesehen haben, ist das Morphismenspektrum $\mathrm{Hom}_{\mathcal{C}}(N, M^{\mathrm{spur}})$ für jedes kompakte Objekt $N\in\mathcal{C}$ isomorph zum Spektrum $(N^{\vee} \underset{(\mathcal{A},\mathcal{M})}{\overset{L}{\otimes}} M)(\ast)$. Daraus folgt, dass $\mathrm{Hom}_\mathcal{C}(Q,M^\mathrm{spur})$ zu $\underset{P\rightarrow Q^\vee}{\colim}\,( P\underset{(\mathcal{A},\mathcal{M})}{\otimes} M)(\ast)$ isomorph ist, wobei $P$ die kompakten Objekte von $\mathcal{C}$ zusammen mit einem Morphismus $P\rightarrow Q^\vee$ durchläuft. Daher gilt
\[M^{\mathrm{spur}}\cong \underset{\substack{Q,\ P \\ 1 \rightarrow P \otimes M \\ P\rightarrow Q^\vee}}{\colim}\ Q\cong \underset{\substack{Q,\ P \\ 1\rightarrow P\otimes M \\ Q\rightarrow P^\vee}}{\colim}\ Q \cong \underset{P,\ 1\rightarrow P\otimes M}{\colim} P^\vee.\qedhere\]
\end{proof}

Wir können jetzt die Dualisierbarkeit der $\infty$-Kategorie der nuklearen Moduln über $(\mathcal{A},\mathcal{M})=(A,A^+)_\blacksquare$ nachweisen.

\begin{satz}\label{dualisierbarkeit_moduln}
Es bezeichne $\mathcal{D}^\mathrm{nuk}(\mathcal{A},\mathcal{M})$ die $\infty$-Kategorie der nuklearen Moduln über $(\mathcal{A},\mathcal{M})$. Der Spurfunktor $\mathcal{D}(\mathcal{A},\mathcal{M})\rightarrow \mathcal{D}(\mathcal{A},\mathcal{M})$ vertauscht mit Kolimites und faktorisiert über die natürliche Einbettung $\mathcal{D}^\mathrm{nuk}(\mathcal{A},\mathcal{M})\hookrightarrow \mathcal{D}(\mathcal{A},\mathcal{M})$. Insbesondere ist $\mathcal{D}^{\mathrm{nuk}}(\mathcal{A},\mathcal{M})$ ein Retrakt einer kompakt erzeugten $\infty$-Kategorie und somit dualisierbar.  
\end{satz}

\begin{proof}
Die Tatsache, dass der Spurfunktor Kolimites erhält, folgt direkt aus der Definition. Gegeben sei ein Objekt $M\in \mathcal{D}(\mathcal{A},\mathcal{M})$. Nach Satz~\ref{alternative_formel} ist das Objekt $M^\mathrm{spur}$ isomorph zu einem Kolimes von Objekten der Form $P^\vee$ mit $P$ kompakt. Ein solches Objekt $P$ ist isomorph zu einem Retrakt eines endlichen Komplexes, dessen Terme alle der Form $\mathcal{M}[S]$ mit $S$ proendlich sind. Aus unseren Argumentationen oben folgt, dass jedes $\mathcal{M}[S]^\vee$ zu einem Kolimes von Tate-Algebren über $A$ isomorph ist. D. h., das Objekt $P^\vee$ liegt in der vollen stabilen $\infty$-Unterkategorie von $\mathcal{D}(\mathcal{A},\mathcal{M})$, welche von $A\langle T \rangle $ erzeugt wird und unter Kolimites abgeschlossenen ist. Daraus folgt außerdem, dass die $\infty$-Kategorie $\mathcal{D}^\mathrm{nuk}(\mathcal{A},\mathcal{M})$ innerhalb der kompakt erzeugten vollen $\infty$-Unterkategorie $\mathcal{D}(\mathcal{A},\mathcal{M})_{\kappa}$ für eine überabzählbare starke Limes-Kardinalzahl $\kappa$ liegt, siehe \cite[Anhang zu Vorlesung II]{Condensed}. Da für jedes $N\in \mathcal{D}^\mathrm{nuk}(\mathcal{A},\mathcal{M})$ die natürliche Abbildung $N^\mathrm{spur}\rightarrow N$ ein Isomorphismus ist, ist $\mathcal{D}^\mathrm{nuk}(\mathcal{A},\mathcal{M})$ also ein Retrakt von $\mathcal{D}(\mathcal{A},\mathcal{M})_{\kappa}$.
\end{proof}

Wir erhalten den folgenden Satz als direkte Konsequenz des Beweises und der Tatsache, dass für jedes Paar von proendlichen Mengen $S,S'$ das Tensorprodukt $\underline{C(S,A)}\underset{(\mathcal{A},\mathcal{M})}{\otimes}\underline{C(S',A)}$ isomorph zu $\underline{C(S\times S',A)}$ ist, und somit von der Wahl des Ringes $A^+$ unabhängig. 

\begin{korollar}\label{unabhängigkeit}
Das folgende Diagramm ist kommutativ:

\begin{center}
\begin{tikzcd}[column sep=-0.5em]
\mathcal{D}((A,A^{\circ})_\blacksquare) \arrow{d}{X\mapsto X^{\mathrm{tr}}} & \subseteq &\mathcal{D}((A,A^{+})_\blacksquare) \arrow{d}{X\mapsto X^{\mathrm{tr}}} & \subseteq & \mathcal{D}((A,\mathbb{Z})_\blacksquare) \arrow{d}{X\mapsto X^{\mathrm{tr}}}  \\
\mathcal{D}^{\mathrm{nuk}}((A,A^{\circ})_\blacksquare) & = & \mathcal{D}^{\mathrm{nuk}}((A,A^{+})_\blacksquare) & = & \mathcal{D}^{\mathrm{nuk}}((A,\mathbb{Z})_\blacksquare)
\end{tikzcd}
\end{center}
Mit anderen Worten sind der Spurfunktor und die volle $\infty$-Unterkategorie $\mathcal{D}^{\mathrm{nuk}}((A,A^{+})_\blacksquare)\subset \mathcal{D}(\underline{A})$ von $A^+$ unabhängig. Außerdem ist die monoidale Struktur auf $\mathcal{D}^{\mathrm{nuk}}((A,A^+)_\blacksquare)$ ebenfalls von $A^+$ unabhängig.
\end{korollar}

Aufgrund dieses Korollars schreiben wir in den folgenden Abschnitten für die $\infty$-Kategorie der nuklearen Moduln $\mathcal{D}^\mathrm{nuk}((A,A^{+})_\blacksquare)$ einfach $\mathrm{Nuk}(A)$. Außerdem ergibt sich aus dem Beweis des obigen Satzes das folgende Korollar.
\begin{korollar} 
Sei $(R,R^+)$ ein diskretes Huber-Paar. Dann ist die $\infty$-Kategorie der nuklearen Moduln über $(R,R^+)_{\blacksquare}$ äquivalent zur algebraischen derivierten $\infty$-Kategorie des Ringes $R$.
\end{korollar}
\newpage
\section{Nukleare Garben}

Im vorhergehenden Abschnitt führten wir die $\infty$-Kategorie der nuklearen Moduln ein und untersuchten ihre Eigenschaften. Unser Hauptresultat, welches im weiteren Verlauf dieser Arbeit eine Schlüsselrolle spielen wird, ist die Tatsache, dass diese Kategorie dualisierbar ist. In diesem Abschnitt erläutern wir die allgemeinere Definition der $\infty$-Kategorie der \textit{nuklearen Garben} auf einem analytischen adischen Raum und zeigen dann mithilfe unserer Ergebnisse über dualisierbare Kategorien, dass sie ebenfalls dualisierbar ist. Wie bereits mehrfach betont, ist diese Tatsache der wichtigste Bestandteil unserer Argumentation im folgenden Abschnitt, in dem wir die $K$-Theorie eines analytischen adischen Raumes definieren und ihre Abstiegseigenschaften analysieren. Wir beginnen mit dem Beweis des folgenden Lemmas, das uns erlaubt, die Ergebnisse des vorhergehenden Abschnitts auf analytische adische Räume anzuwenden.

\begin{lemma}
Jeder vollständige tatesche Huber-Ring $A$ ist schwach proregulär.
\end{lemma}

\begin{proof}
Sei $\varpi$ eine Uniformisierende von $A$. Nach \cite[Proposition 5.6]{WPR} ist die schwache Proregularität des Ringes $A$ äquivalent dazu, dass die $\varpi$-Torsion eines Definitionsringes von $A$ nach oben beschränkt ist, was in unserer Situation in offensichtlicher Weise gilt, denn $\varpi$ ist invertierbar.
\end{proof}

Insbesondere ist die $\infty$-Kategorie der nuklearen Moduln über einem vollständigen tateschen Huber-Ring dualisierbar. Wir wollen nun die $\infty$-Kategorie der nuklearen Garben definieren, indem wir die „lokalen“ Kategorien der nuklearen Moduln auf offenen affinoiden Teilmengen formal verkleben. Dazu benötigen wir den folgenden Abstiegssatz für nukleare Moduln.

\begin{satz}[{\cite[Theorem 5.42]{firstpaper}}]\label{lokaler_nuklearer_abstieg}
Sei $(A,A^+)$ ein vollständiges garbiges analytisches Huber-Paar und $U\subset\Spa (A,A^+)$ eine beliebige affinoide Teilmenge. Es bezeichne $A_U$ den Ring $\mathcal{O}_{\Spa (A,A^+)}(U)$. Dann definiert die Zuordnung $U\mapsto \mathrm{Nuk}(A_U)$ eine Garbe von $\infty$-Kategorien auf $\Spa (A,A^+)$ bezüglich der analytischen Topologie.
\end{satz}

Wir müssen die implizite Struktur der Garbe im obigen Satz erläutern, nämlich wie die entsprechenden Einschränkungsfunktoren aussehen. Seien $U,V$ offene affinoide Teilmengen von $\Spa (A,A^+)$ mit $U\subset V$. Es bezeichne $A_U$ und $A_U^+$ (bzw. $A_V$ und $A_V^+$) die Ringe $\mathcal{O}_{\Spa (A,A^+)}(U)$ und $\mathcal{O}^+_{\Spa (A,A^+)}(U)$ (bzw. $\mathcal{O}_{\Spa (A,A^+)}(V)$ und $\mathcal{O}^+_{\Spa (A,A^+)}(V)$). Der Einschränkungsfunktor $\mathrm{Nuk}(V)\rightarrow \mathrm{Nuk}(U)$ ist durch den Rückzug der nuklearen Moduln gegeben: Der Funktor \[-\underset{(A_V,A_V^+)_\blacksquare}{\otimes}(A_U,A_U^+)_\blacksquare: \mathcal{D}((A_V,A_V^+)_\blacksquare)\rightarrow \mathcal{D}((A_U,A_U^+)_\blacksquare)\] induziert einen Funktor $\mathrm{Nuk}(A_V)\rightarrow \mathrm{Nuk}(A_U)$, wie der Beweis von \cite[Theorem 5.42]{firstpaper} zeigt. Im Weiteren bezeichnen wir diesen Funktor mit $j^\ast$, wobei $j:U\rightarrow V$ die entsprechende offene Einbettung ist.

Anhand des Satzes~\ref{lokaler_nuklearer_abstieg} können wir nun die $\infty$-Kategorie der \textit{nuklearen Garben} auf einem analytischen adischen Raum durch formales Verkleben der entsprechenden Kategorien auf den affinoiden Teilmengen definieren.

\begin{definition}\label{nukleare_garben}
Sei $X$ ein analytischer adischer Raum. Die $\infty$-Kategorie ${\mathrm{Nuk}}(X)$ \textit{der nuklearen Garben auf $X$} ist definiert als der Limes $\varprojlim \mathrm{Nuk}(\mathcal{O}_X(U))$, wobei $U$ die affinoiden offenen Teilmengen von $X$ durchläuft.
\end{definition}

Da alle Einschränkungsfunktoren für affinoide offene Teilmengen mit Kolimites vertauschen, ist die so definierte $\infty$-Kategorie automatisch präsentierbar. Wie wir im Folgenden sehen werden, ist sie auch dualisierbar. Man prüft mithilfe des Satzes~\ref{lokaler_nuklearer_abstieg} leicht nach, dass die Zuordnung $U\mapsto \mathrm{Nuk}(U)$ für $U\subset X$ offen eine Garbe von $\infty$-Kategorien bezüglich der analytischen Topologie definiert. Die entsprechenden Einschränkungsfunktoren werden ebenfalls durch Verkleben der Einschränkungsfunktoren auf den affinoiden Teilmengen konstruiert, die wir im Folgenden auch \textit{Lokalisierungen} oder \textit{Rückzugsfunktoren} nennen. Der erste Name ist durch die Tatsache gerechtfertigt, dass diese Funktoren unter milderen Annahmen linke Bousfield-Lokalisierungen sind. In der Tat seien $U,V$ offene Teilmengen von $X$ mit $j:U\subset V$. Angenommen, die beiden Teilmengen sind quasi-kompakt und quasi-separiert. Falls sie außerdem affinoid sind, definiert der Vergissfunktor $\mathcal{D}((\mathcal{O}_X(U),\mathcal{O}_X^+(U))_\blacksquare)\rightarrow \mathcal{D}((\mathcal{O}_X(V),\mathcal{O}_X^+(V))_\blacksquare)$ einen Funktor $\mathrm{Nuk}(\mathcal{O}_X(U))\rightarrow \mathrm{Nuk}(\mathcal{O}_X(V))$, wie Lemmata~\ref{adische_morphismen} und \ref{rigidität_von_der_nuklerität} unten zeigen. Im Weiteren bezeichnen wir diesen Funktor mit $j_\ast$ und nennen ihn den \textit{direkten Bildfunktor}. Da der Vergissfunktor für alle festen Moduln nach \cite[Proposition 4.12]{firstpaper} volltreu und rechtsadjungiert zum Funktor \[-\underset{(\mathcal{O}_X(V),\mathcal{O}_X^+(V))_\blacksquare}{\otimes}(\mathcal{O}_X(U),\mathcal{O}_X^+(U))_\blacksquare: \mathcal{D}((\mathcal{O}_X(V),\mathcal{O}_X^+(V))_\blacksquare)\rightarrow \mathcal{D}((\mathcal{O}_X(U),\mathcal{O}_X^+(U))_\blacksquare)\] ist, definiert seine Einschränkung auf die $\infty$-Kategorie der nuklearen Moduln einen volltreuen Funktor, welcher zum Rückzugfunktor $j^\ast$ rechtsadjungiert ist. Wenn $U$ und $V$ nun beliebige offene quasi-kompakte quasi-separierte Teilmengen von $X$ sind, konstruiert man formal den direkten Bildfunktor \[j_\ast:\mathrm{Nuk}(U)\rightarrow\mathrm{Nuk}(V).\] Wie man leicht verifiziert, ist der so definierte Funktor volltreu und rechtsadjungiert zum Rückzugsfunktor $j^\ast$. Außerdem vertauscht er wegen der Annahme mit Kolimites, denn er ist automatisch exakt und die Kategorien $\mathrm{Nuk}(U)$ und $\mathrm{Nuk}(V)$ sind isomorph zu \textit{endlichen} Limites von $\infty$-Kategorien der nuklearen Moduln.

\begin{lemma}\label{adische_morphismen}
Sei $f:(A,A^+)\rightarrow (B,B^+)$ ein Morphismus von vollständigen tateschen Huber-Paaren. Dann ist $\underline{B}$ nuklear über $(A,A^+)_\blacksquare$.
\end{lemma}

\begin{proof}
Sei $\varpi$ eine Uniformisierende von $A$. Wir versehen $A$ mit der Struktur eines Moduls über $\mathbb{Z}[[x]]$, die durch $\mathbb{Z}[[x]]\rightarrow A$, $x\mapsto \varpi$ gegeben ist. Dann sind $\underline{A}$ und $\underline{B}$ nach Satz~\ref{A_ist_nuklear} nuklear über $\mathbb{Z}[[x]]_\blacksquare$. Sei $S$ eine proendliche Menge. Zum Nachweis der gewünschten Nuklearität führen wir folgende Rechnung durch:

\[\underline{C(S,A)}\underset{(A,A^+)_\blacksquare}{\overset{L}{\otimes}}\underline{B}\cong \underline{C(S,\mathbb{Z}[[x]])}\underset{\mathbb{Z}[[x]]_\blacksquare}{\overset{L}{\otimes}}\underline{A}\underset{(A,A^+)_\blacksquare}{\overset{L}{\otimes}}\underline{B}\cong \underline{C(S,\mathbb{Z}[[x]])}\underset{\mathbb{Z}[[x]]_\blacksquare}{\overset{L}{\otimes}}\underline{B}\cong \underline{C(S,B)}.\qedhere\]
\end{proof}

\begin{lemma}\label{rigidität_von_der_nuklerität}
Sei $(A,A^+)\rightarrow (B,B^+)$ ein Morphismus von schwach proregulären Huber-Paaren, sodass $\underline{B}$ nuklear über $(A,A^+)_\blacksquare$ ist. Dann ist ein Objekt $M\in \mathcal{D}((B,B^+)_{\blacksquare})$ nuklear über $(B,B^+)_\blacksquare$ genau dann, wenn es nuklear über $(A,A^+)_\blacksquare$ ist. Insbesondere erhält der Vergissfunktor $\mathcal{D}((B,B^+)_{\blacksquare})\rightarrow \mathcal{D}((A,A^+)_{\blacksquare})$ nukleare Objekte.
\end{lemma}

\begin{proof}
Angenommen, $M$ ist nuklear über $(A,A^+)_\blacksquare$. Sei $S$ eine proendliche Menge. Wir führen folgende Rechnung durch:

\[\underline{C(S,B)}\underset{(B,B^+)_\blacksquare}{\overset{L}{\otimes}} M\cong \underline{C(S,A)} \underset{(A,A^+)_\blacksquare}{\overset{L}{\otimes}} \underline{B}\underset{(B,B^+)_\blacksquare}{\overset{L}{\otimes}} M\cong \mathrm{R}\underline{\mathrm{Hom}}_{(A,A^+)_\blacksquare}((A,A^+)_\blacksquare[S],M)\]\[\cong \mathrm{R}\underline{\mathrm{Hom}}_{(B,B^+)_\blacksquare}((B,B^+)_\blacksquare[S],M).\]

Wenn nun $M$ nuklear über $(B,B^+)_\blacksquare$ ist, zeigt eine völlig analoge Rechnung, dass es auch nuklear über $(A,A^+)_\blacksquare$ ist.
\end{proof}

Wir haben nun alle benötigten formalen Eigenschaften der $\infty$-Kategorie der nuklearen Moduln bewiesen und können formal aus Satz~\ref{dualisierbarer_abstieg} den folgenden Satz ableiten.

\begin{satz}
Sei $X$ ein quasi-separierter quasi-kompakter analytischer adischer Raum. Dann ist die $\infty$-Kategorie $\mathrm{Nuk}(X)$ dualisierbar.
\end{satz}

In den folgenden Abschnitten werden wir die folgende Version der Kategorie der nuklearen Garben benötigen.

\begin{definition}
Sei $X$ ein analytischer adischer Raum und $Y$ eine abgeschlossene Teilmenge von $X$. Es bezeichne $U$ das Komplement von $Y$ in $X$. Die $\infty$-Kategorie $\mathrm{Nuk}(X\ \mathrm{auf} \ Y)$ der \textit{nuklearen Garben auf $X$ mit Träger in $Y$} ist definiert als die Faser des Einschränkungsfunktors $\mathrm{Nuk}(X)\rightarrow \mathrm{Nuk}(U)$.
\end{definition}

\begin{korollar}
Sei $X$ ein quasi-kompakter quasi-separierter analytischer adischer Raum und $Y$ eine abgeschlossene Teilmenge von $X$. Angenommen, das Komplement von $Y$ ist quasi-kompakt. Dann ist die Kategorie $\mathrm{Nuk}(X\ \mathrm{auf} \ Y)$ dualisierbar.
\end{korollar}

\begin{proof}
Dies folgt aus Lemma~\ref{dualisierbare_faser}.
\end{proof}

Wir beweisen nun die folgenden weiteren Eigenschaften der $\infty$-Kategorie der nuklearen Garben, welche wir zur Analyse der étalen Abstiegseigenschaften der $K$-Theorie verwenden werden. 

\begin{satz}\label{perfekt-nuklear}
Sei $A\rightarrow B$ ein endlicher étaler Morphismus von vollständigen tateschen Huber-Ringen. Dann gilt
\[\mathrm{Nuk}(B)\cong \mathrm{Nuk}(A)\underset{\mathrm{Perf}(A)}{\otimes}\mathrm{Perf}(B),\]
wobei $\mathrm{Perf}(A)$ (bzw. $\mathrm{Perf}(B)$) die $\infty$-Kategorie der perfekten Komplexe über $A$ (bzw. $B$) bezeichnet.
\end{satz}

\begin{proof}
Nach Annahme ist die $\infty$-Kategorie $\mathrm{Perf}(B)$ äquivalent zur $\infty$-Kategorie $\mathrm{Mod}_B(\mathrm{Perf}(A))$, und somit $\mathrm{Nuk}(A)\underset{\mathrm{Perf(A)}}{\otimes}\mathrm{Perf}(B)$ zu $\mathrm{Mod}_B(\mathrm{Nuk}(A))$. Die $\infty$-Kategorie $\mathcal{D}((A,\mathbb{Z})_\blacksquare)$ (bzw. $\mathcal{D}((B,\mathbb{Z})_\blacksquare)$) ist per Definition die $\infty$-Kategorie $\mathrm{Mod}_{\underline{A}}(\mathcal{D}(\mathbb{Z}_\blacksquare))$ (bzw. $\mathrm{Mod}_{\underline{B}}(\mathcal{D}(\mathbb{Z}_\blacksquare))$). Dementsprechend genügt es zu zeigen, dass ein Objekt $M\in \mathcal{D}((B,\mathbb{Z})_\blacksquare)$ genau dann nuklear über $(B,\mathbb{Z})_\blacksquare$ ist, wenn es nuklear über $(A,\mathbb{Z})_\blacksquare$ ist. Dies folgt aber sofort aus Lemma~\ref{rigidität_von_der_nuklerität}.
\end{proof}

\begin{korollar}[étaler Abstiegssatz für nukleare Garben]\label{étaler_nuklearer_abstieg}
Sei $X$ ein analytischer adischer Raum. Dann definiert die Zuordnung $U\mapsto \mathrm{Nuk}(U)$ für $U$ étale über $X$ eine Garbe von $\infty$-Kategorien bezüglich der étalen Topologie auf $X$.
\end{korollar}

\begin{proof}
Nach Satz~\ref{lokaler_nuklearer_abstieg} und dem Analogon von Satz~\ref{étale=Nisnevich+Galois} für Garben von $\infty$-Kategorien genügt es zu zeigen, dass für jede surjektive endliche étale Abbildung $f:\Spa (B,B^+)\rightarrow \Spa (A,A^+)$ zwischen tateschen affinoiden adischen Räumen die $\infty$-Kategorie $\mathrm{Nuk}(A)$ isomorph zum Limes $\varprojlim \mathrm{Nuk}(B\underset{A}{\otimes}\dots \underset{A}{\otimes} B)$ ist. Für einen Ring $R$ bezeichne $\mathcal{D}(R)$ die algebraische derivierte $\infty$-Kategorie von $R$. Aus Satz~\ref{perfekt-nuklear} lässt sich leicht ableiten \[\mathrm{Nuk}(B\underset{A}{\otimes}\dots \underset{A}{\otimes} B)\cong \mathrm{Nuk}(A)\underset{\mathcal{D}(A)}{\otimes}\mathcal{D}(B\underset{A}{\otimes}\dots \underset{A}{\otimes} B).\]
Wie man leicht nachprüft, ist der Morphismus $\Spec B\rightarrow \Spec A$ treuflach. Es gilt daher 
\[\mathcal{D}(A)\xrightarrow[]{\sim} \mathrm{Tot}\Bigl(\mathcal{D}(B)\rightrightarrows \mathcal{D}(B\underset{A}{\otimes}B) \mathrel{\substack{\textstyle\rightarrow\\[-0.6ex]
                      \textstyle\rightarrow \\[-0.6ex]
                      \textstyle\rightarrow}} \dots\Bigr).\]
Es genügt also zu zeigen, dass die obere Totalisierung mit dem Tensorprodukt $\mathrm{Nuk}(A)\underset{\mathcal{D}(A)}{\otimes}-$ vertauscht. Dies folgt aus \cite[Korollar 3.42]{Mathew}.
\end{proof}

Zum Schluss wollen wir noch den folgenden Hilfsatz beweisen, mit dessen Hilfe wir im folgenden Abschnitt zeigen werden, dass unsere Definition der $K$-Theorie mit der alten übereinstimmt.

\begin{satz}\label{torsion_ist_diskret}
Sei $A$ ein tatescher Huber-Ring, $A_0$ ein Definitionsring von $A$ und $\varpi\in A_0$ eine Uniformisierende. Es bezeichne $\mathrm{Tors}(\varpi^\infty)$ die $\infty$-Kategorie der $\varpi$-Torsionsmoduln in der algebraischen derivierten $\infty$-Kategorie $\mathcal{D}(A_0)$ des Ringes $A_0$. Dann ist die Faser der Lokalisierung $\mathrm{Nuk}(A_0)\rightarrow \mathrm{Nuk}(A)$ äquivalent zur $\infty$-Kategorie $\mathrm{Tors}(\varpi^\infty)$.
\end{satz}

\begin{proof}
Wir betrachten die algebraische derivierte $\infty$-Kategorie $\mathcal{D}(A_0)$ mittels des \textit{Kondensierungsfunktors} als eine volle $\infty$-Unterkategorie von $\mathcal{D}((A_0,\mathbb{Z})_\blacksquare)$, siehe \cite[Theorem 5.9]{firstpaper}. Wir müssen zeigen, dass jeder nukleare Modul $C$ über $A_0$, der nach dem Invertieren von $\varpi$ trivial wird, innerhalb von $\mathcal{D}(A_0)$ liegt. Man prüft direkt nach, dass für jede proendliche Menge $S$ gilt $\underline{C(S,A_0)}[1/\varpi]\cong\underline{C(S,A)}$. Es genügt also zu zeigen, dass der Kegel \[\mathrm{cone}\, (\underline{C(\mathbb{N}\cup \{\infty\},A_0)}\rightarrow \underline{C(\mathbb{N}\cup \{\infty\},A)})\] in $\mathcal{D}(A_0)$ liegt, denn die $\infty$-Kategorie $\mathrm{Nuk}(A_0)$ wird von $\underline{C(\mathbb{N}\cup  \{\infty\},A_0)}$ erzeugt und die volle $\infty$-Unterkategorie $\mathcal{D}(A_0)\subset \mathcal{D}((A_0,\mathbb{Z})_\blacksquare)$ ist abgeschlossen unter Kolimites. Dies folgt aber direkt aus der exakten Sequenz von verdichteten Moduln
\[0\rightarrow \underline{A_0\langle T\rangle }\rightarrow \underline{A\langle T\rangle }\rightarrow \underset{\mathbb{N}}{\bigoplus}A/A_0\rightarrow 0.\qedhere\]
\end{proof}

\newpage
\section{\texorpdfstring{$K$}{K}-Theorie, lokalisierende Invarianten und Abstieg}

Der Begriff der \textit{lokalisierenden Invariante} ist eine sehr präzise und leistungsfähige Interpretation der Idee der algebraischen Invarianten von Kategorien, der schon seit langer Zeit im Mittelpunkt der mathematischen Forschung steht. Wir erinnern an die Definition:

\begin{definition}
Sei $\mathcal{D}$ eine stabile $\infty$-Kategorie. Es bezeichne $\mathrm{Cat}_{\infty}^{\mathrm{perf}}$ die $\infty$-Kategorie der kleinen karoubischen\footnote{d. h. unter Retrakten abgeschlossenen} stabilen $\infty$-Kategorien. Unter einer \textit{lokalisierenden Invariante mit Werten in $\mathcal{D}$} verstehen wir einen Funktor $F:\mathrm{Cat}_{\infty}^{\mathrm{perf}}\rightarrow \mathcal{D}$, der finale Objekte auf finale Objekte und Verdier-Sequenzen auf Fasersequenzen abbildet.\footnote{In der Definition wird nicht gefordert, dass ein solcher Funktor mit Kolimites vertauscht.}
\end{definition}

Ein besonders wichtiges Beispiel für eine lokalisierende Invariante ist die \textit{nicht konnektive $K$-Theorie}. Diese ist ein Funktor $\mathbb{K}:\mathrm{Cat}_{\infty}^{\mathrm{perf}}\rightarrow \mathrm{Sp}$ mit der folgenden universellen Eigenschaft: Sie ist die initiale lokalisierende Invariante mit Werten in $\mathrm{Sp}$ mit einer Abbildung $\mathbb{S}[-]\circ \mathrm{core}\rightarrow \mathbb{K}$. Hierbei bezeichnet $\mathrm{core}$ den Funktor, der eine $\infty$-Kategorie $\mathcal{C}$ auf das maximale Unteranima von $\mathcal{C}$ abbildet. Die konnektive Version der $K$-Theorie ist ein Funktor $K:\mathrm{Cat}_{\infty}^{\mathrm{perf}}\rightarrow \mathrm{Sp}_{\geq 0}$. Sie kann durch eine analoge universelle Eigenschaft gegeben werden, aber wir setzen einfach $K\overset{\mathrm{def}}{=}\Omega^\infty \mathbb{K}$. Die konnektive Version bildet jedoch keine lokalisierende Invariante, wenn man $\mathrm{Sp}_{\geq 0}$ als $\infty$-Unterkategorie von $\mathrm{Sp}$ ansieht: Sie ist nur eine sogenannte \textit{additive Invariante}.

Mithilfe der lokalisierenden Invarianten wurden zahlreiche unterschiedliche Fragestellungen und Objekte aus vielen Bereichen der Mathematik erfolgreich und detailliert untersucht. Sie bereichern insbesondere die algebraische Geometrie, wo sie eng mit deren schwierigsten und tiefsten Problemen verbunden sind. Von besonderer Bedeutung für diese Arbeit ist die folgende Definition:

\begin{definition}\label{definition_lokalisierende_schemata}
Sei $X$ ein quasi-kompaktes quasi-separiertes Schema und $\mathcal{D}$ eine stabile $\infty$-Kategorie. Unter einer \textit{lokalisierenden Invariante auf $X$ mit Werten in $\mathcal{D}$} verstehen wir eine Prägarbe
\[\mathcal{F}:\{\textrm{quasi-kompakte offene Teilmengen von}\ X\}\rightarrow \mathcal{D},\ U\mapsto F(\mathrm{Perf}(U)),\]
wobei $F:\mathrm{Cat}_{\infty}^{\mathrm{perf}}\rightarrow \mathcal{D}$ eine lokalisierende Invariante mit Werten in $\mathcal{D}$ ist. Hierbei bezeichnet $\mathrm{Perf}(U)$ die $\infty$-Kategorie der perfekten Komplexen auf $U$.
\end{definition}

Eine der wichtigsten abstrakten Eigenschaften der schematischen lokalisierenden Invarianten ist der Nisnevich-Abstieg. Dieser ergibt sich nach Thomason-Trobaugh \cite{TT90} (siehe \cite[Appendix A]{CMNN20} für eine moderne Darlegung) automatisch aus der Definition:

\begin{satz}
Sei $X$ ein quasi-kompaktes quasi-separiertes Schema und $\mathcal{F}$ eine lokalisierende Invariante auf $X$ mit Werten in $\mathcal{D}$. Dann bildet $\mathcal{F}$ eine Garbe mit Werten in $\mathcal{D}$ bezüglich der Nisnevich-Topologie.
\end{satz}

Viel subtiler ist die Frage, ob eine gegebene lokalisierende Invariante auf $X$ den stärkeren \textit{étalen} Abstieg erfüllt: Wie das Beispiel von der $K$-Theorie zeigt, ist dies im Allgemeinen nicht immer der Fall.\footnote{Für ein konkretes Gegenbeispiel verweisen wir auf \url{https://mathoverflow.net/questions/239393/simplest-example-of-failure-of-finite-galois-descent-in-algebraic-k-theory}.} In derselben Arbeit \cite{TT90} hat Thomason jedoch – wenn auch unter einigen unnötigen Annahmen – bewiesen, dass nach \textit{chromatischer Lokalisierung} jede solche Invariante étalen Hyperabstieg erfüllt. Sein Theorem in moderner Form lautet:

\begin{satz}[{\cite[Theorem 7.14]{CM21}}]
Sei $X$ ein quasi-kompaktes quasi-separiertes Schema endlicher Krull-Dimension und $p$ eine feste Primzahl. Wir bezeichnen mit $L_n^{f}$ die linke Bousfield-Lokalisierung bezüglich des Spektrums $T(0)\oplus\cdots\oplus T(n)$, wobei $T(i)$ das Teleskop eines $p$-lokalen endlichen Spektrum vom Typ $i$ ist. Angenommen, die virtuellen $p$-lokalen galoisschen kohomologischen Dimensionen der Restklassenkörper von $X$ sind nach oben beschränkt. Dann bildet jede lokalisierende Invariante $\mathcal{F}$ mit Werten in $L_n^f$-lokalen Spektren eine étale Hypergarbe auf $X$.
\end{satz}

In diesem Abschnitt verwenden wir unsere bisherigen Ergebnisse, um die Methoden von \cite{TT90} (in ihrer modernen Form wie in \cite{CMNN20} und \cite{CM21}) an adische Räume anzupassen und um die Analoga der obigen Sätze zu formulieren und zu beweisen. Insbesondere definieren wir die \textit{stetige konnektive} bzw. \textit{nicht konnektive $K$-Theorie} für analytische adische Räume. Im Folgenden setzen wir Vertrautheit mit dem Material im Anhang voraus.

\begin{notation}
\begin{enumerate}[label=(\roman*)]
\item[] 
\item Mit $\mathrm{St}^{\mathrm{cg}}$ (bzw. $\mathrm{St}^{\mathrm{dual}}$) bezeichnen wir die $\infty$-Kategorie der stabilen kompakt erzeugten (bzw. dualisierbaren) Kategorien, deren Morphismen kompakte (bzw. dualisierbare) Funktoren sind.
\item Wir nehmen an, dass alle Huber-Ringe vollständig sind.
\end{enumerate}
\end{notation}

Die ursprünglichen Techniken von \cite{TT90} sind sehr mächtig: Das einzige Hindernis in unserer Situation ist im Wesentlichen die Definition selbst, die anderen Bestandteile des Beweises funktionieren auch im adischen Kontext. Im schematischen Falle benutzt man implizit die Äquivalenz
\[\mathrm{Cat}_{\infty}^{\mathrm{perf}}\xrightarrow{\sim} \mathrm{St}^{\mathrm{cg}},\ \mathcal{C}\mapsto \Ind (\mathcal{C}).\]
In der Tat ist die derivierte Kategorie der quasi-kohärenten Garben auf einem quasi-kompakten quasi-separierten Schema $X$ nach Satz~\ref{kompakte_verklebung} kompakt erzeugt und dabei gilt $\mathcal{D}_{qc}(X)\cong\Ind(\mathrm{Perf}(X))$. In unserer Situation ist die Kategorie der nuklearen Garben jedoch nur dualisierbar. Deswegen müssen wir erst lokalisierende Invarianten auf dualisierbare Kategorien fortsetzen. Dies ist möglich dank des folgenden Satzes von Efimov:

\begin{satz}[{\cite[Theorem 10]{Efimov}}]
Sei $\mathcal{D}$ eine stabile $\infty$-Kategorie. Unter einer \textit{lokalisierenden Invariante von dualisierbaren Kategorien mit Werten in $\mathcal{D}$} verstehen wir einen Funktor $F:\mathrm{St}^{\mathrm{dual}}\rightarrow \mathcal{D}$, der finale Objekte auf finale Objekte und Verdier-Sequenzen auf Fasersequenzen abbildet. Dann liefert der Funktor
\[\mathrm{Fun}(\mathrm{St}^{\mathrm{dual}},\mathcal{D})\rightarrow \mathrm{Fun}(\mathrm{Cat}_{\infty}^{\mathrm{perf}},\mathcal{D}),\ F\mapsto F\circ \Ind\]
eine Äquivalenz zwischen den vollen Unterkategorien von lokalisierenden Invarianten auf den beiden Seiten.
\end{satz}

Mit anderen Worten kann jede lokalisierende Invariante $F:\mathrm{Cat}_{\infty}^{\mathrm{perf}}\rightarrow \mathrm{D}$ eindeutig zu einer lokalisierenden Invariante auf $\mathrm{St}^{\mathrm{dual}}$ fortgesetzt werden. Diese Fortsetzung wird die \textit{stetige Fortsetzung von $F$} genannt und mit $F_{\mathrm{stet}}$ bezeichnet. Sie besitzt zudem die folgende Eigenschaft:

\begin{satz}[{\cite[Lemma 14]{Efimov}}]
Eine lokalisierende Invariante $F$ vertauscht mit (filtrierten) Kolimites genau dann, wenn deren stetige Fortsetzung $F_{\mathrm{stet}}$ mit (filtrierten) Kolimites vertauscht.
\end{satz}

Als Konsequenz erhalten wir einen Kolimites erhaltenden Funktor $\mathbb{K}_{\mathrm{stet}}:\mathrm{St}^{\mathrm{dual}}\rightarrow \mathrm{Sp}$, den wir die \textit{stetige nicht konnektive $K$-Theorie} nennen. Ihr konnektives Analogon ist definiert als $K_{\mathrm{stet}}\overset{\mathrm{def}}{=}\Omega^\infty \mathbb{K}_{\mathrm{stet}}$. Anhand dieser Funktoren können wir nun die stetige konnektive (bzw. nicht konnektive) $K$-Theorie eines schwach proregulären Huber-Ringes definieren.

\begin{definition}
Sei $A$ ein schwach proregulärer Huber-Ring. Dann ist die stetige konnektive (bzw. nicht konnektive) $K$-Theorie von $A$ definiert als $K_{\mathrm{stet}}(\mathrm{Nuk}(A))\ (\mathrm{bzw.}\ \mathbb{K}_{\mathrm{stet}}(\mathrm{Nuk}(A)))$.
\end{definition}

Sei jetzt $A$ ein $I$-adisch vollständiger Ring mit $I$ endlich erzeugt. In dieser Situation kann man eine „stetige $K$-Theorie von $A$“ ohne den Formalismus der verdichteten Mathematik definieren, und zwar als $\varprojlim K(A/I^n)$. Diese Definition spielt eine zentrale Rolle im klassischen Ansatz zur $K$-Theorie rigid-analytischer Räume (siehe z. B. \cite{Morrow}). Wenn $I$ außerdem schwach proregulär ist, was in der rigiden Geometrie immer der Fall ist, stimmen unsere und die klassische Definitionen nach einem Satz von Efimov überein (siehe \cite{Stetigkeitssatz}):

\begin{satz}[Efimov'scher Stetigkeitssatz]\label{Stetigkeitssatz}
Sei $A$ ein $I$-adisch vollständiger Ring, wobei $I$ ein endlich erzeugtes schwach proreguläres Ideal von $A$ ist. Dann liefern die natürlichen Abbildungen \[K_{\mathrm{stet}}(A)\rightarrow K_{\mathrm{stet}}(A/I^n)\cong K(A/I^n)\ (\text{bzw.}\ \mathbb{K}_{\mathrm{stet}}(A)\rightarrow \mathbb{K}_{\mathrm{stet}}(A/I^n)\cong \mathbb{K}(A/I^n))\] für $n\geq 0$ einen Isomorphismus
\[K_{\mathrm{stet}}(A)\xrightarrow{\sim} \varprojlim K(A/I^n)\ (\text{bzw.}\ \mathbb{K}_{\mathrm{stet}}(A)\xrightarrow{\sim} \varprojlim \mathbb{K}(A/I^n)).\]
\end{satz}

Sei nun $A$ ein tatescher Ring mit einem Definitionsring $A_0$ und einer Uniformisierenden $\varpi$. Im klassischen Ansatz definiert man die stetige $K$-Theorie von $A$ durch das kokartesische Produkt 
\begin{center}
\begin{tikzcd}
K(A_0) \arrow[d] \arrow[r] & K(A_0[1/\varpi])\arrow[d] \\
K_{\mathrm{stet}}(A_0) \arrow[r] & K'_{\mathrm{stet}}(A_0[1/\varpi]) \arrow[ul, phantom, "\scalebox{1.5}{$\ulcorner$}" , very near start, color=black]
\end{tikzcd}
\end{center}
Man globalisiert dann die so definierte $K$-Theorie tatescher Ringe auf rigid-analytische Räume durch den sogenannten \textit{pro-cdh-Abstieg}. Das offensichtliche konzeptionelle Problem des Ansatzes, und zwar die Tatsache, dass alle Konstruktionen etwas ad hoc sind, führt dazu, dass die Beweise im nicht noetherschen Falle schwer fassbar sind. Insbesondere ist es nicht klar, wie man den Abstieg für allgemeine analytische Räume nachweist. Mit unserer abstrakteren Definition können wir jedoch den optimalen Abstiegssatz beweisen. Zuerst zeigen wir, dass die Definitionen im affinoiden Fall übereinstimmen. Man betrachte folgendes Diagramm: 
\begin{center}
\begin{tikzcd}
\mathrm{Tor}(\varpi^\infty)\arrow[d]\arrow[r] & \mathcal{D}(A_0) \arrow[d] \arrow[r, "L"] & \mathcal{D}(A_0[1/\varpi]) \arrow[d] \\
\mathrm{Tor}^{\mathrm{nuk}}(\varpi^\infty)\arrow[r] & \mathrm{Nuk}(A_0) \arrow[r, "L'"] & \mathrm{Nuk}(A_0[1/\varpi])
\end{tikzcd}
\end{center}
Hierbei bezeichnen $L$ und $L'$ die kanonischen Lokalisierungsfunktoren und $\mathrm{Tor}(\varpi^\infty)$ und $\mathrm{Tor}^{\mathrm{nuk}}(\varpi^\infty)$ deren Kerne. Dann ist der Funktor $\mathrm{Tors}(\varpi^\infty)\rightarrow \mathrm{Tor}^{\mathrm{nuk}}(\varpi^\infty)$ nach Satz~\ref{torsion_ist_diskret} eine Äquivalenz, d. h., das Diagramm
\begin{equation}\label{kokartesisch}
\begin{tikzcd}
\mathbb{K}(A_0) \arrow[d] \arrow[r] & \mathbb{K}(A[1/\varpi]) \arrow[d] \\
\mathbb{K}_{\mathrm{stet}}(A_0) \arrow[r] & \mathbb{K}_{\mathrm{stet}} (A[1/\varpi])
\end{tikzcd}
\end{equation}
ist (ko-)kartesisch, weshalb $K_{\mathrm{stet}} (A[1/\varpi])\cong K'_{\mathrm{stet}} (A[1/\varpi])$ gilt. Übertragen wir nun Definition~\ref{definition_lokalisierende_schemata} mithilfe der Ergebnisse des vorhergehenden Abschnitts auf allgemeine analytische adische Räume:

\begin{definition}
Sei $X$ ein quasi-kompakter quasi-separierter analytischer adischer Raum und $\mathcal{D}$ eine stabile $\infty$-Kategorie. Unter einer \textit{lokalisierenden Invariante auf $X$ mit Werten in $\mathcal{D}$} verstehen wir eine Prägarbe 
\[\mathcal{F}_{\mathrm{stet}}:\{\textrm{quasi-kompakte offene Teilmengen von}\ X\}\rightarrow \mathcal{D},\ U\mapsto \mathcal{F}_{\mathrm{stet}}(\mathrm{Nuk}(U)),\]
wobei $\mathcal{F}$ eine lokalisierende Invariante mit Werten in $\mathcal{D}$ ist. Im Weiteren schreiben wir für $\mathcal{F}_{\mathrm{stet}}(\mathrm{Nuk}(U))$ einfach $\mathcal{F}_{\mathrm{stet}}(U)$.
\end{definition}

Für unsere Analyse der Abstiegseigenschaften der lokalisierenden Invarianten benötigen wir das folgende elementare Lemma, welches, genau wie im schematischen Falle, eine Schlüsselrolle in unseren weiteren Argumenten spielen wird.

\begin{lemma}\label{fasersequenz}
Sei $X$ ein quasi-kompakter quasi-separierter analytischer adischer Raum und $U$ eine quasi-kompakte offene Teilmenge von $X$. Es bezeichne $Y$ das Komplement von $U$ in $X$. Sei $Z$ eine abgeschlossene Teilmenge von $X$ mit quasi-kompaktem Komplement. Dann sind die Sequenzen \[\mathrm{Nuk}(X\ \mathrm{auf} \ Y)\rightarrow\mathrm{Nuk}(X)\rightarrow \mathrm{Nuk}(U),\]
\[\mathrm{Nuk}(X\ \mathrm{auf} \ Y\cap Z)\rightarrow\mathrm{Nuk}(X\ \mathrm{auf} \ Z)\rightarrow \mathrm{Nuk}(U\ \mathrm{auf} \ U\cap Z)\]
rechts spaltende Verdier-Sequenzen.
\end{lemma}

\begin{proof}
Es bezeichne $j$ die offene Einbettung $U\hookrightarrow X$. Wir betrachten folgende rechts spaltende Verdier-Sequenzen:

\begin{center}
\begin{tikzcd}
\mathcal{D}(X\ \mathrm{auf} \ Y) \arrow[r, shift left=0.5ex] & \mathcal{D}(X) \arrow[r, shift left=0.5ex,"j^\ast"]\arrow[l, shift left=0.5ex] & \mathcal{D}(U) \arrow[l, shift left=0.5ex,"j_\ast"] \\
\mathcal{D}(X\ \mathrm{auf} \ Y\cap Z)\arrow[r, shift left=0.5ex]\arrow[u,hook] & \mathcal{D}(X\ \mathrm{auf} \ Z)\arrow[r, shift left=0.5ex,"j^\ast"]\arrow[u,hook]\arrow[l, shift left=0.5ex] & \mathcal{D}(U\ \mathrm{auf} \ U\cap Z)\arrow[u,hook]\arrow[l, shift left=0.5ex,"j_\ast"]
\end{tikzcd}
\end{center}
Hierbei bezeichnet $j^\ast$ (bzw. $j_\ast$) den Rückzugsfunktor (bzw. den direkten Bildfunktor) entlang von $j$. Die nach rechts weisenden Pfeile auf der linken Seite des Diagramms stellen die natürlichen Einbettungen der Kerne der Funktoren $j^\ast$ und $j^\ast|_{\mathcal{D}(X\ \mathrm{auf} \ Z)}$ dar. Die nach links weisenden Pfeile entsprechen den Funktoren, die der Funktor $j_\ast$ gemäß Lemma~\ref{verdier-sequenzen} liefert. Wir behaupten, dass man die gewünschten Sequenzen erhält, wenn man alle betrachteten Funktoren auf nukleare Garbe einschränkt. Mit anderen Worten behaupten wir, dass die Funktoren $j_\ast$ und $j^\ast$ nukleare Objekte erhalten. Dies folgt aber sofort aus den Diskussionen nach Satz~\ref{lokaler_nuklearer_abstieg} und Definition~\ref{nukleare_garben}.
\end{proof}

Mithilfe dieses Lemmas zeigen wir nun, dass lokalisierende Invarianten in unserer Situation ebenfalls automatisch Nisnevich-Abstieg erfüllen:

\begin{satz}\label{nisnevcih-abstieg}
Sei $X$ ein quasi-kompakter quasi-separierter adischer Raum. Dann erfüllt jede lokalisierende Invariante $\mathcal{F}_{\mathrm{stet}}$ auf $X$ Nisnevich-Abstieg. Insbesondere bildet $\mathbb{K}_{\mathrm{stet}}(-)$ eine Nisnevich-Garbe auf $X$.
\end{satz}

\begin{proof}
Nach Korollar~\ref{excisiv} genügt es zu zeigen, dass für jede elementare Nisnevich'sche Überdeckung $U\hookrightarrow X\xleftarrow{f} V$ das Diagramm
\begin{center}
\begin{tikzcd}
\mathcal{F}_{\mathrm{stet}}(X) \arrow[r] \arrow[d] & \mathcal{F}_{\mathrm{stet}}(U) \arrow[d] \\
\mathcal{F}_{\mathrm{stet}}(V) \arrow[r] & \mathcal{F}_{\mathrm{stet}}(U \underset{X}{\times} V)
\end{tikzcd}\
\end{center}
kartesisch ist. Wir bezeichnen das Komplement von $U$ in $X$ mit $Z$. Die Zeilen des Diagramms
\begin{center}
\begin{tikzcd}
\mathcal{F}_{\mathrm{stet}}(X\ \mathrm{auf}\ Z) \arrow[r] \arrow[d] & \mathcal{F}_{\mathrm{stet}}(X) \arrow[r] \arrow[d] & \mathcal{F}_{\mathrm{stet}}(U) \arrow[d] \\
\mathcal{F}_{\mathrm{stet}}(U \underset{X}{\times} V\ \mathrm{auf}\ f^{-1}(Z)) \arrow[r] & \mathcal{F}_{\mathrm{stet}}(V) \arrow[r] & \mathcal{F}_{\mathrm{stet}}(U \underset{X}{\times} V) 
\end{tikzcd}\
\end{center}
sind nach Lemma~\ref{fasersequenz} Fasersequenzen. Deswegen reicht es aus zu zeigen, dass der natürliche Funktor 
\[\mathrm{Nuk} (X\ \mathrm{auf}\ Z) \rightarrow \mathrm{Nuk} (U \underset{X}{\times} V\ \mathrm{auf}\ f^{-1}(Z))\]
eine Äquivalenz ist. Dies folgt aber sofort aus dem étalen Abstiegssatz für nukleare Garben (Korollar~\ref{étaler_nuklearer_abstieg}).
\end{proof}

\begin{bemerkung}
Man kann in der Situation von Satz~\ref{nisnevcih-abstieg} mithilfe des zweiten Teils des Lemmas~\ref{fasersequenz} nachweisen, dass die Prägarbe $U\mapsto F_\mathrm{stet}(\mathrm{Nuk}(U\ \mathrm{auf}\ Z'\cap U))$ ebenso eine Nisnevich-Garbe bildet, wobei $Z'$ eine abgeschlossene Teilmenge von $X$ mit quasi-kompaktem Komplement ist.  
\end{bemerkung}

Im Gegensatz dazu gilt der étale Abstieg lokalisierender Invarianten auf analytischen adischen Räumen im Allgemeinen nicht. Wie im Fall der Schemata wenden wir die chromatische Lokalisierung an, um dieses Problem zu beheben:

\begin{satz}\label{hauptabstiegssatz}
Sei $X$ ein quasi-kompakter quasi-separierter analytischer adischer Raum endlicher Krull-Dimension und $p$ eine feste Primzahl. Angenommen, die virtuellen $p$-lokalen kohomologischen Dimensionen der galoisschen Siten von $X$ sind nach oben beschränkt. Dann bildet jede lokalisierende Invariante $\mathcal{F}_{\mathrm{stet}}$ mit Werten in $L_n^f$-lokalen Spektren eine étale Hypergarbe auf $X$.
\end{satz}

\begin{proof}
Sei $x$ ein Punkt von $X$. Man betrachte die Prägarbe $x^\ast\mathcal{F}_{\mathrm{stet}}$ auf dem galoisschen Situs $\mathcal{T}_x$, die durch die Einschränkung auf $\mathcal{T}_x$ des Nisnevich-Rückzugs von $\mathcal{F}$ entlang der kanonischen Abbildung $\iota_x:\Spa (\kappa_h(x),\kappa_h^+(x))\rightarrow X$ gegeben ist (siehe die Diskussion unter Definition~\ref{stetige_g_mengen}). Ein Morphismus in $\mathcal{T}_x$ entspricht einer endlichen étalen Abbildung $f: A'\rightarrow A''$ von endlichen étalen Algebren über $A$, wobei $A\cong \underset{x\in U_i}{\colim}\, A_i$ der nicht vervollständigte Nisnevich-Halm der Strukturgarbe in $x$ ist. Hierbei durchläuft $U_i$ die offenen affinoiden Nisnevich-Umgebungen von $x$ und mit $A_i$ wird der Ring $\mathcal{O}_X(U_i)$ bezeichnet. Wir schreiben die Abbildung $f$ als Kolimes $\colim(f_i:A_i'\rightarrow A_i'')$ von endlichen étalen Abbildungen zwischen endlichen étalen Algebren über $A_i$. Mit der üblichen Argumention dürfen wir zudem annehmen, dass die Abbildung $f$ von $f_0$ induziert wird, d. h., $f_i=f_0\underset{A_0'}{\otimes} A_i'$ für jede Umgebung $U_i$.

Es bezeichne $F$ die lokalisierende Invariante auf $\mathrm{Cat}_{\infty}^{\mathrm{perf}}$ mit Werten in $\mathrm{Sp}$, die der lokalisierenden Invariante $\mathcal{F}_{\mathrm{stet}}$ auf $X$ entspricht, und $\mathrm{Cat}_{\infty,A_0'}^{\mathrm{perf}}$ die $\infty$-Kategorie $\mathrm{Mod}_{\mathrm{Perf}(A_0')}(\mathrm{Cat}_{\infty}^{\mathrm{perf}})$. Die Prägarbe $x^\ast\mathcal{F}_\mathrm{stet}$ hat auf $A'$ (bzw. $A''$) den Wert $\colim F_\mathrm{stet}(\mathrm{Nuk}(A_i'))$ (bzw. $\colim F_\mathrm{stet}(\mathrm{Nuk}(A_i''))$). Wie man leicht nachprüft, definiert der Funktor 
\[\mathrm{Cat}_{\infty,A_0'}^{\mathrm{perf}}\rightarrow \mathrm{Sp},\ \mathcal{C}\mapsto \underset{i,\ x\in U_i}{\colim}\, F_\mathrm{stet}(\mathcal{C}\underset{\mathrm{Perf}(A_0')}{\otimes}\mathrm{Nuk}(A_i'))\]
eine lokalisierende Invariante mit Werten in $L_n^f$-lokalen Spektren. Unter Anwendung von \cite[Theorem 5.1 und Proposition 5.4]{CMNN20} auf die endliche étale Abbildung $f_0:A_0'\rightarrow A_0''$ sieht man, dass das Spektrum $\colim F_\mathrm{stet}(\mathrm{Nuk}(A_i'))$ zur Totalisierung
\[\mathrm{Tot}\Bigl(\colim F_\mathrm{stet}(\mathrm{Perf}(A_0'')\underset{\mathrm{Perf}(A_0')}{\otimes}\mathrm{Nuk}(A_i')) \rightrightarrows \colim F_\mathrm{stet}(\mathrm{Perf}(A_0''\underset{A_0'}{\otimes}A_0'')\underset{\mathrm{Perf}(A_0')}{\otimes}\mathrm{Nuk}(A_i'))\mathrel{\substack{\textstyle\rightarrow\\[-0.6ex]
                      \textstyle\rightarrow \\[-0.6ex]
                      \textstyle\rightarrow}} \dots\Bigr)\]
isomorph ist. Für jede endliche étale Algebra $B$ über $A_0$ gilt $\mathrm{Nuk}(B\underset{A_0'}{\otimes} A_i'') \cong \mathrm{Perf}(B)\underset{\mathrm{Perf}(A_0')}{\otimes}\mathrm{Nuk}(A_i')$, was sich leicht mithilfe des Satzes~\ref{perfekt-nuklear} zeigen lässt. Die Prägarbe $x^\ast \mathcal{F}_\mathrm{stet}$ bildet also eine Garbe auf $\mathcal{T}_x$. Die gerade durchgeführte Argumentation zeigt außerdem, dass $\mathcal{F}_{\mathrm{stet}}$ galoisschen Abstieg auf $X$ erfüllt. D. h., die Nisnevich-Hypergarbe $\mathcal{F}_{\mathrm{stet}}$ bildet nach Satz~\ref{étale=Nisnevich+Galois} eine étale Garbe auf $X$. 

Sei $B$ eine endliche étale Algebra über $A$. Wir schreiben diese als Kolimes $\colim B_i$ von endlichen étalen Algebren über $A_i$. Nach \cite[Theorem 7.14]{CM21} ist die Prägarbe $G:B\mapsto \colim F(\mathrm{Perf}(B_i))$ auf $\mathcal{T}_x$ eine hypervollständige Garbe. Man sieht unmittelbar, dass die étale Garbe $x^\ast \mathcal{F}_\mathrm{stet}$ ein Modul über $G$ ist. Da die Hypervervollständigung auf $\mathcal{T}_x$ nach Satz~\ref{tensor-lokalisierung-proendliche-gruppen} eine Tensor-Lokalisierung ist, ist $x^\ast \mathcal{F}_\mathrm{stet}$ also hypervollständig (siehe Lemma~\ref{tensor-lokalisierung}). Die Hypervollständigkeit der étalen Garbe $\mathcal{F}_{\mathrm{stet}}$ ergibt sich deshalb aus dem Lokal-Global-Prinzip für Hypervollstädigkeit (Satz~\ref{lokal-global}) und Satz~\ref{kriterium}.
\end{proof}

Wir wenden uns jetzt einer spezielleren Situation zu. Für den Rest dieses Abschnittes sei $X$ ein quasi-kompakter quasi-separierter analytischer adischer Raum über $\Spa (\mathbb{Z}[1/p],\mathbb{Z}[1/p])$\footnote{Mit anderen Worten ist $p$ invertierbar in $\mathcal{O}^+_S$.} von endlicher Krull-Dimension. Außerdem nehmen wir an, dass die virtuellen $p$-lokalen kohomologischen Dimensionen der galoisschen Siten von $X$ nach oben beschränkt sind. Wir folgen dem Ansatz von Thomason-Trobaugh im schematischen Fall in seiner modernen Form (siehe \cite[Abschnitt 3]{K1}) und beschreiben explizit die étale Hypergarbe $L_{K(1)}K_{\mathrm{stet}}(-)$ auf $X$\footnote{Nach der Teleskopvermutung für die Höhe 1 ist sie isomorph zur étalen Hypergarbe $L_{T(1)}K_{\mathrm{stet}}(-)$.}, wobei $K(1)$ die erste $p$-lokale Morava'sche $K$-Theorie bezeichnet. Es bezeichne $B\mathbb{Z}_p^\times$ den $\infty$-Topos $\mathrm{Sh}(\mathcal{T}_{\mathbb{Z}_p^\times})$ (siehe Definition~\ref{stetige_g_mengen}), wobei $\mathbb{Z}_p^\times$ mit der $p$-adischen Topologie versehen wird. Wir betrachten die volle Unterkategorie $\mathrm{Zykl}_p$ des kleinen étalen Situs $\mathrm{\Acute{E}t}_X$ von $X$, die aus den Objekten der Form $\underset{i=1}{\overset{k}{\coprod}}\Spec \mathbb{Z}[1/p,\zeta_{p^{n_i}}]$ besteht. Die Einbettung $\mathrm{Zykl}_p\subset \mathrm{\Acute{E}t}_X$ und der Morphismus $X\rightarrow \Spec \mathbb{Z}[1/p]$ von lokal geringten Räumen liefern die Morphismen von $\infty$-Topoi
\begin{center}
\begin{tikzcd}
X_{\Acute{e}t} \arrow[rr, bend left, "\pi_X"] \arrow[r] &(\Spec \mathbb{Z}[1/p])_{\Acute{e}t}\arrow[r,"\pi"] & B\mathbb{Z}_p^\times.\\
\end{tikzcd}
\end{center}
Es bezeichne $KU^\wedge_{p}$ die in \cite[Lemma 3.8]{K1} konstruierte $p$-vollständige hypervollständige Garbe von Spektren auf $\mathcal{T}_{\mathbb{Z}_p^\times}$. Dann definieren wir die étale Hypergarbe $KU^\wedge_{p,X}$ auf $X$ als die $p$-Vervollständigung der Hypervervollständigung des Rückzugs $\pi_X^\ast(KU^\wedge_{p})$. Wie man leicht einsehen kann, folgt die Hypervollständigkeit von $KU^\wedge_{p,X}$ aus Lemma~\ref{grundeigenschaften_der_hypervollständigkeit}.

Wir erinnern kurz an die Konstruktion des Morphismus $\pi^\ast(KU^\wedge_{p})\rightarrow L_{K(1)}K(-)$ von étalen Garben auf $\Spec \mathbb{Z}[1/p]$; für weitere Details verweisen wir auf \cite[Konstruktion 3.7 und Theorem 3.9]{K1}. Aus der Abbildung $\mu_{p^\infty}\rightarrow K_1(\mathbb{Z}[\zeta_{p^\infty}])$ ergibt sich der Morphismus von Spektren $\mathbb{S}[B\mu_{p^\infty}]\rightarrow K(\mathbb{Z}[\zeta_{p^\infty}])$, welcher uns einen Morphismus $\mathbb{S}[B\mu_{p^\infty}]_p^\wedge\rightarrow L_{K(1)}K(\mathbb{Z}[\zeta_{p^\infty}])$ liefert. Es bezeichne $\beta\in \pi_2(\mathbb{S}[B\mu_{p^\infty}])$ das Bott'sche Element. Da  das Bild von $\beta$ in $\pi_2(L_{K(1)}K(\mathbb{Z}[\zeta_{p^\infty}]))$ invertierbar ist (siehe \cite[Proposition 3.5]{K1}), induziert der obige Morphismus einen Morphismus $(\mathbb{S}[ B\mu_{p^\infty}]_p^\wedge[\beta^{-1}])_p^\wedge\xrightarrow[]{} L_{K(1)}K(\mathbb{Z}[\zeta_{p^\infty}])$. Nach dem Satz von Snaith gibt es einen Isomorphismus von Spektren $(\mathbb{S}[B\mu_{p^\infty}]_p^\wedge[\beta^{-1}])_p^\wedge\xrightarrow[]{\sim} KU_p^\wedge$\footnote{Hierbei bezeichnet $KU_p^\wedge$ nicht die oben erwähnte Garbe auf $\mathcal{T}_{\mathbb{Z}_p^\times}$, sondern das $p$-vollständige komplexe $K$-Theorie-Spektrum.}, woraus sich der gewünschte Morphismus $\pi^\ast(KU^\wedge_{p})\rightarrow L_{K(1)}K(-)$ von étalen Garben auf $\mathbb{Z}[1/p]$ ergibt.

Sei nun $A$ ein tatescher Ring. Die natürliche Einbettung $\mathcal{D}(A)\hookrightarrow \mathrm{Nuk}(A)$ induziert einen Morphismus $K(A)\rightarrow K_{\mathrm{stet}}(A)$. Dieser liefert uns zusammen mit dem Morphismus $\pi^\ast(KU^\wedge_{p})\rightarrow L_{K(1)}K(-)$ einen Morphismus $\mathrm{rgl}^{-1}:KU^\wedge_{p,X}\rightarrow L_{K(1)}K_{\mathrm{stet}}(-)$ von étalen Hypergarben auf $X$. 

\begin{satz}
Der Morphismus $\mathrm{rgl}^{-1}:KU^\wedge_{p,X}\rightarrow L_{K(1)}K_{\mathrm{stet}}(-)$ ist ein Isomorphismus.
\end{satz}

\begin{proof}
Da die Garben $KU^\wedge_{p,X}$ und $L_{K(1)}K(-)$ hypervollständig sind, genügt es, die Aussage halmweise nachzuweisen. Außerdem reicht es wegen der $p$-Vollständigkeit aus zu zeigen, dass die Halme von $KU^\wedge_{p,X}$ und $L_{K(1)}K(-)$ nach der $p$-Vervollständigung übereinstimmen. Sei $x$ ein Punkt von $X$ und betrachte den geometrischen Punkt $\Bar{\iota}_x:\Spa (\Bar{\kappa}(x),\Bar{\kappa}^+(x))\rightarrow X$. Wie man leicht nachprüft, ist die $p$-Vervollständigung des étalen Halms $(\Bar{\iota}_x)_{\mathrm{\acute{e}t}}^*(KU^\wedge_{p,X})$ isomorph zum Spektrum $KU^\wedge_{p}$. Da der Funktor $L_{K(1)}$ äquivalent zur Verkettung $(-)^\wedge_p\circ L_{KU}$ ist und die Bousfield-Lokalisierung $L_{KU}:\mathrm{Sp}\rightarrow \mathrm{Sp}$ mit Kolimites vertauscht, gilt
\begin{equation}\tag{$\ast$}
\bigl(\underset{x\in U}{\colim}\, L_{K(1)}K_{}(\mathcal{O}_X(U))\bigr)^\wedge_p\cong L_{K(1)}\bigl(\underset{x\in U}{\colim}\, K_{}(\mathcal{O}_X(U))\bigr),
\end{equation}
wobei $U$ die affinoiden étalen Umgebungen von $x$ durchläuft. 

Sei nun $V$ eine tatesche affinoide Umgebung von $x$. Es bezeichne $A_V$ den tateschen Ring $\mathcal{O}_X(V)$ und sei $\varpi$ eine Uniformisierende von $A_V$. Nach dem Efimov'schen Stetigkeitssatz (Satz~\ref{Stetigkeitssatz}) ist die stetige $K$-Theorie $K_{\mathrm{stet}}(A_{V,0})$ isomorph zu $\varprojlim K(A_{V,0}/\varpi^n)$, wobei $A_{V,0}$ ein Definitionsring von $A_V$ mit $1/p,\varpi\in A_V$ ist. Es ergibt sich aus dem Gabber'schen Starrheitssatz \[L_{K(1)}K(A_{V,0})\xrightarrow{\sim} L_{K(1)}K(A_{V,0}/\varpi)\xleftarrow{\sim} L_{K(1)}K(A_{V,0}/\varpi^n)\xleftarrow{\sim} L_{K(1)}K_{\mathrm{stet}}(A_{V,0}).\]
Es gilt also $L_{K(1)}K(A_{V})\xrightarrow{\sim}L_{K(1)}K_{\mathrm{stet}}(A_{V})$, denn Diagramm~\ref{kokartesisch} ist kokartesich. Unter Verwendung des Isomorphismus~($\ast$) folgt daher
\[\bigl(\underset{x\in U}{\colim}\, L_{K(1)}K_{\mathrm{stet}}(U)\bigr)^\wedge_p\cong L_{K(1)}K(\underset{x\in U}{\colim}\,\mathcal{O}_X(U)),\]
wobei $U$ die tateschen affinoiden étalen Umgebungen von $x$ durchläuft. Unter erneuter Anwendung des Gabber'schen Starrheitssatzes sieht man, dass $L_{K(1)}K(\underset{x\in U}{\colim}\,\mathcal{O}_X(U))$ isomorph zu $L_{K(1)}K(k)$ ist, wobei $k=(\underset{x\in U}{\colim}\, \mathcal{O}_X(U))^\wedge$ die Vervollständigung des étalen Halms der Strukturgarbe in $x$ bezeichnet. Der Satz folgt nun aus der Berechnung der $p$-adischen $K$-Theorie separabel abgeschlossener Körper nach Suslin, siehe \cite[Theorem 3.9]{K1}.
\end{proof}
\newpage
\section{Grothendieck-Riemann-Roch}

Entgegen Grothendiecks Willen ändern wir unseren Kurs nicht, sondern lassen uns von unserem Wissens- und Entdeckungsdrang immer tiefer ins logische Delirium führen: Das Ziel dieses Abschnittes ist es, den Satz von Grothendieck-Riemann-Roch für analytische adische Räume zu formulieren und zu beweisen. Im gesamten Abschnitt nehmen wir an, dass alle betrachteten analytischen adischen Räume quasi-kompakt und quasi-separiert und somit alle Morphismen quasi-kompakt sind. Im Folgenden sei $S$ stets ein fester noetherscher analytischer adischer Raum über $\Spa (\mathbb{Z}[1/p],\mathbb{Z}[1/p])$ für eine Primzahl $p$. Außerdem nehmen wir an, dass er die Bedingungen des Satzes~\ref{hauptabstiegssatz} erfüllt. Wir betrachten eine Abbildung $f:X\xrightarrow{} Y$ zwischen noetherschen analytischen adischen Räumen, glatt und eigentlich über $S$.

\begin{satz}[Satz von Grothendieck-Riemann-Roch]
Das folgende Diagramm von étalen Garben von Spektren auf $S$ ist kommutativ: 
\begin{equation*}
\begin{tikzcd}
\mathbb{K}_{X,S}(-)\rar["f_!"]\dar["\operatorname{Td}(X)\operatorname{ch}(-)"'] & \mathbb{K}_{Y,S}(-)\dar["\operatorname{ch}(-)\operatorname{Td}(Y)"]\\
\mathcal{H}_{\mathrm{\Acute{e}t}}(X,\mathbb{Q}_p)\rar["f_*"] & \mathcal{H}_{\mathrm{\Acute{e}t}}(Y,\mathbb{Q}_p)
\end{tikzcd}
\end{equation*}

\begin{center}
\vspace{-1.73cm}
\hspace{-0.6cm}
\resizebox {9.25cm} {!} {\input{die_Teufel}}
\end{center}
\end{satz}

Bevor wir uns dem Beweis zuwenden können, müssen wir alle Notationen und relevanten Begriffe erklären. Wir erinnern zunächst, der Vollständigkeit halber, an die Definitionen der noetherschen Bedingung, der Eigentlichkeit und der Glattheit in der rigiden Geometrie.

\begin{definition}[{\cite[Abschnitt 1.1]{Huber}}]
\begin{enumerate}[label=(\roman*)]
\item[] 
\item Sei $A$ ein analytischer Huber-Ring. Es heißt $A$ \textit{strikt noethersch}, wenn für jedes $n\geq 0$ die Tate-Algebra $A\langle T_1,\dots,T_n\rangle$ noethersch ist.
\item Ein analytischer adischer Raum ${X}$ heißt \textit{noethersch}, wenn es eine affinoide Überdeckung ${X}=\underset{i\in I}{\bigcup} \Spa (A_i,A_i^+)$ mit $A_i$ strikt noethersch gibt.
\end{enumerate}
\end{definition}

\begin{definition}[{\cite[Definitionen 1.2.1, 1.3.1 und 1.3.2]{Huber}}]
Sei ${f}:{X}\xrightarrow{} {Y}$ eine Abbildung zwischen analytischen adischen Räumen.
\begin{enumerate}[label=(\roman*)]
\item Eine Abbildung $r:A\xrightarrow{} B$ zwischen vollständigen analytischen Huber-Ringen heißt \textit{topologisch von endlichem Typ}, wenn sie über eine stetige offene surjektive Abbildung $A\langle T_1,\dots,T_n\rangle \xrightarrow{} B$ für ein $n\geq 0$ faktorisiert.
\item Die Abbildung ${f}$ heißt \textit{schwach von endlichem Typ}, wenn es für jeden Punkt $x\in {X}$ eine affinoide Teilmenge $x\in U\subset {X}$ und eine affinoide Teilmenge $V\subset {Y}$ mit $f(U)\subset V$ gibt, sodass die Abbildung $\mathcal{O}_{{Y}}(V)\xrightarrow{} \mathcal{O}_{{X}}(U)$ topologisch von endlichem Typ ist.
\item Die Abbildung ${f}$ heißt \textit{${}^{+}$schwach von endlichem Typ}, wenn es für jeden Punkt $x\in {X}$ eine affinoide Teilmenge $x\in U\subset {X}$, eine affinoide Teilmenge $V\subset {Y}$ mit $f(U)\subset V$ und eine endliche Menge $E\subset \mathcal{O}_{{X}}^+(U)$ gibt, sodass die Abbildung $\mathcal{O}_{{Y}}(V)\xrightarrow{} \mathcal{O}_{{X}}(U)$ topologisch von endlichem Typ und der Ring $\mathcal{O}_{{X}}^+(U)$ mit dem ganzen Abschluss von $\mathcal{O}_{{Y}}^+[E\cup \mathcal{O}_{{X}}(U)^{\circ\circ}]$ übereinstimmt.
\item Die Abbildung ${f}$ heißt \textit{separiert}, wenn sie schwach von endlichem Typ ist und die Diagonale $\Delta({X})$ abgeschlossen in ${X} \times_Y {X}$ ist.
\item Die Abbildung ${f}$ heißt \textit{universell abgeschlossen}, wenn sie schwach von endlichem Typ ist und der Basiswechsel ${Y}'\times_Y {X}\xrightarrow{} {Y}'$ für jeden Morphismus ${Y}'\xrightarrow{} {Y}$ von adischen Räumen abgeschlossen ist.
\item Die Abbildung ${f}$ heißt \textit{eigentlich}, wenn sie ${}^+$schwach von endlichem Typ, separiert und universell abgeschlossen ist.
\end{enumerate}
\end{definition}

\begin{definition}[{\cite[Korollar 1.6.10]{Huber}}]\label{glattheit}
Sei ${f}:{X}\xrightarrow{} {Y}$ eine Abbildung zwischen noetherschen adischen Räumen. Es heißt ${f}$ \textit{glatt}, wenn es für jede affinoide offene Teilmenge $\Spa (A,A^+)\subset {Y}$ eine offene Teilmenge $U\subset {X}$ mit $f(U)\subset \Spa (A,A^+)$ gibt, sodass sich die Einschränkung $f|_U:U\xrightarrow{} \Spa (A,A^+)$ als Verkettung

\begin{center}
\begin{tikzcd}
U \arrow[dr,"{g}",end anchor={north west}]\arrow[dd,"f|_U"] & \\
& \Spa (A\langle T_1,\dots, T_n\rangle,A^+\langle T_1,\dots, T_n\rangle ) \arrow[dl,"{h}",start anchor={south west}] \\
\Spa (A,A^+) &
\end{tikzcd}
\end{center}
für ein $n\geq 0$ schreiben lässt, wobei ${g}$ étale und ${h}$ die natürliche Projektion ist. 
\end{definition}

Sei $\mathbb{Z}/p^{k}\mathbb{Z}(n)$ die diskrete abelsche Gruppe $\mathbb{Z}/p^{k}\mathbb{Z}$ zusammen mit der stetigen $\mathbb{Z}_p^\times$-Wirkung, die durch $(\alpha,x)\in (\mathbb{Z}_p^\times,\mathbb{Z}/p^{k}\mathbb{Z})\mapsto \alpha^n\cdot x$ gegeben ist. Für einen analytischen adischen Raum $X$ über $\Spec \mathbb{Z}[1/p]$ bezeichne $\underline{\mathbb{Z}/p^{k}\mathbb{Z}}$ den Rückzug von $\mathbb{Z}/p^{k}\mathbb{Z}(n)$ nach $X_{\Acute{e}t}$, wobei $\mathbb{Z}/p^{k}\mathbb{Z}(n)$ als Garbe auf $\mathcal{T}_{\mathbb{Z}_p^\times}$ angesehen wird (siehe die Diskussion nach Satz~\ref{hauptabstiegssatz}). Dann definieren wir die étale Hypergarbe $\underline{\mathbb{Z}_p}(n)$ (bzw. $\underline{\mathbb{Q}_p}(n)$) auf $X$ als die étale Hypergarbe $\underset{k\in \mathbb{N}}{\varprojlim}\, \underline{\mathbb{Z}/p^{k}\mathbb{Z}}(n)$ (bzw. $(\underset{k\in \mathbb{N}}{\varprojlim}\, \underline{\mathbb{Z}/p^{k}\mathbb{Z}}(n))[1/p]$), wobei der Limes in der Kategorie der étalen Garben von Spektren gebildet wird. 

Wir definieren jetzt den \textit{direkten Bildmorphismus} auf dem Niveau der $K$-Theorie für einen eigentlichen lokal vollständigen Durchschnitt ${f}:{X}\xrightarrow{} {Y}$ zwischen noetherschen analytischen adischen Räumen. Wir verwenden hierfür den in \cite[Vorlesung IX]{6-Funktor} entwickelten Sechs-Funktor-Formalismus für feste Garben auf analytischen adischen Räumen. Es bezeichne $\mathcal{D}_{\blacksquare}({X})$ (bzw. $\mathcal{D}_{\blacksquare}({Y})$) die $\infty$-Kategorie der festen Garben auf $X$ (bzw. $Y$). Als Erstes bemerken wir, dass der Funktor ${f}_!: \mathcal{D}_{\blacksquare}({X})\xrightarrow{} \mathcal{D}_{\blacksquare}({Y})$ Nuklearität erhält. In der Tat, sei $N\in \mathcal{D}_{\blacksquare}({X})$ eine nukleare Garbe auf ${X}$ und $M$ ein kompaktes Objekt in $\mathcal{D}_{\blacksquare}({Y})$. Für ein beliebiges Objekt $L\in \mathcal{D}_{\blacksquare}({Y})$ führen wir folgende Rechnung durch:
\[\mathrm{Hom}(L,{f}_!N\otimes M^\vee)\cong \mathrm{Hom}(L,{f}_!(N\otimes {f}^{\ast}(M^\vee)))\cong \mathrm{Hom}({f}^{\ast} L,N\otimes {f}^{\ast}(M^\vee)) \cong \mathrm{Hom}({f}^{\ast} L,\underline{\mathrm{Hom}}({f}^{\ast} M,N))\]\[\cong\mathrm{Hom}({f}^{\ast} (L\otimes M),N)\cong \mathrm{Hom}(L\otimes M,{f}_! N)\cong \mathrm{Hom}(L, \underline{\mathrm{Hom}}(M,{f}_! N)). \]
Hierbei bezeichnet $M^\vee$ das Dual $\underline{\mathrm{Hom}}(M,\mathcal{O}_{{Y}})$ und wir verwenden implizit die Äquivalenz ${f}_!\xrightarrow[]{\sim}{f}_\ast$ und den Isomorphismus ${f}^{\ast}(M^\vee)\cong ({f}^{\ast} M)^\vee$. Wir bemerken nun, dass der Funktor ${f}^!: \mathcal{D}_{\blacksquare}({X})\xrightarrow{} \mathcal{D}_{\blacksquare}({Y})$ ebenfalls Nuklearität erhält und zudem mit Kolimites vertauscht. In der Tat ist er bis auf Twist um eine invertierbare Garbe zum Funktor $f^\ast$ äquivalent. Mit anderen Worten liefert der Funktor $f_!$ durch Einschränken einen dualisierbaren Funktor $f_!: \mathrm{Nuk}({X})\xrightarrow{} \mathrm{Nuk}({Y})$ und induziert somit einen direkten Bildmorphismus $f_!: \mathbb{K}_{\mathrm{stet}}(X)\xrightarrow{} \mathbb{K}_{\mathrm{stet}}(Y)$. Es bezeichne $\mathbb{K}_X(-)$ (bzw. $\mathbb{K}_Y(-)$) die durch die stetige nicht konnektive $K$-theorie definierte étale Prägarbe auf $X$ (bzw. auf $Y$). Aus unseren Überlegungen ergibt sich also ein Morphismus $f_!: f_\ast \mathbb{K}_X(-)\xrightarrow{} \mathbb{K}_Y(-)$ von étalen Prägarben auf $Y$. Wir nehmen jetzt an, dass der Morphismus $f$ ein Morphismus von Räumen über $S$ ist, und bezeichnen die Strukturabbildungen $X\xrightarrow{} S$ und $Y\xrightarrow{} S$ mit ${p}_X$ und ${p}_Y$. In diesem Fall setzen wir $\mathbb{K}_{X,S}(-)={p}_{X\ast}\mathbb{K}_X(-)$ und $\mathbb{K}_{Y,S}(-)={p}_{Y\ast}\mathbb{K}_Y(-)$ und bemerken, dass der Morphismus $f_!: f_\ast \mathbb{K}_X(-)\xrightarrow{} \mathbb{K}_Y(-)$ in offensichtlicher Weise einen Morphismus $\mathbb{K}_{X,S}(-)\xrightarrow{} \mathbb{K}_{Y,S}(-)$ induziert, den wir ebenfalls mit $f_!$ bezeichnen.

Sei nun $f:X\xrightarrow{} Y$ eine eigentliche Abbildung zwischen noetherschen analytischen adischen Räumen, glatt über $S$ der relativen Dimensionen $d_X$ und $d_Y$. Wir bezeichnen wiederum die Strukturabbildung $X\xrightarrow{} S$ (bzw. $Y\xrightarrow{} S$) mit ${p}_X$ (bzw. ${p}_Y$). Wir definieren die relative étale Kohomologie $\mathcal{H}_{\mathrm{\Acute{e}t}}(X,\mathbb{Q}_p)$ (bzw. $\mathcal{H}_{\mathrm{\Acute{e}t}}(Y,\mathbb{Q}_p)$) als das direkte Bild ${p}_{X\ast} (\underset{n\in\mathbb{Z}}{\bigoplus}\underline{\mathbb{Q}_p}(n)[2n])$ (bzw. ${p}_{Y\ast} (\underset{n\in\mathbb{Z}}{\bigoplus}\underline{\mathbb{Q}_p}(n)[2n])$). Wie im klassischen Fall benutzen wir zur Definition des direkten Bildmorphismus $f_\ast$ auf der étalen Kohomologie $\mathcal{H}_{\mathrm{\Acute{e}t}}(-,\mathbb{Q}_p)$ die Poincaré-Dualität:

\begin{satz}[Huber]
Sei ${p}_{{X}}:{X}\xrightarrow{} S$ ein noetherscher analytischer adischer Raum, glatt über $S$ der relativen Dimension $d_X$. Dann gilt für jedes $k\geq 0$
\[{p}_{{X}}^!(\underline{\mathbb{Z}/p^k\mathbb{Z}})\cong \underline{\mathbb{Z}/p^k\mathbb{Z}}(d_X)[2d_X].\]
\end{satz}

\begin{proof}
Für den Beweis verweisen wir auf {\cite[Theorem 6.1.6]{Poincare}}.
\end{proof}

Sei $k>0$ eine natürliche Zahl. Für jedes $n\in\mathbb{Z}$ erhalten wir durch Anwenden der Koeinheit $f_!f^!\xrightarrow{}\mathrm{id}$ auf den Identitätsmorphismus ${p}_Y^!(\underline{\mathbb{Z}/p^k\mathbb{Z}}(n)[2n]) \xrightarrow[]{\sim} {p}_Y^!(\underline{\mathbb{Z}/p^k\mathbb{Z}}(n)[2n])$ von étalen Garben auf $Y$ einen Morphismus $f_\ast {p}_X^!(\underline{\mathbb{Z}/p^k\mathbb{Z}}(n)[2n])\xrightarrow[]{} {p}_Y^!(\underline{\mathbb{Z}/p^n\mathbb{Z}}(n)[2n])$. Dann definieren wir den direkten Bildmorphismus $f_\ast: \mathcal{H}_{\mathrm{\Acute{e}t}}(X,\mathbb{Q}_p)\xrightarrow{} \mathcal{H}_{\mathrm{\Acute{e}t}}(Y,\mathbb{Q}_p)$ über die Diagramme 

\begin{center}
\begin{tikzcd}
{p}_{X\ast} {p}_X^!(\underline{\mathbb{Z}/p^k\mathbb{Z}}(n)[2n]) \arrow[r,]\arrow[d,"\sim" labr] &{p}_{Y\ast}{p}_Y^!(\underline{\mathbb{Z}/p^k\mathbb{Z}}(n)[2n]) \arrow[d,"\sim" labr] \\
{p}_{X\ast} {p}_X^\ast(\underline{\mathbb{Z}/p^k\mathbb{Z}}(n+d_X)[2n+2d_X]) \arrow[r] & {p}_{Y\ast}{p}_Y^\ast(\underline{\mathbb{Z}/p^k\mathbb{Z}}(n+d_Y)[2n+2d_Y])
\end{tikzcd}
\end{center}
für alle $k>0$, $n\in\mathbb{Z}$.

Wir kommen nun zur Definition des \textit{Chern-Charakters} und der \textit{Todd-Klasse}.

\begin{definition}[{\cite[Sequenz 2.2.6]{Huber}}]\label{Kummer}
Sei ${X}$ ein analytischer adischer Raum über $\Spec \mathbb{Z}[1/p]$. Es bezeichne $\mathbb{G}_{m,{X}}$ die étale Garbe $U\mapsto \mathcal{O}^\times_{{X}}(U)$ abelscher Gruppen auf ${X}$. Unter der \textit{Kummer'schen Sequenz} verstehen wir die exakte Sequenz 
\[1\xrightarrow{}\mu_{p^n}\xrightarrow{}\mathbb{G}_{m,{X}}\xrightarrow{\cdot p^n}\mathbb{G}_{m,{X}}\xrightarrow{} 1\]
von étalen Garben abelscher Gruppen auf ${X}$.
\end{definition}

Unter der Identifizierung $\mu_{p^n}= \underline{\mathbb{Z}/p^n\mathbb{Z}}(1)$ für alle $n\geq 0$ erhalten wir aus dieser Sequenz einen Morphismus $\mathbb{G}_{m,{X}}[1]\xrightarrow{}\underline{\mathbb{Z}_p}(1)[2]$, welchen wir zur Definition der \textit{ersten Chern-Klasse} benutzen.

\begin{lemma}[{\cite[Satz 2.2.7]{Huber}}]
Sei ${X}$ ein noetherscher analytischer adischer Raum. Dann ist die étale Kohomologiegruppe $H_{\Acute{e}t}^1({X},\mathbb{G}_{m,{X}})$ isomorph zur Picard-Gruppe von ${X}$ (bezüglich der analytischen Topologie).
\end{lemma}

\begin{definition}\label{erste chernklasse}
Sei ${X}$ ein noetherscher analytischer adischer Raum über $\Spec \mathbb{Z}[1/p]$ und $\mathcal{L}$ ein Geradenbündel auf ${X}$. Dann ist die \textit{erste Chern-Klasse} $c_1(\mathcal{L})$ von $\mathcal{L}$ definiert als das Bild der Klasse von $\mathcal{L}$ in $\mathrm{Pic}({X})$ unter der Abbildung
\[\mathrm{Pic}({X})\xrightarrow{\sim} H_{\Acute{e}t}^1({X},\mathbb{G}_{m,{X}})\xrightarrow{} H_{\Acute{e}t}^2({X},\underline{\mathbb{Z}_{p}}(1)).\]
\end{definition}

Genau wie im Falle von Schemata und komplexen Mannigfaltigkeiten definieren wir nun dadurch den Chern-Charakter und die Todd-Klasse. Sei $X$ ein analytischer adischer Raum $X$ über $\Spec \mathbb{Z}[1/p]$. Im Folgenden bezeichnen wir mit $H^\ast_{\Acute{e}t}(X,\mathbb{Z}_p)$ (bzw. $H^\ast_{\Acute{e}t}(X,\mathbb{Q}_p)$) die direkte Summe von Kohomologiegruppen $\underset{n\in\mathbb{Z}}{\bigoplus}H^{2n} ({X},\underline{\mathbb{Z}_p}(n))$ (bzw. $\underset{n\in\mathbb{Z}}{\bigoplus}H^{2n} ({X},\underline{\mathbb{Q}_p}(n))$).

\begin{satz}
Es gibt eine eindeutig bestimmte funktorielle Zuordnung $\mathcal{V}\mapsto c(\mathcal{V})\in H^\ast_{\Acute{e}t}(X,\mathbb{Z}_p)[[t]]$, wobei $t$ eine formale Variable und $\mathcal{V}$ ein Vektorbündel auf einem noetherschen analytischen adischen Raum ${X}$ über $\Spec \mathbb{Z}[1/p]$ ist, die folgende Eigenschaften besitzt:
\begin{enumerate}[label=(\roman*)]
\item Die Zuordnung $\mathcal{V}\mapsto c(\mathcal{V})$ vertauscht mit dem Rückzug.
\item Für jede exakte Sequenz von Vektorbündeln $0\xrightarrow{} \mathcal{V}' \xrightarrow{} \mathcal{V} \xrightarrow{} \mathcal{V}''\xrightarrow{} 0$ gilt $c(\mathcal{V})=c(\mathcal{V}')\cdot c(\mathcal{V}'')$.
\item Für ein Geradenbündel $\mathcal{L}$ gilt $c(\mathcal{L})=1+c_1(\mathcal{L})t$.
\end{enumerate}
\end{satz}

\begin{proof}
Wie bei Schemata und komplexen Mannigfaltigkeiten folgt der Satz aus dem \textit{Spaltungsprinzip}, welches auch in unserer Situation gilt, denn man kann den Fahnenraum eines Vektorbündels $\mathcal{V}$ auch im adischen Kontext konstruieren.\footnote{Für eine Diskussion der Konsktruktion von $\mathbb{P}(\mathcal{V})$ für ein Vektorbündel $\mathcal{V}$ in der adischen Geometrie verweisen wir auf \cite[Abschnitt 6]{Bogdascha}.} Bezüglich Details konsultiere man etwa \cite{Chernklassen}.
\end{proof}

\begin{satz}
Es gibt eine eindeutig bestimmte funktorielle Zuordnung $\mathcal{V}\mapsto \mathrm{ch}(\mathcal{V})\in H^{\ast} ({X},\mathbb{Q}_p)$, wobei $\mathcal{V}$ ein Vektorbündel auf einem noetherschen analytischen adischen Raum ${X}$ über $\Spec \mathbb{Z}[1/p]$ ist, die folgende Eigenschaften besitzt:
\begin{enumerate}[label=(\roman*)]
\item Die Zuordnung $\mathcal{V}\mapsto \mathrm{ch}(\mathcal{V})$ vertauscht mit dem Rückzug.
\item Für jede exakte Sequenz von Vektorbündeln $0\xrightarrow{} \mathcal{V}' \xrightarrow{} \mathcal{V} \xrightarrow{} \mathcal{V}''\xrightarrow{} 0$ gilt $\mathrm{ch}(\mathcal{V})=\mathrm{ch}(\mathcal{V}') + \mathrm{ch}(\mathcal{V}'')$.
\item Für ein Geradenbündel $\mathcal{L}$ gilt $\mathrm{ch}(\mathcal{L})=e^{c_1(\mathcal{L})}$.
\end{enumerate}
Außerdem gilt die Künneth-Formel $\mathrm{ch}(\mathcal{V}\otimes \mathcal{W}) =\mathrm{ch}(\mathcal{V})\cdot  \mathrm{ch}(\mathcal{W})$ für Vektorbündel $\mathcal{V}$ und $\mathcal{W}$ auf einem noetherschen analytischen adischen Raum ${X}$. 
\end{satz}

\begin{proof}
Die Existenz, die Eindeutigkeit und die Formel $\mathrm{ch}(\mathcal{V}\otimes \mathcal{W}) =\mathrm{ch}(\mathcal{V})\cdot  \mathrm{ch}(\mathcal{W})$ werden genau wie in klassischen Fällen bewiesen, siehe etwa \cite{Cherncharakter}.
\end{proof}

In ähnlicher Weise beweist man schließlich den folgenden Satz:

\begin{satz}\label{todklasse}
Es gibt eine eindeutig bestimmte funktorielle Zuordnung $\mathcal{V}\mapsto \mathrm{Td}(\mathcal{V})\in H^{\ast} ({X},\mathbb{Q}_p)$, wobei $\mathcal{V}$ ein Vektorbündel auf einem noetherschen analytischen adischen Raum ${X}$ über $\Spec \mathbb{Z}[1/p]$ ist, die folgende Eigenschaften besitzt:
\begin{enumerate}[label=(\roman*)]
\item Die Zuordnung $\mathcal{V}\mapsto \mathrm{Td}(\mathcal{V})$ vertauscht mit dem Rückzug.
\item Für jede exakte Sequenz von Vektorbündeln $0\xrightarrow{} \mathcal{V}' \xrightarrow{} \mathcal{V} \xrightarrow{} \mathcal{V}''\xrightarrow{} 0$ gilt $\mathrm{Td}(\mathcal{V})=\mathrm{Td}(\mathcal{V}') \cdot \mathrm{Td}(\mathcal{V}'')$.
\item Für ein Geradenbündel $\mathcal{L}$ gilt $\mathrm{Td}(\mathcal{L})=\frac{c_1(\mathcal{L})}{1-e^{-c_1(\mathcal{L})}}$.
\end{enumerate}
\end{satz}

Sei ${X}$ ein noetherscher analytischer adischer Raum, glatt über $S$. Es bezeichne $\mathcal{T}_{{X}}$ das relative Tangentialbündel von $X$ über $S$. Zur Vereinfachung der Notation schreiben wir für $\mathrm{Td}(\mathcal{T}_{{X}})$ einfach $\mathrm{Td}({X})$. 

Zur vollständigen Erklärung der Formulierung des Satzes von Grothendieck-Riemann-Roch fehlt uns nur noch der \textit{verfeinerte Chern-Charakter} $\mathrm{ch}:\mathbb{K}_{X,S}(-)\xrightarrow{} \mathcal{H}_{\Acute{e}t}(X,\mathbb{Q}_p)$, wobei $X$ ein noetherscher analytischer adischer Raum über $S$ ist, welcher die Bedingungen des Satzes~\ref{hauptabstiegssatz} erfüllt.\footnote{Man bemerke, dass sie automatisch erfüllt sind, sobald $X$ glatt über $S$ ist.} Wir definieren diesen durch den Isomorphismus $KU^\wedge_{p,X}\xrightarrow{\sim} L_{K(1)}K_X(-)$ aus dem vorhergehenden Abschnitt mithilfe des folgenden Satzes.

\begin{satz}\label{zerlegung}
Sei ${X}$ ein analytischer adischer Raum über $\Spec \mathbb{Z}[1/p]$, welcher die Bedingungen des Satzes~\ref{hauptabstiegssatz} erfüllt. Dann wird die étale Garbe $KU^\wedge_{p,X}$ nach dem Invertieren von $p$ isomorph zur étalen Garbe $\underset{n\in\mathbb{Z}}{\bigoplus}\underline{\mathbb{Q}_p}(n)[2n]$ als Garbe von $\mathbb{E}_\infty$-Ringen.\footnote{Zur Definition der Struktur eines $\mathbb{E}_{\infty}$-Ringes auf $\underset{n\in\mathbb{Z}}{\bigoplus}\underline{\mathbb{Q}_p}(n)[2n]$ benutze man die Tatsache, dass die $p$-Vervollständigung des Tensorprodukts $\underline{\mathbb{Z}_p}(n)\otimes \underline{\mathbb{Z}_p}(m)$ isomorph zu $\underline{\mathbb{Z}_p}(n+m)$ ist.}
\end{satz}

\begin{proof}
Zum Beweis gehen wir zum proétalen Situs von $X$ über.\footnote{Für eine Diskussion der proétalen Topologie und insbesondere des proétalen Situs einer proendlichen Gruppe verweisen wir auf \cite{proet}.} Konkreter betrachten wir das folgende Diagramm von geometrischen Morphismen von $\infty$-Topoi (siehe die Diskussion unter Satz~\ref{hauptabstiegssatz}): 
\begin{center}
\begin{tikzcd}
X_{pro\Acute{e}t} \arrow[d,"\alpha"] \arrow[rr, bend left, "\Tilde{\pi}_X"] \arrow[r] &(\Spec \mathbb{Z}[1/p])_{pro\Acute{e}t}\arrow[r,"\Tilde{\pi}"] \arrow[d,"\beta"] & (B\mathbb{Z}_p^\times)_{pro\Acute{e}t} \arrow[d,"\gamma"]\\
X_{\Acute{e}t} \arrow[rr, bend right, "\pi_X"]  \arrow[r] &(\Spec \mathbb{Z}[1/p])_{\Acute{e}t}\arrow[r,"\pi"] & B\mathbb{Z}_p^\times
\end{tikzcd}
\end{center}
Um unsere Argumente zu verdeutlichen, ändern wir unsere Notationen für den Verlauf dieses Beweises etwas ab. Für eine natürliche Zahl $k\geq 0$ und eine ganze Zahl $n\in \mathbb{Z}$ bezeichnen wir mit $\mathbb{Z}/p^k\mathbb{Z}(n)$ das diskrete Spektrum $\mathbb{Z}/p^k\mathbb{Z}$ zusammen mit der $\mathbb{Z}_p^\times$-Wirkung, die durch $(\alpha,x)\in (\mathbb{Z}_p^\times,\mathbb{Z}/p^k\mathbb{Z})\mapsto \alpha^n\cdot x$ gegeben ist. Im Folgenden identifizieren wir es mit der unter der Yoneda-Einbettung entsprechenden Garbe auf $B\mathbb{Z}_p^\times$. Es bezeichne $\mathbb{Z}_{p,B\mathbb{Z}_p^\times}(n)$ die Garbe von Spektren auf $B\mathbb{Z}_p^\times$, die als $\varprojlim \mathbb{Z}/p^k\mathbb{Z}(n)$ gegeben ist. Hierbei wird der Limes in der Kategorie der Garben von Spektren auf $B\mathbb{Z}_p^\times$ gebildet. Die in \cite[Lemma 3.8]{K1} konstruierte Garbe $KU^\wedge_{p}$ auf $B\mathbb{Z}_p^\times$ wird im Weiteren mit $KU^\wedge_{p,B\mathbb{Z}_p^\times}$ bezeichnet. Wir betrachten nun die Rückzüge dieser Garben nach $X_{\Acute{e}t}$. Es bezeichne $\mathbb{Z}_{p,X_{\Acute{e}t}}$ (bzw. $KU^\wedge_{p,X_{\Acute{e}t}}$) die $p$-Vervollständigung der Hypervervollständigung des Rückzugs $\pi_X^\ast \mathbb{Z}_{p,B\mathbb{Z}_p^\times}$ (bzw. $\pi_X^\ast KU^\wedge_{p,B\mathbb{Z}_p^\times}$). Man bemerke dabei, dass die so definierte Garbe $\mathbb{Z}_{p,X_{\Acute{e}t}}$ mit der unter Definition~\ref{glattheit} definierten Garbe $\underline{\mathbb{Z}_p}(n)$ übereinstimmt, denn die Garben $\pi_X^\ast\mathbb{Z}/p^k\mathbb{Z}(n)$ für $k\geq 0$ sind alle hypervollständig. Darüber hinaus bezeichnen wir mit $\mathbb{Q}_{p,X_{\Acute{e}t}}$ die Garbe $\mathbb{Z}_{p,X_{\Acute{e}t}}[1/p]$.

Wir definieren nun analoge Garben auf den proétalen Siten. Da der proétale $\infty$-Topos von $\mathbb{Z}_p^\times$ äquivalent zur $\infty$-Kategorie der verdichteten Animen mit $\mathbb{Z}_p^\times$-Wirkung ist (siehe \cite[Lemma 4.3.2]{proet}), wobei $\mathbb{Z}_p^\times$ mit der $p$-adischen Topologie versehen wird, ist er hypervollständig. Für $n\in\mathbb{Z}$ bezeichne $\mathbb{Z}_{p,(B\mathbb{Z}_p^\times)_{pro\Acute{e}t}}(n)$ (bzw. $\mathbb{Q}_{p,(B\mathbb{Z}_p^\times)_{pro\Acute{e}t}}(n)$) die topologische abelsche Gruppe $\mathbb{Z}_p$ (bzw. $\mathbb{Q}_p$) zusammen mit der $\mathbb{Z}_p^\times$-Wirkung, die durch $(\alpha,x)\in (\mathbb{Z}_p^\times,\mathbb{Z}_p)\mapsto \alpha^n\cdot x$ (bzw. $(\alpha,x)\in (\mathbb{Z}_p^\times,\mathbb{Q}_p)\mapsto \alpha^n\cdot x$) gegeben ist. Im Weiteren identifizieren wir sie mit den entsprechenden Garben abelscher Gruppen auf $(B\mathbb{Z}_p^\times)_{pro\Acute{e}t}$. Wie man direkt nachprüft, etwa unter Benutzung der Tatsache, dass der direkte Bildfunktor $\gamma_\ast$ mit Limites und filtrierten Kolimites in $\mathrm{Sh}((B\mathbb{Z}_p^\times)_{pro\Acute{e}t},\mathrm{Sp}_{\leq k})$ für jedes feste $k\in\mathbb{Z}$ vertauscht und darstellbare Objekte erhält, sind die kanonischen Abbildungen 
\[\mathbb{Z}_{p,B\mathbb{Z}_p^\times}(n)\rightarrow \gamma_\ast \mathbb{Z}_{p,(B\mathbb{Z}_p^\times)_{pro\Acute{e}t}}(n) \quad \text{und} \quad  \mathbb{Q}_{p,B\mathbb{Z}_p^\times}(n) \rightarrow \gamma_\ast \mathbb{Q}_{p,(B\mathbb{Z}_p^\times)_{pro\Acute{e}t}}(n)\]
Isomorphismen. Wir definieren nun die Garbe $KU^\wedge_{p,(B\mathbb{Z}_p^\times)_{pro\Acute{e}t}}$ auf $(B\mathbb{Z}_p^\times)_{pro\Acute{e}t}$ als \[(\mathbb{S}[B^2\mathbb{Z}_{p,(B\mathbb{Z}^\times_{p})_{{pro\Acute{e}t}}}(1)]^\wedge_p[\beta^{-1}])^\wedge_p,\]
wobei $\beta$ das Bott'sche Element bezeichnet. Man überzeugt sich leicht davon, dass die kanonische Abbildung $KU^\wedge_{p,B\mathbb{Z}_p^\times}\rightarrow \gamma_\ast KU^\wedge_{p,(B\mathbb{Z}_p^\times)_{pro\Acute{e}t}}$ ein Isomorphismus ist, denn der Funktor $\gamma_\ast$ vertauscht mit Postnikow-Limites. Es bezeichne \[\mathbb{Z}_{p,X_{pro\Acute{e}t}} \quad \text{bzw.}\quad \mathbb{Q}_{p,X_{pro\Acute{e}t}} \quad \text{bzw.}\quad KU^\wedge_{p,X_{pro\Acute{e}t}} \quad \text{bzw.}\quad KU^\wedge_{p,X_{pro\Acute{e}t}}[1/p]\] die Hypervervollständigung des Rückzugs \[\Tilde{\pi}_X^\ast \mathbb{Z}_{p,(B\mathbb{Z}_p^\times)_{pro\Acute{e}t}}(n)\quad \text{bzw.}\quad \Tilde{\pi}_X^\ast \mathbb{Q}_{p,(B\mathbb{Z}_p^\times)_{pro\Acute{e}t}}(n)\quad \text{bzw.}\quad \Tilde{\pi}_X^\ast KU^\wedge_{p,(B\mathbb{Z}_p^\times)_{pro\Acute{e}t}}\quad \text{bzw.}\quad \pi_\ast KU^\wedge_{p,X_{pro\Acute{e}t}}[1/p].\] Da der Rückzugfunktor $\Tilde{\pi}_X^\ast$ mit Limites vertauscht\footnote{Dies ergibt sich aus den Eigenschaften der proétalen Topologie.}, sind die so definierten Garben $\mathbb{Z}_{p,X_{pro\Acute{e}t}}$ und $KU^\wedge_{p,X_{pro\Acute{e}t}}$ zudem $p$-vollständig. Man prüft nun leicht nach, dass die kanonischen Abbildungen
\[\mathbb{Z}_{p,X_{\Acute{e}t}}(n) \rightarrow \alpha_\ast \mathbb{Z}_{p,X_{pro\Acute{e}t}}(n),\quad \mathbb{Q}_{p,X_{\Acute{e}t}}(n) \rightarrow \alpha_\ast \mathbb{Q}_{p,X_{pro\Acute{e}t}}(n),\]\[KU^\wedge_{p,X_{\Acute{e}t}}\rightarrow \alpha_\ast KU^\wedge_{p,X_{pro\Acute{e}t}}\quad \text{und}\quad 
KU^\wedge_{p,X_{\Acute{e}t}}[1/p]\rightarrow \alpha_\ast KU^\wedge_{p,X_{pro\Acute{e}t}}[1/p]\]
ebenfalls Isomorphismen sind.

Aus unseren Argumentationen oben folgt, dass es genügt, einen Isomorphismus 
\[KU^\wedge_{p,(B\mathbb{Z}_p^\times)_{pro\Acute{e}t}}\xrightarrow{\sim} \underset{n\in\mathbb{Z}}{\bigoplus}\mathbb{Q}_{p,(B\mathbb{Z}_p^\times)_{pro\Acute{e}t}}(n)[2n]\]
von verdichteten Ringspektren mit $\mathbb{Z}_p^\times$-Wirkung zu konstruieren. Wir betrachten zwei Fälle.

\textit{Fall $p\neq 2$.} Sei $g\in\mathbb{Z}\subset \mathbb{Z}_p$ ein topologischer Erzeuger der multiplikativen Gruppe $\mathbb{Z}_p^\times$. Wir betrachten den geometrischen Morphismus $\nu:(B\mathbb{Z})_{pro\Acute{e}t}\rightarrow (B\mathbb{Z}_p^\times)_{pro\Acute{e}t}$ von $\infty$-Topoi, welcher von der Abbildung $\mathbb{Z}\rightarrow \mathbb{Z}_p^\times,\ 1\mapsto g$ induziert wird. Es bezeichne $\mathbb{S}_p$ das verdichtete Spektrum $\varprojlim \mathbb{S}/p^k$, wobei $\mathbb{S}/p^k$ als diskretes verdichtetes Spektrum angesehen und der Limes in der $\infty$-Kategorie der verdichteten Spektren gebildet wird. Wir versehen es im Weiteren mit der trivialen $\mathbb{Z}_p^\times$-Wirkung. Wir behaupten, dass die Einschränkung des Rückzugsfunktors $\nu^\ast$ auf feste\footnote{Ein verdichtetes Spektrum mit $\mathbb{Z}_p^\times$-Wirkung heißt \textit{fest}, wenn dessen zugrunde liegendes verdichtetes Spektrum fest ist.} $\mathbb{S}_p$-Moduln mit $\mathbb{Z}_p^\times$-Wirkung volltreu ist. Wie man leicht einsehen kann, ist unsere Behauptung dazu äquivalent, dass das Tensorprodukt $\mathbb{S}_p[\mathbb{Z}_p^\times]^\blacksquare\underset{\mathbb{S}_p[g^\mathbb{Z}]}{\otimes^\blacksquare} \mathbb{S}_p[\mathbb{Z}_p^\times]^\blacksquare$ zum festen Spektrum $\mathbb{S}_p[\mathbb{Z}_p^\times]^\blacksquare$ isomorph ist. Hierbei bezeichnet $(-)^\blacksquare$ den Verfestigungsfunktor. Es bezeichne $\mu_{p-1}\subset \mathbb{Z}_p^\times$ die Untergruppe der $p$-ten Einheitswurzeln. Wir führen folgende Rechnung durch:
\[\mathbb{S}_p[\mathbb{Z}_p^\times]^\blacksquare\overset{\mathrm{def}}{=} \underset{k\geq 1}{\varprojlim}\, \mathbb{S}_p[\mu_{p-1}\times ((1+p\mathbb{Z}_p)/(1+p^k\mathbb{Z}_p))]\overset{\mathrm{log}}{\cong} \underset{k\geq 1}{\varprojlim}\, \mathbb{S}_p[\mu_{p-1}\times \mathbb{Z}/p^{k-1}\mathbb{Z}]\cong\mathbb{S}_p[\mu_{p-1}]\otimes \mathbb{S}_p[[T]],\]
wobei $T$ eine formale Variable ist und $g$ auf $\mathbb{S}_p[[T]]$ durch $\cdot (1+T)$ wirkt. Mit anderen Worten müssen wir zeigen, dass
\[\mathbb{S}_p[[1-T]]\underset{\mathbb{S}_p[T]}{\otimes^\blacksquare} \mathbb{S}_p[[1-T]] \cong \mathbb{S}_p[[1-T]]^\blacksquare\]
gilt. Da das analytische verdichtete Ringspektrum $(\mathbb{S}_{p}[T],\mathbb{S}_{p})_\blacksquare$ isomorph zum analytischen verdichteten Ringspektrum $\mathbb{S}_{p}[T]_\blacksquare$ ist, ist die linke Seite isomorph zu $\mathbb{S}_p[[1-T]]\underset{\mathbb{S}_p[T]_\blacksquare}{\otimes} \mathbb{S}_p[[1-T]]$. Wir schreiben $\mathbb{S}_p[[1-T]]$ als Faser \[\mathrm{fib}\, (\underset{k\in\mathbb{N}}{\prod}\mathbb{S}_p[T]/(1-T)^k\xrightarrow{} \underset{k\in\mathbb{N}}{\prod}\mathbb{S}_p[T]/(1-T)^k).\]
Man bemerke dabei, dass 
\[\mathbb{S}_p[T]/(1-T)^k\cong \mathrm{cofib}\, (\mathbb{S}_p[T]\xrightarrow[]{\cdot (1-T)^k} \mathbb{S}_p[T])\quad \text{und}\quad \underset{k\in\mathbb{N}}{\prod}\mathbb{S}_p[T]\cong \mathbb{S}_p[T]_\blacksquare [\mathbb{N}\cup \{\infty\}].\] Der gewünschte Isomorphismus wird nun durch direkte Rechnung verifiziert.

Es bezeichne $\mathbb{Q}_p(n)$ (bzw. $KU^\wedge_p$) das zugrunde liegende verdichtete Spektrum von $\mathbb{Q}_{p,(B\mathbb{Z}_p^\times)_{pro\Acute{e}t}}(n)$ (bzw. $KU^\wedge_{p,(B\mathbb{Z}_p^\times)_{pro\Acute{e}t}}$) versehen mit dessen $g^\mathbb{Z}$-Wirkung. Aus dem obigen Absatz folgt, dass es genügt, einen Isomorphismus 
\[KU^\wedge_{p}\rightarrow \underset{n\in\mathbb{Z}}{\bigoplus}\mathbb{Q}_{p}(n)[2n]\]
von verdichteten Ringspektren mit $g^\mathbb{Z}$-Wirkung zu konstruieren. Es bezeichne $\mathbb{Q}(n)$ die diskrete topologische Gruppe $\mathbb{Q}$ zusammen mit der $g^\mathbb{Z}$-Wirkung, die durch $(g^k,x)\in g^\mathbb{Z}\times \mathbb{Q}\mapsto g^{nk}\cdot x$ gegeben ist.  Wie man leicht nachprüft, gilt $\mathbb{Q}_{p}(n)\cong \mathbb{Q}(n)\otimes \mathbb{S}_p$. Auf dem Niveau der zugrunde liegenden verdichteten Ringspektren, d. h. ohne $g^\mathbb{Z}$-Wirkung, gilt außerdem $KU^\wedge_{p}\cong KU \otimes \mathbb{S}_p$, wobei $KU$ als diskretes verdichtetes Ringspektrum angesehen wird. In der Tat folgt dies leicht aus den Isomorphismen $ku^\wedge_p\cong ku \otimes \mathbb{S}_p$ und $KU^\wedge_p\cong ku^\wedge_p[\beta^{-1}]$, wobei $\beta$ das Bott'sche Element bezeichnet. Insbesondere gilt $\psi: (KU\otimes \mathbb{Q})\otimes \mathbb{S}_p\cong  KU^\wedge_p[1/p]$. Wir wollen nun die linke Seite mit einer $g^\mathbb{Z}$-Wirkung ausstatten, sodass $\psi$ zu einem $g^\mathbb{Z}$-äquivarianten Isomorphismus von verdichteten Ringspektren wird. Der Snaith'sche Satz besagt, dass $KU$ zum Spektrum $\mathbb{S}[B^2 \mathbb{Z}][\beta^{-1}]$ isomorph ist. Die Endomorphismen $B^2 \mathbb{Z}\xrightarrow{\cdot g^k} B^2 \mathbb{Z}$ für $k\geq 0$ induzieren in offensichtlicher Weise Endomorphismen $\phi_k:KU\rightarrow KU$. Es bezeichne $\Phi$ die kanonische Abbildung $KU\rightarrow KU^\wedge_p$. Wie man leicht einsehen kann, gilt $\Phi \circ\phi_k=g^k\cdot \Phi$ für jedes $k\geq 0$, wobei $g^k\cdot$ auf der rechten Seite der Gleichung die Wirkung von $g^k$ auf $KU^\wedge_p$ bezeichnet. Man überzeugt sich nun leicht davon, dass die Endomorphismen $\phi_k$ Automorphismen von $KU\otimes \mathbb{Q}$ induzieren und somit eine Wirkung von $g^\mathbb{Z}$ auf $KU\otimes \mathbb{Q}$, sodass die Abbildung $\psi$ zu einem $g^\mathbb{Z}$-äquivarianten Isomorphismus wird.

Mit anderen Worten haben wir das Problem auf folgendes reduziert: Wir müssen einen $g^\mathbb{Z}$-äquivarianten Isomorphismus $KU\otimes \mathbb{Q}\xrightarrow[]{\sim}\underset{n\in \mathbb{Z}}{\bigoplus} \mathbb{Q}(n)[2n]$ von Ringspektren konstruieren. Dies folgt aber direkt aus dem Isomorphismus $KU\cong \mathbb{S}[B^2 \mathbb{Z}][\beta^{-1}]$.

\textit{Fall $p=2$}. Dieser Fall ist tatsächlich völlig analog zu dem obigen. Der einzige Unterschied besteht darin, dass die Multiplikative Gruppe $\mathbb{Z}_2^\times$ keine dichte zyklische Untergruppe besitzt. Es gibt aber eine ganze Zahl $g$, sodass die Untergruppe $\{\pm 1\}\times g^\mathbb{Z}$ dicht in $\mathbb{Z}_2^\times$ ist. Wir ersetzen also die Gruppe $g^\mathbb{Z}$ im obigen Beweis mit dieser Untergruppe $\{\pm 1\}\times g^\mathbb{Z}$ und argumentieren dann wie oben.
\end{proof}

Sei $X$ ein noetherscher analytischer adischer Raum über $S$, welcher die Bedingungen des Satzes~\ref{hauptabstiegssatz} erfüllt. Wir definieren den verfeinerten Chern-Charakter $\mathrm{ch}:\mathbb{K}_{X,S}(-)\xrightarrow{}\mathcal{H}_{\mathrm{\Acute{e}t}}(X,\mathbb{Q}_p)$ als den Morphismus, der sich durch Anwenden des direkten Bildfunktors ${p}_{X\ast}$ auf die Verkettung
\[\mathbb{K}_X(-)\xrightarrow{} L_{K(1)}K_X(-)\cong KU_{X,p}^\wedge \xrightarrow{} \underset{n\in\mathbb{Z}}{\bigoplus}\underline{\mathbb{Q}_p}(n)[2n]\]
ergibt. Der so definierte Chern-Charakter induziert insbesondere durch Bildung von $\pi_0$ eine Abbildung $K_0(\mathrm{Vect}_X)\xrightarrow{}  H^{\ast}(X,\mathbb{Q}_p)$, wobei $\mathrm{Vect}_X$ das Gruppoid der Vektorbündel auf $X$ bezeichnet. Wir behaupten, dass diese Abbildung mit dem mittels der Kummer'schen Sequenz oben definierten Chern-Charakter $\mathrm{ch}:K_0(\mathrm{Vect}_X)\xrightarrow{}H^{\ast} (X,\mathbb{Q}_p)$ übereinstimmt. Wir betrachten das folgende Diagramm:
\begin{center}
\begin{tikzcd}
\mathbb{S}[ B\mathbb{G}_{m,X}] \arrow[rd] \arrow[rrrrrd, bend left=8] & & & & \\
\mathbb{S}[ B\mu_{p^\infty} ] \arrow[r] \arrow[u] & \mathbb{K}_X(-) \arrow[r] & L_{K(1)}K_X(-) & {KU^\wedge_{p,X}} \arrow[r,"\sim"]\arrow[l,swap,"\sim"] & {\underset{n\in\mathbb{Z}}{\bigoplus}\underline{\mathbb{Q}_p}(n)\beta^n[2n]} \arrow[r] & {\underline{\mathbb{Q}_p}(1)\beta[2]} \\
\end{tikzcd}
\end{center}
Der Morphismus $\mathbb{S}[ B\mu_{p^\infty}]\xrightarrow[]{}\mathbb{K}_X(-)$ (bzw. $\mathbb{S}[ B\mathbb{G}_{m,X}]\xrightarrow[]{}\mathbb{K}_X(-)$) in dem Diagramm entsteht durch Adjunktion aus den Abbildungen $\mu_{p^\infty}(R)\xrightarrow{}K_1(R)$ (bzw. $R^\times\xrightarrow{}K_1(R)$) für einen beliebigen Ring $R$. Der schräge Pfeil oben stellt den Morphismus dar, welcher durch Adjunktion dem durch die Kummer'sche Sequenz definierten Morphismus $\mathbb{G}_{m,X}[1]\xrightarrow{} \mathbb{Q}_p[2]$ entspricht. Die anderen Morphismen sind wie in der Diskussion oben. Die Frage, ob die beiden definierten Chern-Charaktere übereinstimmen, ist genau dazu äquivalent, ob das rechte Dreieck kommutativ ist. Dies folgt aber unmittelbar aus der Definition des Morphismus $KU^\wedge_{p,X}\xrightarrow{}L_{K(1)}K_X(-)$.

Wir kommen nun zum eigentlichen Beweis des Satzes von Grothendieck-Riemann-Roch. Nach all unserer Arbeit lässt sich der Satz genau in derselben Weise wie in der Situation von \cite[Vorlesung 15]{Complex} beweisen. Wir erklären kurz die wichtigsten Schritte. Wir betrachten zunächst den folgenden abstrakten Rahmen (siehe \cite[Definition 15.6]{Complex}). Sei $\mathrm{Man}_S$ die Kategorie der noetherschen analytischen adischen Räume, glatt und eigentlich über $S$, und sei $(K,\otimes)$ eine symmetrische monoidale $1$-Kategorie, abgeschlossen unter endlichen Produkten.

\begin{enumerate}[label=(\roman*)]
\item Unter einer Kohomologietheorie auf $\mathrm{Man}_S$ mit Werten in $K$ verstehen wir einen kontravarianten symmetrischen monoidalen Funktor $H^\ast: \mathrm{Man}_S^\mathrm{op}\xrightarrow{} K$, der disjunkte Vereinigungen auf Produkte abbildet und die folgende Eigenschaft besitzt: Für jedes ${X}\in \mathrm{Man}_S$ und jedes Vektorbündel $\mathcal{V}$ auf ${X}$ positiven Rangs ist die Abbildung $H^\ast (X)\xrightarrow{} H^\ast (\mathbb{P}(\mathcal{V}))$ ein Monomorphismus.
\item Sei $H^\ast$ eine Kohomologietheorie auf $\mathrm{Man}_S$ mit Werten in $K$. Es bezeichne $\mathrm{Man}_S^{\cong}$ das maximale Untergruppoid von $\mathrm{Man}_S$. Unter einer \textit{direkten Bildstruktur} auf $H^\ast$ verstehen wir einen kovarianten Funktor $H_\ast: \mathrm{Man}_S\xrightarrow{} K$ mit einer Identifizierung $H_\ast|_{\mathrm{Man}_S^{\cong}} \cong H^\ast|_{\mathrm{Man}_S^{\cong}}$, der die folgenden Eigenschaften besitzt.
\begin{enumerate}[label=(\alph*)]
\item Sei $h:{Z}\xrightarrow[]{}{Z}'$ ein Morphismus in $\mathrm{Man}_S$. Es bezeichne ${h}^\ast$ (bzw. ${h}_\ast$) die Abbildung $H^\ast({h})$ (bzw. $H_\ast({h})$). Für jedes transverse kartesische Diagramm
\begin{center}
\begin{tikzcd}
{X}' \arrow[d,"{f}'"] \arrow[r,"{g}'"] & {X} \arrow[d,"{f}"] \\
{Y}' \arrow[r,"{g}"] & {Y}
\end{tikzcd}
\end{center}
in $\mathrm{Man}_S$ gilt ${g}^\ast {f}_\ast \cong  {f}'_\ast{g}'{}^\ast:H^\ast({X})\xrightarrow{} H^\ast({Y}') $. 
\item Sei $f:{X}\xrightarrow[]{}{Y}$ ein Morphismus in $\mathrm{Man}_S$ und $Z$ ein Objekt in $\mathrm{Man}_S$. Es gilt 
\[f_\ast\otimes \mathrm{id}_{H^\ast(Z)}\cong (f_\ast\otimes \mathrm{id}_Z)_\ast:H^\ast (X\times_S Z)\cong H^\ast(X)\otimes H^\ast(Z)\xrightarrow{} H^\ast(Y)\otimes H^\ast(Z) \cong H^\ast (Y\times_S Z).\]
\item Sei $\mathcal{L}$ ein Geradenbündel auf einem Raum ${X}\in\mathrm{Man}_S$. Die Abbildung
\[H^\ast (\mathbb{P}(\mathcal{L}\oplus \mathcal{O}_{{X}}))\xrightarrow[]{({p}_\ast,\infty^\ast)}H^\ast({X})\times H^\ast ({X})\]
ist ein Monomorphismus. Hierbei bezeichnet $p:\mathbb{P}(\mathcal{L}\oplus \mathcal{O}_{{X}})\xrightarrow{} X$ die kanonische Projektion und $\infty:X\xrightarrow{} \mathbb{P}(\mathcal{L}\oplus \mathcal{O}_{{X}})$ den $\infty$-Schnitt.
\end{enumerate}
\end{enumerate}

In unserer Situation nehmen wir die derivierte Kategorie $D_{\Acute{e}t}(S,\mathbb{Z})$ der étalen Garben auf $S$ als die Kategorie $K$ und die Zuordnung ${X}\mapsto {p}_{{X}\ast}L_{K(1)}K_{{X}}[1/p]\cong \mathcal{H}_{\Acute{e}t}(X,\mathbb{Q}_p)$ als die Kohomologietheorie $H^\ast$, wobei ${p}_{{X}}$ die Struktur Abbildung ${X}\xrightarrow{} S$ bezeichnet. Der Satz von Grothendieck-Riemann-Roch kann nun wie folgt interpretiert werden: Er besagt, dass die zwei definierten direkten Bildstrukturen, nämlich diejenigen, die durch den Funktor $f_\ast:\mathbb{K}_{X,S}\xrightarrow{} \mathbb{K}_{Y,S}$ bzw. die Poincaré-Dualität definiert werden, bis auf Multiplikation mit der Todd-Klasse übereinstimmen.

Zum Beweis des Satzes führen wir wie in \cite[Vorlesung 15]{Complex} den Begriff der \textit{Euler'schen Klasse} eines Geradenbündels ein. Für einen Raum ${X}\in \mathrm{Man}_S$ und ein Geradenbündel $\mathcal{L}$ auf ${X}$ bezeichne $0:{X}\xrightarrow{}\mathbb{P}(\mathcal{L}^\ast\oplus \mathcal{O}_{{X}})$ den Nullschnitt. Für eine Kohomologietherie $(H^\ast,H_\ast)$ zusammen mit einer direkten Bildstruktur definieren wir $e(\mathcal{L})\in H^\ast ({X})$ als $e(\mathcal{L})=0^\ast 0_\ast (1)$. Wie man leicht nachprüft (siehe \cite[Beispiel 15.10]{Complex}), ist in unserer Situation die Euler'sche Klasse von $\mathcal{L}$ bezüglich der $K$-theoretischen Bildstruktur gleich $1-e^{c_1(\mathcal{L})}$ und bezüglich der anderen gleich $c_1(\mathcal{L})$.

Sei nun $H_\ast$ eine direkte Bildstruktur auf $\mathcal{H}_{\Acute{e}t}(X,\mathbb{Q}_p)$ und ${f}:{X}\xrightarrow{}{Y}$ ein Morphismus in $\mathrm{Man}_S$. Wie man in direkter Weise nachprüft, definiert die Formel $\mathrm{Td}({Y})^{-1}H_\ast({f})\mathrm{Td}({X})$ ebenfalls eine direkte Bildstruktur (für eine allgemeine Definition des Twists einer direkten Bildstruktur verweisen wir auf \cite[Bemerkung 15.11]{Complex}). Insbesondere genügt es zum Beweis des Satzes von Grothendieck-Riemann-Roch zu zeigen, dass zwei direkte Bildstrukturen auf einer Kohomologietheorie übereinstimmen, falls die entsprechenden Theorien der Euler'schen Klassen übereinstimmen. Der Beweis dieser Aussage ist wörtlich derselbe wie in der Situation von \cite[Theorem 15.12]{Complex}.
\newpage
\appendix
\section{Die Nisnevich-Topologie}

In diesem Abschnitt des Anhangs geben wir eine Definition der Nisnevich-Topologie von analytischen adischen Räumen und beweisen einige ihrer Eigenschaften. Im gesamten Abschnitt nehmen wir an, dass alle betrachteten analytischen adischen Räume quasi-kompakt und quasi-separiert und somit alle Morphismen quasi-kompakt sind. Wenn ein Huber-Paar ohne anderslautenden Kommentar auftritt, wird es als vollständig angenommen.

\begin{notation} 
\begin{enumerate}[label=(\roman*)]
\item[]
\item Unter einem \textit{Topos} verstehen wir einen \textit{$1$-Topos}.
\item Sei $X$ ein analytischer adischer Raum. Mit $(\kappa(x),\kappa^+(x))$ bezeichnen wir den vollständigen Restklassenkörper von $X$ in $x\in X$.
\end{enumerate}
\end{notation}

\begin{definition}[{\cite[Definition 7.1]{Perf}}, {\cite[Definition 7.5.1]{Berkeley}}]
\begin{enumerate}[label=(\roman*)]
\item[]
\item Ein Morphismus $f: X \rightarrow Y$ von analytischen adischen Räumen heißt \textit{endlich étale}, wenn für jede affinoide offene Teilmenge $U=\Spa (R,R^+)\subset Y$ das Urbild $f^{-1}(U)\cong \Spa (\Tilde{R},\Tilde{R}^+)$ ebenso affinoid ist, sodass die Abbildung $R\rightarrow \Tilde{R}$ endlich étale und $\Tilde{R}^+$ der ganze Abschluss von $R^+$ ist.
\item Ein Morphismus $f: X\rightarrow Y$ von analytischen adischen Räumen heißt \textit{étale}, wenn jeder Punkt $x\in X$ eine offene Umgebung $U$ hat, sodass es eine offene Teilmenge $V\subset Y$ mit $f(U)\subset V$ gibt, für die die Einchränkung $f:U\rightarrow V$ als eine offene Einbettung gefolgt von einem endlichen étalen Morphismus faktorisiert werden kann.
\end{enumerate}
\end{definition}

\begin{bemerkung}
In der Definition von endlichen étalen Morphismen reicht es, nur die Existenz einer Zariski-Überdeckung mit der obigen Eigenschaft zu fordern, denn endliche projektive Moduln erfüllen Zariski-Abstieg nach \cite[Theorem 1.4]{firstpaper}.
\end{bemerkung}

\begin{lemma}\label{offen}
Ein étaler Morphismus von analytischen adischen Räumen ist offen.
\end{lemma}

\begin{proof}
Es genügt zu zeigen, dass jeder endliche étale Morphismus offen ist. In diesem Falle ist der Beweis fast wörtlich derselbe wie in der Situation von \cite[Lemma 1.7.9]{Huber}: Man ignoriere dessen ersten Absatz und bemerke, dass die Abbildung $X_B\rightarrow X_A$ offen ist, denn $X_B=\mathrm{Spv}(f)^{-1}(X_A)$.
\end{proof}

Im Folgenden machen wir reichlich Gebrauch von der Theorie der Bewertungskörper. Obwohl wir hier keine Einleitung in dieses Thema geben können, erinnern wir dennoch an einige grundlegende Definitionen, vor allem um die Notationen festzulegen.

\begin{definition}
\begin{enumerate}[label=(\roman*)]
\item Ein (eventuell nicht vollständiges) Huber-Paar $(K,K^+)$ heißt \textit{analytischer affinoider Körper}, wenn $K$ ein nicht archimedischer Körper\footnote{also ein topologischer Körper, dessen Topologie von einer Bewertung vom Rang $1$ erzeugt wird} und $K^+$ dessen offener Bewertungsring ist.
\item Ein analytischer affinoider Körper $(K,K^+)$ heißt \textit{henselsch}, wenn der Ring $K^+$ henselsch ist.
\end{enumerate}
\end{definition}

\begin{lemma}\label{nulldimensionale Körper}
Sei $K$ ein (eventuell nicht vollständiger) nicht archimedischer Körper. Dann gilt:
\begin{enumerate}[label=(\roman*)]
\item Ist $K$ vollständig, so ist der affinoide Körper $(K,K^\circ)$ henselsch.
\item Ist $K$ separabel abgeschlossen, so ist jeder affinoide Körper der Form $(K,K^+)$ henselsch.
\end{enumerate}
\end{lemma}

\begin{proof}
\begin{enumerate}[label=(\roman*)]
\item Siehe \cite[Aufgabe VI.§8.6(b)]{Bourbaki}.
\item Dies folgt direkt aus der Definition.
\end{enumerate}
\end{proof}

\begin{definition}\label{henselisierung}
Sei $k$ ein Körper und $\mathcal{O}$ ein Bewertungsring von $k$. Die \textit{Henselisierung} von $(k,\mathcal{O})$ ist gegeben durch $(k_h,\mathcal{O}_h)$, wobei $\mathcal{O}_h$ die Henseliesirung von $\mathcal{O}$ und $k_h$ der Quotientskörper von $\mathcal{O}_h$ ist.
\end{definition}

\begin{lemma}
Sei $k$ ein Körper und $\mathcal{O}$ ein Bewertungsring von $k$. Man betrachte den separablen Abschluss $k^{\mathrm{sep}}$ von $k$ und einen Bewertungsring $\mathcal{O}^{\mathrm{sep}}$ mit $\mathcal{O}^{\mathrm{sep}}\cap k=\mathcal{O}$. Setzen wir 
\[G=\{\sigma\in \mathrm{Gal}(k^{\mathrm{sep}}/k)|\sigma(\mathcal{O}^{\mathrm{sep}})=\mathcal{O}^{\mathrm{sep}}\},\]
so existiert ein eindeutiger $K$-Isomorphismus 
\[(k_h,\mathcal{O}_h)\cong ((k^{\mathrm{sep}})^G,(\mathcal{O}^{\mathrm{sep}})^G).\]
\end{lemma}

\begin{proof}
Siehe \cite[Theorem 5.2.2]{Engler}.
\end{proof}

Aus dem Blickwinkel der rigiden Geometrie gibt es keinen Unterschied zwischen Henselisierung eines affinoiden Körpers und deren Vervollständigung, wie das folgende triviale Lemma zeigt. 

\begin{lemma} Sei $(K,K^+)$ ein (eventuell nicht vollständiger) affinoider Körper.
\begin{enumerate}[label=(\roman*)]
\item Ist $(K,K^+)$ henselsch, so ist $(\Hat{K},\Hat{K}^+)$ ebenso henselsch.
\item Die Vervollständigung der Henselisierung von $(K,K^+)$ ist isomorph zu der Vervollständigung der Henselisierung von $(\Hat{K},\Hat{K}^+)$.
\end{enumerate}
\end{lemma}

\begin{proof}
\begin{enumerate}[label=(\roman*)]
\item Dies folgt direkt aus den Definitionen (man bemerke $\Hat{K}^+/\Hat{\mathfrak{m}}^+\cong K^+/\mathfrak{m}^+$).
\item Dies folgt direkt aus den universellen Eigenschaften der Henselisierung und der Vervollständigung.
\end{enumerate}
\end{proof}

\begin{definition}
Sei $(K,K^+)$ ein vollständiger affinoider Körper.
\begin{enumerate}[label=(\roman*)]
\item Unter der \textit{Henselisierung von $(K,K^+)$} verstehen wir die Vervollständigung der Henselisierung des Paares $(K,K^+)$ im Sinne von Definition~\ref{henselisierung}. Mit $(K_h,K_h^+)$ wird im Weiteren immer diese vervollständigte Henselisierung bezeichnet.
\item Es bezeichne $\Bar{K}$ die Vervollständigung eines separablen Abschlusses von $K$, welche ebenfalls separabel abgeschlossen ist, und $\bar{K}^+$ einen beliebigen offenen Bewertungsring von $\bar{K}$ mit $\bar{K}^+\cap K=K^+$.\footnote{Solche Ringe sind nach \cite[Theorem 3.2.14]{Engler} konjugiert.} Wir nennen den affinoiden Körper $(\Bar{K},\Bar{K}^+)$ einen \textit{separablen Abschluss von $(K,K^+)$}.
\end{enumerate}
\end{definition}

\begin{definition} Sei $X$ ein analytischer adischer Raum.
\begin{enumerate}[label=(\roman*)]
\item Die zugrunde liegende Kategorie des \textit{kleinen étalen Situs} $\mathrm{\acute{e}t}_X$ von $X$ ist definiert als die Kategorie von étalen Räumen über $X$. Eine \textit{étale Überdeckung} ist eine Familie $\{Y_i\rightarrow Y\}_{i\in I}$ von étalen Morphismen mit der Eigenschaft, dass die Abbildung $\coprod Y_i\rightarrow Y$ surjektiv ist.  
\item Die zugrunde liegende Kategorie des \textit{kleinen Nisnevich-Situs} $\mathrm{Nis}_X$ von $X$ ist ebenso definiert als die Kategorie von étalen Räumen über $X$. Eine Nisnevich-Überdeckung ist eine étale Überdeckung $\{Y_i\rightarrow Y\}_{i\in I}$ mit der Eigenschaft, dass für jeden henselschen affinoiden Körper $(K,K^+)$ die Abbildung \[\coprod Y_i(K,K^+)\rightarrow Y(K,K^+)\] surjektiv ist.
\end{enumerate}
\end{definition}

\begin{bemerkung}
Aus Lemma~\ref{offen} lässt sich leicht ableiten, dass jede étale Überdeckung eine endliche Teilüberdeckung enthält. Da spektrale Räume quasi-kompakt bezüglich der konstruierbaren Topologie sind, folgt aus Satz~\ref{spaltung} unten, dass dasselbe auch für die Nisnevich-Topologie gilt.
\end{bemerkung}

\begin{lemma}
Ein Basiswechsel einer Nisnevich-Überdeckung ist eine Nisnevich-Überdeckung.
\end{lemma}

\begin{proof}
Dies folgt direkt aus der Definition, da ein Basiswechsel von einem étalen Morphismus ebenfalls étale ist.
\end{proof}

\begin{definition}
Sei $X$ ein analytischer adischer Raum. 
\begin{enumerate}[label=(\roman*)]
\item Unter einem \textit{geometrischen Punkt von $X$} verstehen wir einen affinoiden Körper der Form $(\Bar{\kappa}(x),\Bar{\kappa}^+(x))$ für $x\in X$ zusammen mit der kanonischen Abbildung
\[\Bar{\iota}_x:\Spa (\Bar{\kappa}(x),\Bar{\kappa}^+(x))\rightarrow X.\]
Ein solcher Punkt definiert einen Punkt des étalen Topos $X_{\mathrm{\acute{e}t}}$ über den Funktor
\[X_{\mathrm{\mathrm{\acute{e}t}}}\rightarrow \mathrm{Sets},\ \mathcal{F}\mapsto (\Bar{\iota}_x)_{\mathrm{\acute{e}t}}^*(\mathcal{F})(\Spa (\Bar{\kappa}(x),\Bar{\kappa}^+(x))).\]
\item Unter einem \textit{henselschen Punkt von $X$} verstehen wir einen affinoiden Körper $(K,K^+)$ mit einer Abbildung
\[\iota_x:\Spa (K,K^+)\rightarrow \Spa (\kappa_h(x),\kappa_h^+(x))\rightarrow X,\]
wobei $x\in X$, $K$ eine separable Erweiterung von $\kappa_h(x)$ und $K^+$ der ganze Abschluss von $\kappa_h^+(x)$ ist. Ein solcher Punkt definiert einen Punkt des Nisnevich-Topos $X_{\mathrm{Nis}}$ über den Funktor
\[X_{\mathrm{Nis}}\rightarrow \mathrm{Sets},\ \mathcal{F}\mapsto (\iota_x)_{Nis}^*(\mathcal{F})(\Spa (K,K^+)).\]
\end{enumerate}
\end{definition}

\begin{satz}\label{punkte}
Sei $X$ ein analytischer adischer Raum. Dann gilt:

\begin{enumerate}[label=(\roman*)]
\item Die Punkte des étalen Topos von $X$ bilden eine konservative Familie und sind genau die geometrischen Punkte von $X$.
\item Die Punkte des Nisnevich-Topos von $X$ bilden eine konservative Familie und sind genau die henselschen Punkte von $X$.
\end{enumerate}
\end{satz}

\begin{proof}
Dies wird analog zu \cite[Theorem VIII.7.9]{SGA4} bewiesen.
\end{proof}

Wie bei Schemata ist der étale Abstieg für analytische adische Räume die Kombination des Nisnevich- und des Galois-Abstiegs.

\begin{definition}
Sei $X$ ein analytischer adischer Raum und $\mathcal{F}$ eine étale Prägarbe von Animen auf $X$. Man sagt, dass $\mathcal{F}$ \textit{Galois-Abstieg erfüllt}, wenn für jede affinoide offene Teilmenge $\Spa (A,A^+)\subset X$ und jede galoissche Erweiterung $(A,A^+)\rightarrow (B,B^+)$ mit Galois-Gruppe $G$\footnote{D. h., die Abbildung $A\rightarrow B$ ist eine galoissche Erweiterung mit Galois-Gruppe $G$ und $B^+$ ist der ganze Abschluss von $A^+$.} die natürliche Abbildung $\mathcal{F}(\Spa (A,A^+))\rightarrow \mathcal{F}(\Spa (B,B^+))^G$\footnote{Hierbei bezeichnet $(-)^G$ die Homotopiefixpunkte.} ein Isomorphismus ist.
\end{definition}

\begin{satz}\label{étale=Nisnevich+Galois}
Sei $X$ ein analytischer adischer Raum. Dann ist eine étale Prägarbe $\mathcal{F}$ von Animen auf $X$ eine étale Garbe genau dann, wenn sie eine Nisnevich-Garbe ist und Galois-Abstieg erfüllt.
\end{satz}

\begin{proof}
Dies wird analog zu \cite[Theorem B.7.6.1]{SAG} bewiesen.
\end{proof}

Wir wollen im Folgenden die Nisnevich-Topologie eines analytischen adischen Raumes etwas genauer untersuchen.

\begin{definition}\label{grad_def}
Sei $f:Y\rightarrow X$ ein Morphismus von analytischen adischen Räumen, und seien $y\in Y$, $x=f(y)$. Man betrachte das kartesische Diagramm
\begin{center}
\begin{tikzcd}
Y \arrow[d,"f"] & U \arrow[l,swap,"\phi"]\arrow[d,"\psi"] \\
X & \arrow[l] \Spa (C,C^+)
\end{tikzcd}
\end{center}
wobei $(C,C^+)$ ein (vollständiger) separabel abgeschlossener affinoider Körper ist, mit abgeschlossenem Punkt $\bar{x}$, der auf $x$ abgebildet wird.
Der \textit{Grad} $d(x/y)$ ist gegeben durch die Mächtigkeit der Menge $\{\tilde{y}\in U|\phi(\tilde{y})=y,\psi(\tilde{y})=\bar{x}\}$.
\end{definition}

\begin{lemma}\label{grad_wohl}
Ist der Morphismus $f$ in der Situation von Definition~\ref{grad_def} étale, so ist der Grad wohldefiniert. Das heißt, dieser ist von $(C,C^+)$ unabhängig und endlich.
\end{lemma}

\begin{proof}
Wir dürfen ohne Einschränkung annehmen, dass $f: Y\rightarrow X$ endlich étale ist und $X$ und $Y$ affinoid sind. Wir betrachten einen separablen Abschluss $(C,C^+)=(\bar{\kappa}(x),\bar{\kappa}^+(x))$. Wie man leicht nachprüft, ist der Basiswechsel $U$ eine endliche disjunkte Vereinigung von affinoiden Räumen der Form $\Spa (C,\Tilde{C}^+)$ mit $C^+\subset \Tilde{C}^+$: Der Grad $d(y/x)$ ist also endlich. Die Unabhängigkeit von $(C,C^+)$ folgt aus der Tatsache, dass jeder Morphismus $\Spa (C,C^+)\rightarrow X$ über $(\bar{\kappa}(x),\bar{\kappa}^+(x))\rightarrow X$ faktorisiert.
\end{proof}

\begin{lemma}\label{grad_formel}
Sei 
\begin{center}
\begin{tikzcd}
Y \arrow[d,"f"] & Y' \arrow[l,swap,"g'"]\arrow[d,"f'"] \\
X & \arrow[l,"g"] X'
\end{tikzcd}
\end{center}
ein kartesisches Diagramm mit $f$ und $g$ étale. Sind $x\in X$, $y\in Y$, $x'\in X'$ mit $f(y)=x$, $g(x')=x$, so gilt 
\[d(y/x)=\underset{\substack{y'\in (f')^{-1}(x') \\ g(y')=y}}{\sum}d(y'/x').\] 
\end{lemma}

\begin{proof}
Die Formel wird durch eine leichte Rechnung direkt aus der Definition hergeleitet.
\end{proof}

\begin{lemma}\label{etale lokal}
Sei $f:Y\rightarrow X$ ein separierter étaler Morphismus von analytischen adischen Räumen. Dann gibt es eine surjektive étale Abbildung $X'\rightarrow X$, sodass der Morphismus $Y\underset{X}{\times}X'\xrightarrow{f'}X'$ zu einem Morphismus der Form
\[\underset{k=1}{\overset{n}{\coprod}}U_i\xrightarrow[]{\coprod j_k} X'\] 
isomorph ist, wobei $j_k: U_k\hookrightarrow X'$ quasi-kompakt offen sind.
\end{lemma}

\begin{proof}
Ohne Einschränkung dürfen wir annehmen, dass $X$ affinoid und tatesch ist. Wegen der Separiertheit gibt es für jedes Paar von Punkten in $Y$ mit demselben Bild eine affinoide Teilmenge von $Y$, die die beiden Punkte enthält. Demnach dürfen wir auch annehmen, dass $Y$ affinoid ist. Schließlich kann man annehmen, dass $f$ endlich étale ist, denn wir können nötigenfalls zu einer kleineren offenen Teilmenge von $X$ übergehen. 

Sei $x\in X$ und betrachte $\Spa (\bar{\kappa}(x),\bar{\kappa}^+(x)) \rightarrow X$. Man sieht leicht, dass der Basiswechsel von $f$ nach $\Spa (\bar{\kappa}(x),\bar{\kappa}^+(x))$ die gewünschte Eigenschaft besitzt. Man schreibe nun $\Spa (\bar{\kappa}(x),\bar{\kappa}^+(x))$ als $\underset{x\in U}{\varprojlim} U$ für $U\rightarrow X$ étale. Dann liefern \cite[Proposition 5.4.53]{Almost} und das étale Analogon von \cite[Proposition III.6.3.7]{Morel} eine étale Umgebung $U(x)\rightarrow X$ von $x$, für die der Basiswechsel $Y\underset{X}{\times}U(x)\rightarrow U(x)$ von der gewünschten Form ist. Da $X$ quasi-kompakt ist, wird es durch endlich viele solcher Umgebungen überdeckt.
\end{proof}

\begin{satz}
Sei $X$ ein analytischer adischer Raum. Eine étale Überdeckung $\coprod Y_i\rightarrow X$ ist eine Nisnevich-Überdeckung genau dann, wenn es für jedes $x\in X$ einen Punkt $y\in \coprod Y_i$ mit $d(y/x)=1$ gibt.
\end{satz}

\begin{proof}
Gegeben seien $x\in X$ und $y\in \coprod Y_i$ mit $f(y)=x$. Wie man leicht nachprüft, ist die Bedingung, dass es eine Hochhebung von $\Spa (\kappa_h(x),\kappa_h^+(x))\rightarrow X$ gibt, die den abgeschlossenen Punkt von $\Spa (\kappa_h(x),\kappa_h^+(x))$ nach $y$ abbildet, dazu äquivalent, dass es ein kommutatives Diagramm 
\begin{center}
\begin{tikzcd}
(\kappa(x),\kappa^+(x))\arrow [r]\arrow[d] & (\kappa_h(x),\kappa_h^+(x))\\
(\kappa(y),\kappa^+(y))\arrow[ur,dashed,"\psi"]  & 
\end{tikzcd}
\end{center}
gibt. Wir setzen $(C,C^+)=(\Bar{\kappa}(x),\Bar{\kappa}^+(x))$ und betrachten das Tensorprodukt 
\[(\kappa(y),\kappa^+(y))\underset{(\kappa(x),\kappa^+(x))}{\otimes}(C,C^+)\]
von analytischen Huber-Paaren. Durch direkte Rechnung verifiziert man, dass es zu einem endlichen direkten Produkt von Huber-Paaren der Form $(C,\Tilde{C}^+)$ mit $C^+\subset \tilde{C}^+$ isomorph ist. Wir behaupten nun, dass die Bedingung oben ihrerseits dazu äquivalent ist, dass dieses Tensorprodukt genau eine Kopie von $(C,C^+)$ enthält.

Angenommen, es gibt ein kommutatives Diagramm wie oben. Man sieht unmittelbar, dass
\[(\kappa(y),\kappa^+(y))\underset{(\kappa(x),\kappa^+(x))}{\otimes}(\kappa_h(x),\kappa_h^+(x))\cong \underset{\sigma\in G}{\prod} (\kappa_h(x),\kappa_h^+(x)+\sigma (\kappa^+(y))), \]
 wobei $G$ die Menge von $\kappa(x)$-Einbettungen von $\kappa(y)$ in $\kappa_h(x)$ bezeichnet. Wir behaupten, dass das Bild von $\kappa^+(y)$ unter jedem $\sigma\in G\setminus \{\psi \}$ nicht ganz in $\kappa_h^+(x)$ enthalten ist. In der Tat, wenn es ein $\sigma\in G\setminus \{\psi \}$ mit $\sigma (\kappa^+(y))\subset \kappa_h^+(x)$ gibt, dann kann man einen nicht trivialen stetigen Automorphismus von $(\kappa_h(x),\kappa_h^+(x))$ konstruieren, was der universellen Eigenschaft der Henselisierung widerspricht.

Sei nun $\kappa(y)$ eine separable Erweiterung von $\kappa(x)$, die nicht ganz in $\kappa_h(x)$ liegt. Hierbei haben wir implizit eine Einbettung von $(\kappa(y),\kappa^+(y))$ in $(C,C^+)$ gewählt. Wir setzen \[(\kappa_{h,x}(y),\kappa_{h,x}^+(y))=(\kappa(y),\kappa^+(y))\cap (\kappa_h(x),\kappa_h^+(x))\]
und schreiben $\kappa(y)$ als $\kappa_{h,x}(y)(\alpha)$ für $\alpha\in C\setminus \kappa_h(x)$. Durch eine ähnliche Rechnung wie oben verifiziert man, dass
\[(\kappa(y),\kappa^+(y)) \underset{(\kappa(x),\kappa^+(x))}{\otimes} (\kappa_h(x),\kappa_h^+(x)) \cong (\kappa_h(x)(\alpha),\overline{\kappa_h^+(x)}) \times  \prod (\kappa_h(x)(\alpha),\overline{\tilde{\kappa}_h^+(x)}),\]
wobei $\kappa_h^+(x)\subsetneq \Tilde{\kappa_h}^+(x)$ und $\overline{(-)}$ den ganzen Abschluss bezeichnet. Da die dem Ring $\kappa_h^+(x)$ entsprechende Bewertung $\nu:\kappa_h(x)\rightarrow \Gamma \sqcup \{0\}$ eine eindeutige Fortsetzung auf $C$ besitzt, prüft man leicht nach, dass
\[(\kappa_h(x)(\alpha),\overline{\kappa_h^+(x)})\underset{(\kappa_h(x),\kappa_h^+(x))}{\otimes}(C,C^+)\cong \underset{i=1}{\overset{n}{\prod}} (C,C^+),\ \mathrm{wobei}\ n=[\kappa(y):\kappa_{h,x}(y)].\]
\end{proof}

\begin{definition}
Sei $f:Y\rightarrow X$ ein étaler Morphismus von analytischen adischen Räumen und $W$ eine beliebige Teilmenge von $Y$. Es heißt $f|_W$ \textit{vom Grad $1$}, wenn $d(y/f(y))=1$ für alle $y\in W$ ist.
\end{definition}

\begin{satz}\label{spaltung}
Sei $f :Y\xrightarrow{} X$ ein étaler Morphismus von analytischen adischen Räumen. Sind $x\in X$, $y\in Y$ mit $f(y)=x$ und $d(y/x)=1$, so gibt es konstruierbare Teilmengen $x\in K\subset X$ und $y\in L\subset Y$ mit $\phi(L)\subset K$, sodass $f|_L:L\xrightarrow{} K$ ein Homöomorphismus vom Grad $1$ ist.
\end{satz}

\begin{proof}
Nach Verkleinerung von $Y$ dürfen wir annehmen, dass $f^{-1}(x)=y$ gilt. Es existiert nach dem Beweis von Lemma~\ref{etale lokal} eine quasi-kompakte étale Umgebung $g:X'\rightarrow X$ von $x$ mit $g^{-1}(x)=\{x'\}$, sodass der Basiswechsel
\begin{center}
\begin{tikzcd}
Y \arrow[d,"f"] & Y' \arrow[l,swap,"g'"]\arrow[d,"f'"] \\
X & \arrow[l,"g"] X'
\end{tikzcd}
\end{center}
eine endliche Vereinigung von quasi-kompakten offenen Einbettungen ist. Wir schreiben dann $Y'$ als $\underset{i=1}{\overset{n}{\coprod}}U_i$ für $U_i\subset X'$ quasi-kompakt offen. Aus Lemma~\ref{grad_formel} folgt, dass $x'\in X'$ genau ein Urbild $y'\in Y$ hat. Ohne Einschränkung dürfen wir $y\in U_1$ annehmen. Die Teilmenge $g'(U_i)$ ist nach Lemma~\ref{offen} offen für jedes $i=1,\dots,n$. Sie ist zudem quasi-kompakt, denn das Bild einer quasi-kompakten Teilmenge ist ebenfalls quasi-kompakt. Wir setzen nun $K=g(U_1)\setminus g(\underset{i=2}{\overset{n}{\bigcup}}U_i)$ und $L=f^{-1}(K)$. Man sieht unmittelbar, dass $L\subset g'(U_1)\setminus g'(\underset{i=2}{\overset{n}{\bigcup}}U_i)$. Wir behaupten, dass $f|_L:L\xrightarrow{} K$ ein Homöomorphismus vom Grad $1$ ist. Die Abbildung ist offensichtlich surjektiv. Sei $\Tilde{x}\in K$ und seien $\Tilde{y}_1,\Tilde{y}_2\in Y$ dessen Urbilder unter $f$. Wie man leicht verifiziert, etwa unter Verwendung der universellen Eigenschaft des Faserprodukts, existieren für jedes $\Tilde{x}'\in X'$ mit $g(\Tilde{x}')=\Tilde{x}$ Punkte $\Tilde{y}_1',\Tilde{y}_2'\in Y'$ mit $f'(\Tilde{y}_1')=f(\Tilde{y}_2')=\Tilde{x}'$ und $g'(\Tilde{y}_1')=\Tilde{y}_1,g'(\Tilde{y}_2')=\Tilde{y}_2$. Da $\Tilde{y}_1',\Tilde{y}_2'$ in $U_1\subset Y'$ liegen und $U_1\rightarrow X$ injektiv ist, sind $\Tilde{y}_1'$ und $\Tilde{y}_2'$ gleich, also auch $\Tilde{y}_1$ und $\Tilde{y}_2$. Aus Lemma~\ref{grad_formel} folgt weiter, dass die Abbildung $f|_L:L\xrightarrow{} K$ vom Grad $1$ ist, denn jeder Punkt in $U_1\setminus \underset{i=2}{\overset{n}{\bigcup}}U_i$ besitzt genau ein Urbild in $Y'$. Es bleibt also zu prüfen, dass die Umkehrfunktion stetig ist. Dies ist klar, denn $f^{-1}(K)=L$ und $f$ ist offen.
\end{proof}

\begin{satz}
Sei $f:Y\xrightarrow{} X$ eine Nisnevich-Überdeckung. Dann gibt es eine endliche Folge
\[ \emptyset =Z_n\subset Z_{n-1}\subset\dots\subset Z_1\subset Z_0=X\]
von abgeschlossenen Teilmengen von $X$ mit quasi-kompakten Komplementen mit der Eigenschaft, dass für jedes $i\in \{0,\dots,n-1\}$ die Einschränkung 
\[f: \Tilde{Z}_i\setminus \Tilde{Z}_{i+1} \xrightarrow{} Z_i\setminus Z_{i+1}\]  
spaltend ist, wobei $\Tilde{Z}_i$ das Urbild $f^{-1}(Z_i)$ bezeichnet. Das heißt, es gibt eine offene quasi-kompakte Teilmenge $Z_i'\subset \Tilde{Z}_i\setminus \Tilde{Z}_{i+1}$ (in der Topologie von $\Tilde{Z}_i\setminus \Tilde{Z}_{i+1}$), sodass $f:Z_i'\rightarrow Z_i\setminus Z_{i+1}$ ein Homöomorphismus vom Grad $1$ ist.
\end{satz}

\begin{proof}
Nach Satz~\ref{spaltung} hat jeder Punkt $x\in X$ eine konstuierbare Umgebung $K=U\cap W$ mit $U$ quasi-kompakt offen und $W$ abgeschlossen mit quasi-kompaktem Komplement, sodass es eine konstruierbare Teilmenge $L\subset Y$ mit $f(L)\subset K$ gibt, für die die Einschränkung $f|_L:L\rightarrow K$ ein Homöomorphismus vom Grad $1$ ist. Im Beweis des Satzes~\ref{spaltung} haben wir $L$ innerhalb einer quasi-kompakten offenen Teilmenge $Y'$ von $Y$ konstruiert, für die $f^{-1}(K)\cap Y' =L$ gilt. Deswegen dürfen wir annehmen, dass $L$ quasi-kompakt offen in der Topologie von $f^{-1}(K)$ ist. 

Da spektrale Räume quasi-kompakt bezüglich der konstruierbaren Topologie sind, wird $X$ durch endlich viele Teilmengen $K_1=U_1\cap W_1,\dots,K_n=U_n\cap W_n$ von dieser Form überdeckt. Wir beweisen nun den Satz mit Induktion nach $n$. Der Fall $n=1$ ist trivial. Nehmen
wir die Behauptung für $n-1$ als bewiesen an. Man betrachte 
\[Y\setminus f^{-1}(W_1)\xrightarrow{f} X\setminus W_1 \ \mathrm{und}\ f^{-1}(X\setminus U_1)\xrightarrow{f} X\setminus U_1.\] 
Nach Induktionsannahme gibt es Folgen 
\[ \emptyset =S_n\subset S_{n-1}\subset\dots\subset S_1\subset S_0=X\setminus W_1,\]
\[ \emptyset =R_{m}\subset R_{m-1}\subset\dots\subset R_1\subset R_0=X\setminus U_1\]
von abgeschlossenen Teilmengen von $X\setminus W_1$ bzw. $X\setminus U_1$ mit quasi-kompakten Komplementen mit der gewünschten Eigenschaft. Da $K_1=W_1\setminus (X\setminus U_1)$ eine gewünschte Spaltung besitzt, liefert dies die Folgen
\[ W_1\subset W_1\cup S_{n-1}\subset\dots\subset W_1\cup S_1\subset W_1\cup S_0=X,\]
\[ \emptyset =R_{m}\subset R_{m-1}\subset\dots\subset R_1\subset R_0=X\setminus U_1 \subset W_1,\]
welche die gewünschte Eigenschaft besitzen.
\end{proof}

\begin{definition}
Sei $X$ ein analytischer adischer Raum. Eine Nisnevich-Überdeckung \[\{i:U\rightarrow X,\ f:V\rightarrow X\}\] heißt \textit{elementar Nisnevich'sch}, wenn $i$ eine offene Einbettung und $f^{-1}(X\setminus U)\rightarrow X\setminus U$ ein Homöomorphismus vom Grad $1$ ist.
\end{definition}

\begin{satz}
Sei $X$ ein analytischer adischer Raum. Dann ist die Nisnevich-Topologie von elementaren Nisnevich-Überdeckungen erzeugt.
\end{satz}

\begin{proof}
Wörtlich derselbe Beweis wie in der Situation von \cite[Proposition 1.4]{MV}.
\end{proof}

\begin{definition}
Sei $X$ ein analytischer adischer Raum. Ein kartesisches Diagramm von analytischen adischen Räumen
\begin{center}
\begin{tikzcd}
Y \arrow[r] \arrow[d] & V \arrow[d] \\
U \arrow[r] & X
\end{tikzcd}\
\end{center}
heißt \textit{elementares Nisnevich-Quadrat}, wenn $\{U\rightarrow X,\ V\rightarrow X\}$ eine elementare Nisnevich-Überdeckung bildet.
\end{definition}

\begin{korollar}\label{excisiv}
Sei $X$ ein analytischer adischer Raum. Dann ist eine Prägarbe $\mathcal{F}$ von Animen bzw. Spektren auf dem Nisnevich-Situs $\mathrm{Nis}_X$ genau dann eine Garbe, wenn sie elementare Nisnevich-Quadrate auf kartesische Diagramme abbildet und $\mathcal{F}(\varnothing)=\ast$ gilt.
\end{korollar}

\begin{proof}
Dies folgt aus einem Satz von Voevodsky, siehe \cite[Theorem 3.2.5]{AHW}.
\end{proof}

\newpage
\section{Garben und Hypergarben}

In diesem Abschnitt des Anhangs untersuchen wir die Beziehung zwischen Garben und Hypergarben von Animen bzw. Spektren auf den \mbox{Zariski-,} Nisnevich- und étalen Siten eines analytischen adischen Raums. Es ist eines unserer Hauptziele, notwendige und hinreichende Bedingungen zu finden, die eine Hypergarbe auf dem Nisnevich-Situs erfüllen muss, um eine étale Hypergarbe zu sein. Dabei folgen wir dem Ansatz von Clausen und Mathew (siehe \cite{CM21}) im Fall von Schemata, welcher sich eins zu eins in die Situation von analytischen adischen Räumen übertragen lässt. Zunächst fassen wir allgemein gültige Definitionen und Techniken aus Abschnitten 2 und 3 von \cite{CM21} zusammen, welche wir danach in unserer Situation anwenden werden.

\begin{notation}
\begin{enumerate}[label=(\roman*)]
\item[]
\item Unter einem \textit{Situs} verstehen wir eine $1$-Kategorie zusammen mit einer Grothendieck-Topologie. Wir nehmen an, dass alle Siten ein finales Objekt enthalten.
\item Sei $\mathcal{T}$ ein Situs. Mit $\mathrm{Sh}(\mathcal{T})$ (bzw. $\mathrm{Sh}(\mathcal{T},\mathrm{Sp})$) bezeichnen wir die $\infty$-Kategorie von Garben von Animen (bzw. Spektren) auf $\mathcal{T}$. 
\item Sei $X$ ein Schema oder ein analytischer adischer Raum. Mit $X_{\mathrm{Zar}}$ (bzw. $X_{\mathrm{Nis}}$ bzw. $X_{\mathrm{\Acute{e}t}}$) bezeichnen wir den Zariski- (bzw. Nisnevich- bzw. étalen) $\infty$-Topos von $X$.
\item Sei $\mathcal{P}$ eine feste Teilmenge von Primzahlen. Ein Spektrum $X$ heißt \textit{$\mathcal{P}$-lokal}, falls die Multiplikation mit jeder Primzahl $q\not\in \mathcal{P}$ invertierbar auf $X$ ist. Die volle $\infty$-Unterkategorie von $\mathcal{P}$-lokalen Spektren wird mit $\mathrm{Sp}_{\mathcal{P}}$ bezeichnet. Im Falle $\mathcal{P}=\{$alle Primzahlen$\}$ ist die Bedingung leer.
\item Sei $k$ ein Körper. Für eine Primzahl $p$ bezeichnen wir mit $\mathrm{cd}_p(k)$ (bzw. $\mathrm{vcd}_p(k)$) die galoissche $p$-lokale kohomologische (bzw. virtuelle kohomologische) Dimension von $k$. Die \textit{galoissche $\mathcal{P}$-lokale kohomologische} (bzw. \textit{virtuelle kohomologische}) \textit{Dimension von $k$} ist definiert als $\underset{p\in\mathcal{P}}{\mathrm{sup}}\, \mathrm{cd}_p(k)$ (bzw. $\underset{p\in\mathcal{P}}{\mathrm{sup}}\, \mathrm{vcd}_p(k)$).
\end{enumerate}
\end{notation}

Wir erinnern zunächst an die Definitionen der Hyper- und Postnikow-Vollständigkeit, die dem gesamten Abschnitt zugrunde liegen.

\begin{definition}[{\cite[Definition 2.4]{CM21}}]\label{hyper}
Sei $\mathcal{T}$ ein Situs und $\mathcal{F}$ eine Garbe von Animen (bzw. Spektren) auf $\mathcal{T}$. Wir nennen $\mathcal{F}$
\begin{enumerate}[label=(\roman*)]
\item \textit{azyklisch}, falls $\pi_n(\mathcal{F})=0$ für alle $n\in\mathbb{N}$ (bzw. $n\in\mathbb{Z}$) ist;
\item \textit{hypervollständig} oder eine \textit{Hypergarbe}, falls $\mathrm{Hom}_{\mathrm{Sh}(\mathcal{T})}(\mathcal{G},\mathcal{F})=0$ bzw. $\mathrm{Hom}_{\mathrm{Sh}(\mathcal{T},\mathrm{Sp})}(\mathcal{G},\mathcal{F})=0$ für alle azyklischen Garben $\mathcal{G}$ auf $\mathcal{T}$ ist;
\item \textit{Postnikow-vollständig}, falls der Morphismus $\mathcal{F}\xrightarrow[]{} \underset{n}{\varprojlim}\, \tau_{\leq n} \mathcal{F}$ ein Isomorphismus ist.
\end{enumerate}
Die Kategorie $\mathrm{Sh}(\mathcal{T})$ heißt \textit{hypervollständig}, falls jede Garbe von Animen auf $\mathcal{T}$ hypervollständig ist.
\end{definition}

\begin{bemerkung}
Definition~\ref{hyper}(ii) ist äquivalent zur üblichen Definition mittels Hyperüberdeckungen. Außerdem ist eine Garbe von Spektren hypervollständig genau dann, wenn deren zugrunde liegende Garbe von Animen hypervollständig ist. Für Beweise der beiden Aussagen verweisen wir auf \cite[Beispiel 2.5]{CM21} und die dort zitierten Referenzen.
\end{bemerkung}

\begin{lemma}[{\cite[Lemma 2.6]{CM21}}]\label{grundeigenschaften_der_hypervollständigkeit}
\begin{enumerate}[label=(\roman*)]
\item[]
\item Die volle Unterkategorie von hypervollständigen Garben von Animen bzw. Spektren ist abgeschlossen unter Limites.
\item Eine nach oben beschränkte\footnote{D. h., die Homotopiegruppen $\pi_i$ verschwenden für alle $i$ oberhalb einer bestimmten Zahl.} Garbe von Animen bzw. Spektren ist hypervollständig.
\item Eine Postnikow-vollständige Garbe ist hypervollständig.
\end{enumerate}
\end{lemma}

Obwohl die Begriffe von Hyper- und Postnikow-Vollständigkeit im Allgemeinen nicht äquivalent sind (siehe bspw. \cite[Beispiel 1.30]{MV}), stimmen sie nach \cite{MoRei} für eine große Klasse von $\infty$-Topoi überein. In dieser Arbeit kommen wir jedoch ohne die Ergebnisse von \cite{MoRei} aus und verwenden ein klassisches, auf kohomologischer Dimension basierendes Kriterium.

\begin{definition}[{\cite[Definition 2.8]{CM21}}]
Sei $\mathcal{T}$ ein Situs und $\mathcal{F}$ eine Garbe von konnektiven $\mathcal{P}$-lokalen Spektren auf $\mathcal{T}$. 
\begin{enumerate}[label=(\roman*)]
\item Sei $\mathcal{A}$ eine Garbe von abelschen Gruppen auf $\mathcal{T}$. Die \textit{$i$-te Kohomologiegruppe von $\mathcal{F}$ mit Koeffizienten in $\mathcal{A}$} ist gegeben durch 
\[H^i(\mathcal{F};\mathcal{A})=\pi_0\mathrm{Hom}_{\mathrm{Sh(\mathcal{T},\mathrm{Sp})}}(\mathcal{F},\Sigma^i\mathcal{A}).\]
\item Die Garbe $\mathcal{F}$ heißt \textit{von kohomologischer Dimension $\leq d$}, falls $H^i(\mathcal{F};\mathcal{A})=0$ für jede Garbe von $\mathcal{P}$-lokalen abelschen Gruppen $\mathcal{A}$ und jedes $i\geq d$ ist.
\item Es heißt \textit{$\mathrm{Sh(\mathcal{T},\mathrm{Sp}_{\mathcal{P},\geq 0})}$ besitzt genügend Objekte $\mathcal{P}$-lokaler kohomologischer Dimension $\leq d$}, wenn es für jedes $\mathcal{G}\in \mathrm{Sh(\mathcal{T},\mathrm{Sp}_{\mathcal{P},\geq 0})}$ einen Morphismus $f:\mathcal{H}\rightarrow \mathcal{G}$ in $\mathrm{Sh(\mathcal{T},\mathrm{Sp}_{\mathcal{P},\geq 0})}$ gibt, sodass $\pi_0(f)$ ein Epimorphismus und $\mathcal{H}$ von $\mathcal{P}$-lokaler kohomologischer Dimension $\leq d$ ist.
\end{enumerate}
\end{definition}

\begin{satz}[{\cite[Proposition 2.10]{CM21}}]\label{äquivalenz_vollständig}
Sei $\mathcal{T}$ ein Situs und $\mathcal{F}\in \mathrm{Sh(\mathcal{T},\mathrm{Sp}_{\mathcal{P}})}$. Angenommen, $\mathrm{Sh(\mathcal{T},\mathrm{Sp}_{\mathcal{P},\geq 0})}$ besitzt genügend Objekte von $\mathcal{P}$-lokaler kohomologischer Dimension $\leq d$. Dann ist $\mathcal{F}$ hypervollständig genau dann, wenn es Postnikow-vollständig ist. 
\end{satz}

Diese Vollständigkeitsbegriffe spielen im Verlauf dieser Arbeit eine besondere Rolle. Die für unsere Zwecke wichtigsten Eigenschaften dieser Begriffe sind in den beiden folgenden Sätzen zusammengefasst.

\begin{satz}
Sei $\mathcal{T}$ ein Situs.
\begin{enumerate}[label=(\roman*)]
\item Ein Morphismus $\mathcal{F}\rightarrow \mathcal{G}$ von hypervollständigen Garben von Animen bzw. Spektren ist genau dann ein Isomorphismus, wenn der induzierte Morphismus $\pi_i(\mathcal{F})\rightarrow \pi_i(\mathcal{G})$ zwischen den Homotopiegruppen ein Isomorphismus für alle $i\in \mathbb{N}$ bzw. $i\in \mathbb{Z}$ ist.
\item Angenommen, der $1$-Topos der Garben von Mengen auf $\mathcal{T}$ besitzt genügend Punkte. Sei $x$ ein Punkt von $\mathrm{Sh}(\mathcal{T},\mathrm{Set})$. Dann gilt für eine Garbe $\mathcal{F}$ von Animen bzw. Spektren
\[(\pi_i(\mathcal{F}))_x\cong \pi_i(\mathcal{F}_x).\]
Insbesondere ist ein Morphismus von hypervollständigen Garben ein Isomorphismus genau dann, wenn er halmweise ein Isomorphismus ist.
\end{enumerate}
\end{satz}

\begin{proof}
Der erste Teil folgt direkt aus der Definition. Für den zweiten Teil bemerke, dass die Abschneidungsfunktoren $\tau_{\leq n}$ und $\tau_{\geq n}$ für jedes $n$ mit dem Rückzug vertauschen. 
\end{proof}

\begin{satz}[{\cite[Proposition 2.13]{CM21}}]
Sei $\mathcal{T}$ ein Situs. Gegeben sei eine Garbe von Spektren bzw. konnektiven Spektren $\mathcal{F}$ bzw. $\mathcal{G}$. Dann existiert eine bedingt konvergente Spektralsequenz 
\[E^{p,q}_2=H^p(\mathcal{G};\pi_q(\mathcal{F}))\implies \pi_{q-p}\mathrm{Hom}_{\mathrm{Sh}(\mathcal{T},\mathrm{Sp})}(\mathcal{G},\underset{n}{\varprojlim}\, \tau_{\leq n}\mathcal{F}).\]
\end{satz}

Die Hyper- und Postnikow-Vollständigkeit sind lokale Eigenschaften im folgenden Sinne:

\begin{lemma}[Lokal-Global-Prinzip für Hyper- und Postnikow-Vollständigkeit, {\cite[Proposition 2.25]{CM21}}]\label{lokal-global}
Sei $\mathcal{T}$ ein Situs und $\{X_i \}_{i\in I}$ eine Überdeckung des finalen Objekts von $\mathcal{T}$. Dann ist eine Garbe $\mathcal{F}$ von Animen oder Spektren hyper- bzw. Postnikow-vollständig genau dann, wenn ihre Einschränkung auf den Situs $\mathcal{T}_{/X_i}$ der Objekte über $X_i$ für jedes $i\in I$ hyper- bzw. Postnikow-vollständig ist.
\end{lemma}

Sei $\mathcal{T}$ ein Situs. Es bezeichne $\mathcal{C}$ (bzw. $\mathcal{C}^h$) die $\infty$-Kategorie der Garben von $\mathcal{P}$-lokalen Spektren auf $\mathcal{T}$ (bzw. die volle $\infty$-Unterkategorie der hypervollständigen Garben von $\mathcal{P}$-lokalen Spektren auf $\mathcal{T}$). Dann ist $\mathcal{C}^h$ eine linke Bousfield-Lokalisierung von $\mathcal{C}$, siehe \cite[Proposition 2.14]{CM21}. D. h., der natürliche Einbettungsfunktor $\mathcal{C}^h\rightarrow \mathcal{C}$ besitzt einen zugänglichen linksadjungierten Funktor $\mathcal{C}\rightarrow\mathcal{C}^h$, den wir den \textit{Hypervervollständigungsfunktor} nennen. Von besonderer Bedeutung ist der folgende Fall, in dem die volle $\infty$-Unterkategorie $\mathcal{C}^h$ außerdem unter Kolimites abgeschlossen ist.

\begin{lemma}[{\cite[Lemma 2.23]{CM21}}]\label{tensor-lokalisierung}
Die Situation sei wie oben. Es bezeichne $1^h$ die Hypervervollständigung der monoidalen Einheit von $\mathcal{C}$. Dann sind die folgenden Bedingungen äquivalent: 
\begin{enumerate}[label=(\roman*)]
\item Die volle $\infty$-Unterkategorie $\mathcal{C}^h\subset \mathcal{C}$ ist abgeschlossen unter Kolimites und dem Tensorproduktfunktor $-\otimes X$ für jedes $X\in \mathcal{C}$.
\item Für jedes $X\in \mathcal{C}$ ist das Objekt $1^h\otimes X$ hypervollständig.
\item Für jedes $X\in \mathcal{C}$ ist das Objekt $1^h\otimes X$ isomorph zur Hypervervollständigung von $X$.
\item Der Vergissfunktor $\mathrm{Mod}_{1^h}(\mathcal{C})\rightarrow \mathcal{C}$ ist volltreu mit wesentlichem Bild $\mathcal{C}^h$.
\item Jedes Objekt $X\in\mathcal{C}$, welches eine Struktur eines Moduls über einer Algebra $A\in\mathrm{Alg}(\mathcal{C}^h)$ besitzt, ist hypervollständig.
\end{enumerate}
\end{lemma}

\begin{definition}[{\cite[Definition 2.24]{CM21}}]
Sei $\mathcal{T}$ ein Situs. Wenn eine der äquivalenten Bedingungen aus Lemma~\ref{tensor-lokalisierung} erfüllt ist, nennen wir den Hypervervollständigungsfunktor eine \textit{Tensor-Lokalisierung}.
\end{definition}

Eines der für uns wichtigsten Ergebnisse von \cite{CM21} ist eine präzise Analyse von Hypervollständigkeit in Termen von Diagrammen. Wir erinnern zunächst an die folgenden Hilfsdefinitionen.  

\begin{definition}[{\cite[Definition 2.29]{CM21}}]
Sei $\mathcal{T}$ ein Situs, abgeschlossen unter endlichen Limites. Es heißt $\mathcal{T}$ \textit{finitistisch}, wenn jede Überdeckung in $\mathcal{T}$ eine endliche Teilüberdeckung besitzt.
\end{definition}

\begin{definition}[{\cite[Definition 2.30]{CM21}}]
Sei $\mathcal{T}$ ein finitistischer Situs. Es heißt \textit{$\mathcal{T}$ von $\mathcal{P}$-lokaler kohomologischer Dimension $\leq d$}, wenn die $\mathcal{P}$-Lokalisierung der darstellbaren Garbe $\Sigma^{\infty}_{+}h_x$\footnote{also die $\mathcal{P}$-Lokalisierung der Garbifizierung der naiven Prägarbe $y\mapsto \Sigma^{\infty}_{+}\mathrm{Hom}_{\mathcal{T}}(y,x)$} von $\mathcal{P}$-lokaler kohomologischer Dimension $\leq d$ für jedes $x\in\mathcal{T}$ ist.
\end{definition}

\begin{definition}[{\cite[Definition 2.33]{CM21}}]
Sei $m\geq 0$ eine natürliche Zahl.
\begin{enumerate}[label=(\roman*)]
\item Ein filtriertes Spektrum $\dots\rightarrow X_{-1}\rightarrow X_0\rightarrow X_1\rightarrow \dots$ heißt \textit{$m$-nilpotent} (bzw. \textit{schwach $m$-nilpotent}), wenn für jedes $i\in\mathbb{Z}$ die Abbildung $X_i\rightarrow X_{i+m+1}$ nullhomotop ist (bzw. die induzierte Abbildung auf den Homotopiegruppen trivial ist).
\item Wir nennen ein augmentiertes kosimpliziales Spektrum $X^{\bullet}\in\mathrm{Fun}(\Delta^+,\mathrm{Sp})$ \textit{$m$-rasch konvergent} (bzw. \textit{schwach $m$-rasch konvergent}), wenn der Turm \[\{\cofib (X^{-1}\rightarrow \mathrm{Tot_n(X^{\bullet})})\}_n\]
$m$-nilpotent (bzw. schwach $m$-nilpotent) ist.
\end{enumerate}
\end{definition}

Der folgende Satz erklärt die Beziehung zwischen Nilpotenz und Hypervollständigkeit und wird uns später bei der Untersuchung von étalen Hypergarben helfen.

\begin{satz}[{\cite[Proposition 2.35]{CM21}}]\label{hypervollständigkeit_überdeckungen}
Sei $\mathcal{T}$ ein finitistischer Situs von $\mathcal{P}$-lokaler kohomologischer Dimension $\leq d$. Für eine Garbe von $\mathcal{P}$-lokalen Spektren $\mathcal{F}$ auf $\mathcal{T}$ sind äquivalent:
\begin{enumerate}[label=(\roman*)]
\item $\mathcal{F}$ ist hypervollständig (oder äquivalent dazu Postnikow-vollständig).
\item Für jede abgeschnittene Hyperüberdeckung $y_{\bullet}$ in $\mathcal{T}$ von einem Objekt $x\in\mathcal{T}$ ist das augmentierte kosimpliziale Spektrum
\[\mathcal{F}(x)\rightarrow \mathcal{F}(y_\bullet)\]
$d$-rasch konvergent.
\end{enumerate}
\end{satz}

Als weitere Anwendung des Nilpotenzbegriffs geben Clausen und Mathew eine explizite Bedingung für die Vertauschung von filtrierten Kolimites und Totalisierungen:

\begin{lemma}[{\cite[Lemma 2.34]{CM21}}]\label{totalisierungen}
Sei $(X_i^{\bullet})_{i\in I}$ ein filtriertes System von augmentierten kosimplizialen Spektren, wobei $I$ eine partiell geordnete Menge ist. Angenommen, es gibt ein $m\geq 0$, sodass $X_i^{\bullet}$ schwach $m$-rasch konvergent für jedes $i\in I$ ist. Dann ist das augmentierte kosimpliziale Spektrum $\underset{i\in I}{\colim}X_i^{\bullet}$ schwach $m$-rasch konvergent und der Morphismus
\[\underset{i\in I}{\colim}\mathrm{Tot}(X_i^{\bullet})\rightarrow \mathrm{Tot}(\underset{i\in I}{\colim}X_i^{\bullet})\]
ein Isomorphismus.
\end{lemma}

Unser erstes Ziel ist es zu zeigen, dass der Nisnevich Situs eines endlichdimensionalen analytischen adischen Raums hypervollständig ist. Tatsächlich zeigen wir mehr: Der Nisnevich-Situs eines solchen Raums besitzt endliche \textit{Homotopiedimension}.

\begin{definition}[{\cite[Proposition 6.5.1.12, Definition 7.2.1.1]{HTT}}]
Sei $\mathcal{T}$ ein Situs und $n\geq -1$.
\begin{enumerate}[label=(\roman*)]
\item Eine Garbe von Animen $\mathcal{F}\in\mathrm{Sh}(\mathcal{T})$ heißt \textit{$n$-konnektiv}, wenn die Abschneidung $\tau_{n-1}\mathcal{F}$ ein finales Objekt von $\mathrm{Sh}(\mathcal{T})$ ist.
\item Es heißt $\mathrm{Sh}(\mathcal{T})$ \textit{von Homotopiedimension $\leq n$}, wenn jede $n$-konnektive Garbe $\mathcal{F}\in \mathrm{Sh}(\mathcal{T})$ einen Morphismus $\ast\rightarrow \mathcal{F}$ aus dem finalen Objekt besitzt.
\end{enumerate}
\end{definition}

Wir erinnern an das folgende wichtige Kriterium für Hypervollständigkeit:

\begin{satz}[{\cite[Korollare 7.2.1.12 und 7.2.2.30]{HTT}}]
Sei $\mathcal{T}$ ein Situs. Ist die $\infty$-Kategorie $\mathrm{Sh}(\mathcal{T})$ von Homotopiedimension $\leq n$, so ist sie hypervollständig und von kohomologischer Dimension $\leq n$. Insbesondere ist jede Garbe auf $\mathcal{T}$ Postnikow-vollständig.
\end{satz}

In unseren Beweisen nutzen wir die Hypervollständigkeit der gegebenen Garben und reduzieren die gewünschten Sätze auf bestimmte Aussagen auf dem Niveau von Halmen. Im Allgemeinen können wir eine Prägarbe auf der Kategorie von étalen Huber-Paaren über einem analytischen Huber-Paar $(A,A^+)$ durch Anwendung der Kan-Erweiterung auf die Kategorie von \textit{ind-étalen} Huber-Paaren über $(A,A^+)$ fortsetzen. 

\begin{definition}[{\cite[Definition 8.2.1]{Berkeley}}]
Ein Morphismus $f:(A,A^+)\rightarrow (B,B^+)$ von analytischen Huber-Paaren heißt \textit{ind-étale}, wenn $(B,B^+)$ zur Vervollständigung des Kolimes $(\colim\, (A_i,A_i^+))^{\wedge}$ eines filtrierten Systems von étalen Huber-Paaren über $(A,A^+)$ isomorph ist. 
\end{definition}

\begin{konstruktion}[vgl. {\cite[Konstruktion 4.30]{CM21}}]\label{kan_erweiterung}
Sei $X=\Spa (A,A^+)$ ein affinoider analytischer adischer Raum und $\mathcal{F}$ eine Prägarbe auf der Kategorie von étalen Räumen über $X$. Für ein ind-étales Huber-Paar $(B,B^+)\cong (\colim\, (A_i,A_i^+))^{\wedge}$ setzen wir
\[\mathcal{F}(B,B^+)=\colim\, \mathcal{F}(A_i,A_i^+).\]
\end{konstruktion}

\begin{bemerkung}
Streng genommen sind die Werte der oben gegebenen Fortsetzung nur bis auf Isomorphie definiert. In der „korrekten“ Definition wird das System $(A_i,A_i^+)_{i\in I}$ durch das System aller étalen Huber-Paare über $(A,A^+)$ mit einer Abbildung nach $(B,B^+)$ ersetzt.
\end{bemerkung}

\begin{lemma}\label{Nisnevich-Lokalisierung}
Sei $(B,B^+)=(\colim\, (A_i,A_i^+))^{\wedge}$ ein ind-étales Huber-Paar über einem tateschen Huber-Paar $(A,A^+)$. Dann definiert Konstruktion~\ref{kan_erweiterung} für jede Garbe von Animen (bzw. Spektren) auf dem Nisnevich-Situs von $\Spa (A,A^+)$ eine Nisnevich-Garbe von Animen (bzw. Spektren) auf $\Spa (B,B^+)$. Mit anderen Worten wird es keine Garbifizierung in der Konstruktion des Rückzugs entlang von $\Spa (B,B^+)\rightarrow \Spa (A,A^+)$ benötigt.
\end{lemma}

\begin{proof}
Nach Korollar~\ref{excisiv} genügt es nachzuprüfen, dass die Fortsetzung elementare Nisnevich'sche Quadrate auf kartesische Diagramme schickt. Aus \cite[Proposition 5.4.53]{Almost} und dem Beweis von \cite[Proposition III.6.3.7]{Morel} folgt, dass die Kategorie von étalen Huber-Paaren über $(B,B^+)$ und die Kategorie von étalen Huber-Paaren über $\colim\, (A_i,A_i^+)$ isomorph sind. Wie man leicht nachprüft, wird jedes elementare Nisnevich'sche Quadrat für $\colim\, (A_i,A_i^+)$ von einem solchen Quadrat für ein $(A_i,A^+_i)$ induziert.
\end{proof}

\begin{bemerkung}
Wie im schematischen Falle, gilt das étale Analogon des Lemmas~\ref{Nisnevich-Lokalisierung} nicht.
\end{bemerkung}

Als letzte Vorbereitung, bevor wir uns dem ersten Beweis zuwenden, erinnern wir an folgendes wichtige Resultat von \cite{CM21}:

\begin{satz}[{\cite[Theoreme 3.14 und 3.30]{CM21}}]\label{spektrale_räume_dimension}
Sei $X$ ein quasi-kompaktes quasi-separiertes Schema endlicher Krull-Dimension. Für einen Punkt $x\in X$ bezeichne man mit $i_x$ bzw. $\iota_x$ die kanonische Abbildung $\Spec \mathcal{O}_{X,x}\rightarrow X$ bzw. $\Spec \mathcal{O}^h_{X,x}\rightarrow X$. Hierbei bezeichnet $\mathcal{O}^h_{X,x}$ die Henselisierung von $\mathcal{O}_{X,x}$. Sei $n\geq 0$ eine natürliche Zahl.

\begin{enumerate}[label=(\roman*)]
\item Sei $\mathcal{F}$ eine Garbe von Animen auf dem Zariski-Situs von $X$. Ist der Zariski-Halm 
\[\Gamma (\Spec \mathcal{O}_{X,x},(i_x)_{\mathrm{Zar}}^\ast \mathcal{F})\]
$(\mathrm{dim}\overline{\{x\}}+n)$-konnektiv für jeden Punkt $x\in X$, so ist $\mathcal{F}(X)$ $n$-konnektiv. 
\item Sei $\mathcal{F}$ eine Garbe von Animen auf dem Nisnevich-Situs von $X$. Ist der Nisnevich-Halm 
\[\Gamma (\Spec \mathcal{O}^h_{X,x},(\iota_x)_{Nis}^\ast \mathcal{F})\]
$(\mathrm{dim}\overline{\{x\}}+n)$-konnektiv für jeden Punkt $x\in X$, so ist $\mathcal{F}(X)$ $n$-konnektiv. 
\end{enumerate}
Insbesondere ist die Homotopiedimension des Zariski- bzw. Nisnevich-Situs kleiner oder gleich $\mathrm{dim}\,X$.
\end{satz}

Und hier ist nun der erste versprochene Beweis. Da jeder spektrale topologische Raum zum Spektrum eines kommutativen Ringes isomorph ist, gilt der erste Teil des Satzes auch für quasi-kompakte quasi-separierte analytische adische Räume. Tatsächlich gilt in dieser Situation auch der zweite Teil, wie der folgende Satz zeigt:

\begin{satz}
Sei $X$ ein quasi-kompakter quasi-separierter analytischer adischer Raum endlicher Krull-Dimension und $\mathcal{F}$ eine Garbe von Animen auf dem Nisnevich-Situs von $X$. Mit $\iota_x$ bezeichnen wir die kanonische Abbildung $(\kappa_h(x),\kappa_h^+(x))\rightarrow X$ für die Henselisierung des Restklassenkörpers von $X$ in $x$. Sei $n\geq 0$ eine natürliche Zahl. Ist der Nisnevich-Halm \[\Gamma (\Spa (\kappa_h(x),\kappa_h^+(x)),(\iota_x)_{Nis}^\ast \mathcal{F})\] $(\mathrm{dim}\overline{\{x\}}+n)$-konnektiv für jeden Punkt $x\in X$, so ist $\mathcal{F}(X)$ $n$-konnektiv.\footnote{Die Bedingung wird nur an die Henselisierungen der Punkte von $X$ gestellt: In diesem Fall ist sie für die anderen henselschen Punkte automatisch erfüllt.} Insbesondere ist die Homotopiedimension kleiner oder gleich $\mathrm{dim}\,X$. 
\end{satz}

\begin{proof}
Wir beweisen den Satz mit Induktion nach $\mathrm{dim}\,X$. Der Fall $\mathrm{dim}\,X=-1$ ist trivial. Gelte nun die Behauptung für $\mathrm{dim}\,X-1$. Nach Satz~\ref{spektrale_räume_dimension} reicht es aus zu zeigen, dass der Zariski-Halm von $\mathcal{F}$ in $x$ $(\mathrm{dim}\overline{\{x\}}+n)$-konnektiv für jeden Punkt $x\in X$ ist.
Sei $j:U\hookrightarrow X$ eine quasi-kompakte offene Umgebung von $x$. Man betrachte folgendes kommutative Diagramm:
\begin{center}
\begin{tikzcd}
\Spa (\kappa_h(x),\kappa_h^+(x))_{\mathrm{Nis}}\arrow[r]\arrow[d] &\Spa (\kappa(x),\kappa^+(x))_{\mathrm{Nis}}\arrow[d]\arrow[r,"i_{\mathrm{Nis}}"] & U_{\mathrm{Nis}} \arrow[d,"\beta"] \arrow[r,"j_{\mathrm{Nis}}"] & X_{\mathrm{Nis}} \arrow[d,"\alpha"] \\
\Spa (\kappa_h(x),\kappa_h^+(x))_{\mathrm{Zar}}\arrow[r] &\Spa (\kappa(x),\kappa^+(x))_{\mathrm{Zar}}\arrow[r,"i_{\mathrm{Zar}}"] & U_{\mathrm{Zar}} \arrow[r,"j_{\mathrm{Zar}}"] & X_{\mathrm{Zar}}
\end{tikzcd}
\end{center}
Hierbei stellen die vertikalen Pfeile die kanonischen Morphismen von entsprechenden $\infty$-Topoi dar, die eine Nisnevich-Garbe nach der zugrunde liegenden Zariski-Garbe abbilden. Mit der Notation dieses Diagramms ist die Konnektivität, die wir nachweisen wollen, offenbar dazu äquivalent, dass das Anima \[\Gamma (\Spa (\kappa(x),\kappa^+(x)),(j \circ i)_{\mathrm{Zar}}^\ast \alpha_\ast (\mathcal{F}))\] $(\mathrm{dim}\overline{\{x\}}+k)$-konnektiv ist. Man sieht unmittelbar, dass die natürliche Transformation $j_{\mathrm{Zar}}^\ast\alpha_\ast\rightarrow \beta_\ast j_{\mathrm{Nis}}^\ast$ ein Isomorphismus ist. Da $\Spa (\kappa(x),\kappa^+(x))\cong \underset{x\in U}{\varprojlim}\, U$ gilt, wobei $U$ die offenen Umgebungen von $x$ durchläuft, genügt es wegen der vorigen Behauptung und Lemma~\ref{Nisnevich-Lokalisierung} folgendes zu zeigen: Sei $(K,K^+)$ ein affinoider Körper endlicher Krull-Dimension und $\mathcal{F}$ eine Garbe auf dessen Nisnevich-Situs. Sind die Bedingungen des Satzes erfüllt, so ist das Anima $\mathcal{F}(\Spa (K,K^+))$ $n$-konnektiv. Mit anderen Worten reicht es aus, den Satz im Falle $X=\Spa (K,K^+)$ nachzuweisen.

Wir betrachten nun den vorstehenden Fall. Wir bezeichnen mit $x$ den abgeschlossenen Punkt von $\Spa (K,K^+)$. Nach Annahme gibt es eine Nisnevich-Umgebung von $x$\footnote{D. h., der Grad der Abbildung im abgeschlossen Punkt ist gleich $1$.} der Form $\Spa (\Tilde{K},\Tilde{K}^+)$ für einen affinoiden Körper $(\Tilde{K},\Tilde{K}^+)$, sodass das Anima $\mathcal{F} (\Spa (\Tilde{K},\Tilde{K}^+))$ nicht leer ist. Zudem sind die Animen $\mathcal{F} (\Spa (K,K^+)\setminus \{x\})$ und $\mathcal{F}(\Spa (\Tilde{K},\Tilde{K}^+)\underset{\Spa (K,K^+)}{\times}(\Spa (K,K^+)\setminus \{x\}))$ nach Induktionsannahme $n+1$-konnektiv. Da das Paar
\[\{\Spa (K,K^+)\setminus \{x\}, \Spa (\Tilde{K},\Tilde{K}^+)\}\]
eine elementare Nisnevich-Überdeckung bildet, ist das Diagramm
\begin{center}
\begin{tikzcd}
\mathcal{F}(\Spa (K,K^+))\arrow[r]\arrow[d] &\mathcal{F}(\Spa (K,K^+)\setminus \{x\})\arrow[d]\\
\mathcal{F}(\Spa (\Tilde{K},\Tilde{K}^+))\arrow[r] &\mathcal{F}(\Spa (\Tilde{K},\Tilde{K}^+)\underset{\Spa (K,K^+)}{\times}(\Spa (K,K^+)\setminus \{x\}))
\end{tikzcd}
\end{center}
kartesisch, weshalb $\mathcal{F}(\Spa (K,K^+))$ nicht leer ist. Die höhere Konnektivität wird völlig analog nachgeprüft.
\end{proof}

Wir beginnen nun mit der Analyse der Hypervollständigkeit auf dem étalen Situs eines analytischen adischen Raums $X$. Unser erstes Ziel ist ein Hypervollständigkeitskriterium in Termen von Punkten von $X$. Wir erinnern zunächst an die folgende Definition.

\begin{definition}[{\cite[Definition 4.1]{CM21}}]\label{stetige_g_mengen}
Sei $G$ eine proendliche Gruppe. Die zugrunde liegende Kategorie des \textit{Situs $\mathcal{T}_G$ von stetigen $G$-Mengen} ist definiert als die Kategorie von endlichen Mengen mit stetiger $G$-Wirkung. Eine Familie $\{M_i\rightarrow M\}_{i\in I}$ von Morphismen bildet eine Überdeckung, wenn die Abbildung $\coprod M_i\rightarrow M$ surjektiv ist. 
\end{definition}

Für einen quasi-kompakten quasi-separierten analytischen adischen Raum $X$ und einen Punkt $x\in X$ nennen wir den Situs von stetigen $\mathrm{Gal}(\kappa_h(x))$-Mengen den \textit{galoisschen Situs von $X$ in $x$}. Wir bezeichnen ihn mit $\mathcal{T}_x$. In der Terminologie von pseudo-adischen Räumen (siehe \cite[Abschnitt 1.10]{Huber}) ist der entsprechende ($\infty$-)Topos nach \cite[Proposition 2.3.10]{Huber} zum étalen ($\infty$-)Topos des pseudo-adischen Raums $(\Spa(\kappa(x),\kappa^+(x)),\Tilde{x})$ äquivalent, wobei $\Tilde{x}$ der abgeschlossene Punkt bezeichnet.

Das Kriterium, das wir im Folgenden beweisen werden, gilt für analytische adische Räume, für die die étalen kohomologischen Dimensionen der Räume in den entsprechenden étalen Siten nach oben beschränkt sind. Daher untersuchen wir zunächst die Beziehung zwischen diesen „globalen“ Dimensionen und den „lokalen“ Dimensionen der galoisschen Siten.

\begin{satz}[vgl. {\cite[Korollar 3.29]{CM21}}]
Sei $X$ ein quasi-kompakter quasi-separierter analytischer adischer Raum der Dimension $d$. Es bezeichne $\mathrm{cd_x}$ die $\mathcal{P}$-lokale kohomologische Dimension des galoisschen Situs von $X$ in $x$. Dann gilt:

\[\underset{x\in X}{\mathrm{sup}}\, \mathrm{cd}_x\leq \underset{U\in \Acute{E}t_X}{\mathrm{sup}}\, \mathrm{CohDim}_\mathcal{P}(U_{\Acute{e}t})\leq d+\underset{x\in X}{\mathrm{sup}}\, \mathrm{cd}_x.\]
\end{satz}

\begin{proof}
Wir beweisen zunächst die erste Ungleichung. Sei $x$ ein Punkt von $X$. Wie man leicht nachprüft, ist der direkte Bildfunktor $r_\ast:\mathcal{T}_x\rightarrow \Spa (\kappa_h(x),\kappa_h^+(x))$ exakt, also gilt \[\mathrm{cd}_x\leq \mathrm{CohDim}_\mathcal{P}(\Spa (\kappa_h(x),\kappa_h^+(x))_{\Acute{e}t}).\] 
Da der adische Raum $\Spa (\kappa_h(x),\kappa_h^+(x))$ isomorph zum inversen Limes der Nisnevich-Umgebungen von $x$ ist, folgt die Aussage also aus dem adischen Analogon von \cite[\href{https://stacks.math.columbia.edu/tag/03Q4}{Tag 03Q4}]{stacks-project}.

Die zweite Ungleichung lässt sich leicht aus \cite[Proposition 2.8.1]{Huber} ableiten.
\end{proof}

Sei $X=\Spa (A,A^+)$ ein affinoider analytischer adischer Raum. Unter dem \textit{affinoiden étalen Situs von $X$} verstehen wir die Kategorie affinoider étaler Räume über $X$. Eine Überdeckung in diesem Situs ist per Definition eine Familie von Abbildungen zwischen affinoiden étalen Räumen über $X$, die eine Überdeckung im üblichen étalen Situs von $X$ bildet. Das folgende Lemma, das in der gleichen Weise wie in \cite{CM21} nachgewiesen wird, wird uns im Weiteren erlauben, den Beweis des gewünschten Kriteriums auf das Analogon für den affinoiden Situs zu reduzieren.

\begin{lemma}[vgl. {\cite[Proposition 4.34]{CM21}}]\label{garben_affinoid}
Sei $X$ ein quasi-kompakter quasi-separierter analytischer adischer Raum und $\mathcal{F}$ eine Nisnevich-Garbe von Animen bzw. Spektren auf $X$. Gilt eine der folgenden Aussagen für jede étale Abbildung $f:U\rightarrow X$, mit $U$ affinoid tatesch, so gilt deren Analogon für $\mathcal{F}$ auf dem üblichen étalen Situs:

\begin{enumerate}[label=(\roman*)]
\item der Rückzug $f^\ast \mathcal{F}$ ist eine Garbe auf dem affinoiden étalen Situs von $U$;
\item der Rückzug $f^\ast \mathcal{F}$ ist eine hypervollständige Garbe auf dem affinoiden étalen Situs von $U$;
\item der Rückzug $f^\ast \mathcal{F}$ ist eine Postnikow-vollständige Garbe auf dem affinoiden étalen Situs von $U$;
\item der Rückzug $f^\ast \mathcal{F}$ ist isomorph zu der trivialen Garbe auf dem affinoiden étalen Situs von $U$;
\end{enumerate}
\end{lemma}

Als letzter Schritt vor der Formulierung des Kriteriums erinnern wir an die Analyse von \cite{CM21} der Hypervollständigkeit auf dem Situs der stetigen Mengen einer pro-endlichen Gruppe.

\begin{definition}[{\cite[Definition 4.8]{CM21}}]
Sei $K$ eine endliche Gruppe und $X$ ein Spektrum mit $K$-Wirkung. Die $K$-Wirkung heißt \textit{schwach $m$-nilpotent}, falls das augmentierte kosimpliziale Spektrum
\[X^{hK}\rightarrow X\rightrightarrows \underset{K}{\prod} X\mathrel{\substack{\textstyle\rightarrow\\[-0.6ex]
                      \textstyle\rightarrow \\[-0.6ex]
                      \textstyle\rightarrow}} \dots \]
welches die Homotopiefixpunkte berechnet, schwach $m$-rasch konvergent ist.
\end{definition}

\begin{satz}[{\cite[Proposition 4.16]{CM21}}]\label{hypervollständigkeit_galois}
Sei $G$ eine pro-endliche Gruppe $\mathcal{P}$-lokaler kohomologischer Dimension $d$. Für eine Garbe von $\mathcal{P}$-lokalen Spektren $\mathcal{F}$ auf $\mathcal{T}_G$ sind äquivalent:
\begin{enumerate}[label=(\roman*)]
\item Die Garbe $\mathcal{F}$ ist hypervollständig.\footnote{Bzw. Postnikow-vollständig, was in diesem Falle nach Satz~\ref{äquivalenz_vollständig} äquivalent zur Hypervollständigkeit von $\mathcal{F}$ ist.}
\item Es gibt ein $m\geq 0$, sodass für jeden offenen Normalteiler $N\subset H$ einer offenen Untergruppe von $G$ die $H/N$-Wirkung auf dem Spektrum $\mathcal{F}(G/N)$ schwach $m$-nilpotent ist.
\item Für jeden offenen Normalteiler $N\subset H$ einer offenen Untergruppe von $G$ ist die $H/N$-Wirkung auf dem Spektrum $\mathcal{F}(G/N)$ schwach $d$-nilpotent ist.
\end{enumerate}
\end{satz}

\begin{satz}[{\cite[Proposition 4.17]{CM21}}]\label{tensor-lokalisierung-proendliche-gruppen}
Sei $G$ eine pro-endliche Gruppe $\mathcal{P}$-lokaler virtueller kohomologischer Dimension $d$. Dann ist der Hypervervollständigungsfunktor für Garben von $\mathcal{P}$-lokalen Spektren auf $\mathcal{T}_G$ eine Tensor-Lokalisierung.
\end{satz}

Schließlich sind wir bereit, das versprochene Kriterium zu formulieren und nachzuweisen:

\begin{satz}[vgl. {\cite[Theorem 4.36]{CM21}}]\label{kriterium}
Sei $X$ ein quasi-kompakter quasi-separierter analytischer adischer Raum, sodass die $\mathcal{P}$-lokalen étalen kohomologischen Dimensionen von $\{U\rightarrow X\}\in \mathrm{\acute{E}t}_X$ nach oben beschränkt sind. Ist $\mathcal{F}$ eine hypervollständige Nisnevich-Garbe von $\mathcal{P}$-lokalen Spektren, so ist sie eine hypervollständige étale Garbe genau dann, wenn für alle Punkte $x\in X$ der Nisnevich-Rückzug entlang von $\Spa (\kappa_h(x),\kappa_h^+(x))\rightarrow X$ eine hypervollständige Garbe auf $\mathcal{T}_x$ definiert.
\end{satz}

\begin{proof}
Unser Beweis ist identisch zu dem Beweis von \cite[Theorem 4.36]{CM21}. Sei $d$ eine obere Schranke für die $\mathcal{P}$-lokalen étalen kohomologischen Dimensionen von $\{U\rightarrow X\}\in \mathrm{\acute{E}t}_X$. Angenommen, $\mathcal{F}$ ist eine hypervollständige étale Garbe. Seien $\kappa_h(x)\rightarrow k\rightarrow k'$ endliche separable Körpererweiterungen und $k^+$ (bzw. $k'^+$) der ganze Abschluss von $\kappa_h(x)^+$ in $k$ (bzw. in $k'$). Wir schreiben $(\kappa_h(x),\kappa_h^+(x))$ als $(\underset{U}{\colim} (\mathcal{O}_X(U),\mathcal{O}_X^+(U)))^\wedge$, wobei $U$ die Nisnevich-Umgebungen von $x$ durchläuft. Wie man mithilfe von \cite[Proposition 5.4.53]{Almost} leicht nachprüft, kann die Abbildung $(k,k^+)\rightarrow (k',k'^+)$ als filtrierten Kolimes 
\[\biggl(\underset{U}{\colim}\, (B_U,B_U^+)\rightarrow (B_U',B_U'^+)\biggr)^\wedge\]
von volltreuen\footnote{D. h., die induzierte Abbildung zwischen adischen Spektren ist surjektiv.} étalen Abbildungen von étalen Huber-Paaren über $(\mathcal{O}_X(U),\mathcal{O}_X^+(U))$ geschrieben werden. Nach Satz~\ref{hypervollständigkeit_überdeckungen} ist der Čech'sche Nerv 
\[\mathcal{F}(B_U,B_U^+)\rightarrow \mathcal{F}(B_U',B_U'^+)\rightrightarrows\cdots\]
$d$-rasch konvergent, weshalb das kosimpliziale Diagramm 
\[\mathcal{F}(k,k^+)\rightarrow \mathcal{F}(k',k'^+)\rightrightarrows\cdots\]
wegen Lemma~\ref{totalisierungen} schwach $d$-rasch konvergent ist. Der Nisnevich-Rückzug von $\mathcal{F}$ definiert also eine étale Garbe auf $\mathcal{T}_x$, die zudem nach Satz~\ref{hypervollständigkeit_galois} hypervollständig ist.

Sei nun umgekehrt $\mathcal{F}$ eine Nisnevich-Garbe, sodass ihr Nisnevich-Rückzug entlang der Abbildung $\Spa (\kappa_h(x),\kappa_h^+(x))\rightarrow X$ für jeden Punkt $x\in X$ eine hypervollständige Garbe auf $\mathcal{T}_x$ definiert. Sei $\mathcal{F}\rightarrow \Tilde{\mathcal{F}}$ die étale Hypervervollständigung von $\mathcal{F}$. Wir betrachten das Diagramm
\begin{center}
\begin{tikzcd}
\mathrm{Sh}(\mathcal{T}_x,\mathrm{Sp}) \arrow[r]\arrow[d] &\Spa (\kappa_h(x),\kappa_h^+(x))_{\mathrm{\Acute{e}t}}\arrow[d,"\beta"]\arrow[r,"\iota_{\mathrm{\Acute{e}t}}"] & X_{\mathrm{\mathrm{\Acute{e}t}}} \arrow[d,"\alpha"] \\
\mathrm{PSh}_{\prod}(\mathcal{T}_x,\mathrm{Sp}) \arrow[r] &\Spa (\kappa_h(x),\kappa_h^+(x))_{\mathrm{Nis}}\arrow[r,"\iota_{\mathrm{Nis}}"] & X_{\mathrm{Nis}}
\end{tikzcd}
\end{center}
wobei $\mathrm{PSh}_{\prod}(\mathcal{T}_x,\mathrm{Sp})$ die volle $\infty$-Unterkategorie derjenigen Prägarben auf $\mathcal{T}_x$ bezeichnet, welche disjunkte Summen von $\mathrm{Gal}(\kappa_h(x))$-Mengen auf Produkte abbilden. Nach Satz~\ref{punkte} und Lemma~\ref{Nisnevich-Lokalisierung} genügt es zu zeigen, dass die Abbildung \[\iota^\ast_{\mathrm{Nis}} \mathcal{F}\rightarrow \iota^\ast_{\mathrm{Nis}}\alpha_\ast \Tilde{\mathcal{F}}\] ein Isomorphismus ist. Die Nisnevich-Garbe $\iota^\ast_{\mathrm{Nis}}\mathcal{F}$ (bzw. $\iota^\ast_{\mathrm{Nis}}\alpha_\ast \Tilde{\mathcal{F}}$) definiert nach Annahme (bzw. wegen der ersten Hälfte des Beweises) eine hypervollständige Garbe auf $\mathcal{T}_x$. Da die Hypervervollständigung Halme nicht ändert, definieren sie dieselbe Garbe auf $\mathcal{T}_x$. Mit anderen Worten ist die Abbildung  \[\iota^\ast_{\mathrm{Nis}} \mathcal{F} (\Spa (k,k^+))\rightarrow \iota^\ast_{\mathrm{Nis}}\alpha_\ast \Tilde{\mathcal{F}} (\Spa (k,k^+))\] ein Isomorphismus für jeden affinoiden Körper $\Spa (k,k^+)$, wobei $k$ eine endliche separable Erweiterung von $\kappa_h(x)$ und $k^+$ der ganze Abschluss von $\kappa_h(x)^+$ in $k$ ist. Durch Betrachtung von offenen Teilmengen von $\Spa (\kappa_h(x),\kappa_h^+(x))$\footnote{Sie entsprechen den Generalisierungen von $x$ in $X$.} erhalten wir die gewünschte Aussage für offene Teilmengen solcher affinoiden Körper.
\end{proof}

Daraus lässt sich genau wie im schematischen Falle das folgende Korollar ableiten.

\begin{korollar}[vgl. {\cite[Korollar 4.40]{CM21}}]\label{adische_räume_tensor_lokalisierung}
Sei $X$ ein quasi-kompakter quasi-separierter analytischer adischer Raum endlicher Krull-Dimension, sodass die virtuellen $\mathcal{P}$-lokalen kohomologischen Dimensionen der galoisschen Siten von $X$ nach oben beschränkt sind. Dann ist die Hypervervollständigung von $\mathcal{P}$-lokalen Spektren eine Tensor-Lokalisierung.
\end{korollar}
\newpage
\printbibliography

\end{document}